\documentclass{article}

\usepackage{arxiv}

\usepackage[utf8]{inputenc} 
\usepackage[T1]{fontenc}    
\usepackage{hyperref}       
\usepackage{url}            
\usepackage{booktabs}       
\usepackage{amsfonts}       
\usepackage{amsmath}
\usepackage{amsthm}
\usepackage{bm}             
\usepackage{nicefrac}       
\usepackage{float}          
\usepackage{microtype}      
\usepackage{lipsum}
\usepackage{graphicx}
\usepackage{subcaption}    
\usepackage[ruled]{algorithm2e} 
\graphicspath{ {./images/} }

\newtheorem{theo}{Theorem}[section]
\newtheorem{proposition}{Proposition}[section]
\newtheorem{lemma}{Lemma}[section]

\newtheorem{remark}{Remark}[section]
\newtheorem{corollary}{Corollary}[section]

\newcommand{\projector}{\ensuremath{\Pi^\nabla_{E}}}
\newcommand{\ustar}{\ensuremath{\mathbf{u}^*}}

\newcommand{\abilinear}[4]{\ensuremath{a_{h,#1}^{#2}\left(#3,#4;\ustar\right)}}
\newcommand{\missingkernel}{\ensuremath{\ker(\projector)}}

\newcommand{\increment}[1]{\ensuremath{\delta \mathbf{#1}}}
\makeatletter
\newcommand{\basis}[1][\@nil]{\ensuremath{\bm{\varphi}\ifx#1\@nil\else_{#1}\fi}}
\makeatother
\makeatletter
\newcommand{\ahE}{\@ifnextchar[{\ahE@opt}{\ensuremath{a_{h,E}}}}
\def\ahE@opt[#1]{%
  \@ifnextchar[{\ahE@opttwo[#1]}{\ahE@optone[#1]}%
}
\def\ahE@optone[#1]{\ensuremath{a_{h,E}^{#1}}}
\def\ahE@opttwo[#1][#2]{\ensuremath{a_{h,E}^{#1,\mathrm{#2}}}}
\makeatother

\newcommand{\fracsobolev}{\ensuremath{H^{1/2}(\partial E)}}

\title{An Investigation of Stabilization Scaling in Finite-Strain Virtual Element Methods for Hyperelasticity}

\author{
 Paulo Akira F. Enabe \\
    Escola Politénica\\
    University of São Paulo\\
    Department of Structural and Geotechnical Engineering\\
  \texttt{paulo.enabe@usp.br} \\
   \And
 Rodrigo Provasi \\
    Escola Politénica\\
    University of São Paulo\\
    Department of Structural and Geotechnical Engineering\\
  \texttt{provasi@usp.br} \\
}

\begin{document}
\maketitle
\begin{abstract}
    Low-order virtual element methods (VEM) compute a consistent finite-strain contribution through polynomial projections and rely on stabilization to control the unresolved modes in the projector kernel. In current hyperelastic VEM practice, stabilization is often defined by integrating a nonlinear surrogate energy over an auxiliary sub-triangulation and scaled through modified Lam\'e parameters and incompressibility factors; this can introduce sensitivity to the arbitrary internal tessellation, complicate consistent Newton linearization, and, most critically, inject bulk-dependent proxies into shear-type kernel penalties, artificially stiffening isochoric missing modes in the nearly incompressible regime. This work develops a submesh-free, kernel-only stabilization that decouples deviatoric and volumetric channels and is explicitly designed to scale like the current Newton tangent energy on the kernel: the deviatoric term is scaled solely by a shear measure and enhanced by bounded geometry-driven directional weights, while the volumetric term is scaled by an independent bulk measure and can be capped or suppressed as $\nu\to 1/2$. A spectral framework is established in which the canonical VEM stability requirement on the kernel is characterized by generalized Rayleigh quotients and eigenvalue bounds, and it is shown under standard polygon regularity assumptions that the deviatoric stabilization is uniformly equivalent to $\mu_E|\cdot|_{1,E}^2$ on the kernel with constants independent of mesh size and Poisson ratio. Element-level diagnostics confirm that classical surrogate-based stabilizations assign bulk-driven energy to isochoric kernel modes as $\nu\to 1/2$, whereas the proposed decoupled stabilization remains shear-scaled; kernel spectra and Cook's membrane simulations in the nearly incompressible regime further support improved robustness across polygonal mesh families.
\end{abstract}


\section{Introduction}

\paragraph{}
The virtual element method (VEM) is a conforming Galerkin framework for partial differential equations on general polygonal and polyhedral meshes, in which local approximation spaces are not represented by explicit interior basis functions. Instead, each virtual space is characterized by degrees of freedom and computable polynomial projection operators, which enable the exact evaluation of the consistent contribution on a prescribed polynomial subspace. In the lowest-order conforming setting considered here, the discrete fields are uniquely determined by vertex degrees of freedom on the boundary and admit a computable constant-gradient representation through the polynomial projection. This structure induces the characteristic VEM decomposition into a consistency contribution, expressed solely in terms of projected quantities, and a stabilization contribution that supplies control on the complementary subspace not captured by the projection. Applications of the Virtual Element Method can be found in \cite{Xu2025,Sharma2025,Qiu2025, Wriggers2022, Wang2026}.

\paragraph{}
The objective of this work is the design of stabilization terms for finite-strain hyperelastic VEM that remain robust with respect to polygon geometry and material parameters, with particular emphasis on the nearly incompressible regime. Since the consistent contribution depends only on the projected component, all deformation patterns in the projector kernel are invisible to the consistent energy and must receive stiffness exclusively through stabilization; consequently, the stabilization acts as an energy (eigenvalue) assignment mechanism for unresolved deformation modes. Many existing finite-strain stabilizations combine two features that can compromise this assignment: evaluation through an internal sub-triangulation and parameter scalings that allow volumetric proxies to enter shear-type penalties. The first introduces a dependence on an arbitrary internal tessellation that is extrinsic to the VEM space, while the second may distort the scaling of isochoric kernel modes as $\nu \to 1/2$. The goal, therefore, is to construct a submesh-free stabilization that acts only on $\missingkernel$, preserves deviatoric/volumetric decoupling within the stabilization channel, and is spectrally equivalent to the element's current tangent energy on $\missingkernel$, with equivalence constants that remain uniform over admissible polygon families and do not deteriorate as $\nu \to 1/2$.

\paragraph{}
In low-order VEM, the consistent contribution is computed via polynomial projections, whereas stability and coercivity are recovered by adding a stabilization that acts exclusively on the unresolved component of the discrete space. In the linear setting, this mechanism is by now classical: the stabilization is not intended to model additional physics, but to provide spectral control on the projector kernel with the correct scaling relative to the target bilinear form. This viewpoint is articulated, for instance, in \cite{beirao2013vem,beirao2014hitchhiker} and systematized in the stabilization survey \cite{Mascotto2023}. Extending the same principle to finite-strain hyperelasticity is conceptually direct, but it raises a delicate design requirement: the missing modes must be assigned physically meaningful energies without contaminating the consistent projection and without introducing artifacts that depend on the mesh geometry or on material parameters, as emphasized in \cite{Mascotto2023}.

\paragraph{}
A widespread strategy in nonlinear (hyperelastic) VEM is to define the stabilization through a surrogate strain-energy density and to approximate the non-computable contribution by integrating that surrogate over an internal subdivision of the polygon into sub-elements; representative examples are given in \cite{Wriggers2017,vanHuyssteen2020}. This choice has two structural drawbacks. First, the stabilization becomes tied to an arbitrary internal tessellation: its value, and hence the stiffness assigned to the missing modes, may depend on the orientation and quality of the chosen sub-triangulation rather than on the polygonal element and the VEM space itself, as can be inferred from the constructions in \cite{Wriggers2017,vanHuyssteen2020}. Second, the stabilization channel ceases to be a pure kernel-control device and instead inherits the full constitutive complexity of the hyperelastic model, together with the burden of consistent differentiation for Newton-type methods; this is precisely the regime in which one would like the stabilization to remain inexpensive, robust, and submesh-free, as already noted in \cite{Wriggers2017}.

\paragraph{}
A second, equally consequential, set of choices concerns how the stabilization is scaled with material parameters, particularly near incompressibility. In several hyperelastic VEM formulations, the stabilization is parameterized through modified Lam\'e constants or equivalent compressibility factors; see, e.g., \cite{Wriggers2017,vanHuyssteen2020}. In \cite{vanHuyssteen2020}, for instance, an incompressibility-dependent scaling is introduced and a truncated Taylor expansion of the volumetric Lam\'e parameter is employed to improve robustness as $\nu \to 1/2$. While such modifications can be effective in practice, they can blur the essential deviatoric/volumetric separation: volumetric proxies are injected into otherwise shear-type stabilization choices, and a volumetric penalty may remain active on the projector kernel even in the nearly incompressible regime, a concern discussed in \cite{vanHuyssteen2020} and framed by the stability yardstick in \cite{Mascotto2023}. The resulting risk is an incorrect spectral assignment on the missing modes: instead of scaling like the current tangent energy on $\ker\Pi$, the stabilization may over-penalize isochoric mechanisms (leading to artificially stiff deviatoric responses) and/or leave a residual compressibility trace when volumetric stiffness should enter primarily through the consistent part; this interpretation is consistent with the discussion in \cite{Mascotto2023}.

\paragraph{}
A closely related formulation to \cite{Wriggers2017,vanHuyssteen2020} is developed in \cite{vanHuyssteen2021} for plane problems of transversely isotropic hyperelasticity. The work in \cite{vanHuyssteen2021} retains the same structural stabilization paradigm as \cite{vanHuyssteen2020}: the non-computable contribution is approximated by introducing a neo-Hookean-type stabilization strain-energy density with modified parameters and evaluating the resulting stabilization energy through an internal decomposition of each polygonal element into triangular subdomains. In contrast to the isotropic setting, \cite{vanHuyssteen2021} proposes a stabilization parameter construction that explicitly incorporates anisotropic material features, together with geometric measures of element distortion, in order to obtain robust performance across a range of transversely isotropic material models and in limiting regimes including near-incompressibility and near-inextensibility.

\paragraph{}
Closely related stabilization constructions also appear in finite-deformation VEM formulations for elasto--plasticity. The works \cite{WriggersHudobivnik2017,Hudobivnik2019} adopt the canonical low-order VEM split into a projected (constant-gradient) consistency contribution and a complementary term acting on the non-projected part, and they introduce the stabilization at the level of an incremental pseudo-energy functional. In particular, the stabilization is built through an energy-difference structure of the form $U^{\mathrm{stab}} = \widehat U(u_h) - \widehat U(\Pi u_h)$, where $\widehat U$ is a modified positive-definite pseudo-energy evaluated on an internal simplex partition (triangles in 2D and tetrahedra in 3D), while the subtraction enforces consistency of the overall formulation. The constitutive content entering the stabilization is chosen in a neo-Hookean-type form with modified parameters, selected through geometry-based criteria intended to enhance the stiffness assigned to missing (hourglass/bending-type) mechanisms and to mitigate locking effects in nearly incompressible settings; in the elasto--plastic case, the stabilization additionally employs history variables obtained from the projected state. This places \cite{WriggersHudobivnik2017,Hudobivnik2019} within the same broad stabilization paradigm as finite-strain hyperelastic VEM approaches based on surrogate energies and internal sub-partitions: the stabilization remains an explicit energy assignment on the unresolved modes (i.e., on $\missingkernel$) and its practical behavior can reflect both the internal tessellation choice and the adopted parameter scaling strategy. 

\paragraph{}
These concerns are not merely conceptual. Stabilization parameters effectively assign energies, equivalently eigenvalues, to the modes not represented by the consistent contribution. Consequently, a mode-insensitive scalar scaling can misrepresent the missing-mode spectrum, particularly on distorted or anisotropic polygons where kernel modes are strongly directional. This eigenvalue perspective is developed in \cite{Fujimoto2024} and complements the robustness considerations emphasized in \cite{Mascotto2023}. In this sense, classical one-parameter stabilizations may fail to remain uniformly comparable to the target (tangent) energy on $\ker\Pi$ across polygon shapes and parameter regimes, which is precisely the stability criterion highlighted in \cite{Mascotto2023}.

\paragraph{}
The work in \cite{Fujimoto2024} makes this point explicit by reinterpreting VEM stabilization as an eigenvalue (energy) assignment mechanism. By analyzing the spectrum of the element stiffness split $\mathbf{K}=\mathbf{K}^{c}+\mathbf{K}^{s}$ for $k=1$, where $\mathbf{K}^{c}$ denotes the consistency contribution and $\mathbf{K}^{s}$ the stabilization contribution, the authors in \cite{Fujimoto2024} show that $\mathbf{K}^{c}$ reproduces the correct energies for constant-strain modes while leaving bending/hourglass-type modes with (near-)zero energy, and that common scalar choices of the stabilization parameter effectively select the stiffness of these missing modes, sometimes with poor scaling.

\paragraph{}
Taken together, these observations motivate a stabilization design that (i) remains strictly a kernel-control mechanism rather than a surrogate bulk-integration channel, (ii) preserves deviatoric/volumetric decoupling within the stabilization itself, and (iii) assigns directional, mode-aware stiffness to the unresolved modes so that the stabilization is spectrally equivalent to the element's current tangent energy on $\missingkernel$ across polygon shapes and as $\nu \to 1/2$, in the sense advocated in \cite{Mascotto2023} and guided by the modal interpretation in \cite{Fujimoto2024}.

\paragraph{}
An alternative line of research aims at eliminating the stabilization term altogether by modifying the projection machinery so that the discrete bilinear form is coercive using projected quantities only. The work in \cite{Berrone2024} proposes a stabilization-free virtual element formulation for the Poisson problem in which the key ingredient is a gradient projection onto divergence-free polynomial vector fields (equivalently, curls of suitable scalar polynomial spaces). The coercivity of the resulting discrete form is therefore shifted from a separate kernel-control bilinear form to the choice of the projection space itself: well-posedness follows once the projected-gradient operator is proven to be coercive on the virtual space, under standard mesh regularity assumptions and with a locally adapted enrichment degree. 

\paragraph{}
This idea is further developed in \cite{Berrone2025}, where a high-order stabilization-free scheme is constructed for general 2D second-order elliptic operators with variable coefficients. The work in \cite{Berrone2025} introduces a family of polynomial projection spaces for gradients and derives a necessary and sufficient condition ensuring stability of the projection and, consequently, well-posedness of the discrete problem. In particular, the degree of the projection space may depend on the element geometry and can be selected algorithmically, thereby preserving the VEM philosophy of operating with degrees of freedom and computable projections while avoiding an explicit stabilization term.

\paragraph{}
Beyond a priori theory, the absence of a stabilization bilinear form has also been exploited in the context of adaptivity. The paper \cite{Berrone2026} develops a residual a posteriori error analysis for stabilization-free VEM applied to the Poisson equation and shows that, with a suitable error measure, one can establish an equivalence with standard residual estimators that is typically obscured by stabilization contributions in classical VEM analyses. This highlights a structural advantage of stabilization-free formulations when the goal is a sharp estimator-efficiency theory. 

\paragraph{}
While the stabilization-free approach provides a principled route to avoid arbitrary kernel penalties in linear elliptic settings, it relies on projection spaces and coercivity mechanisms that are tightly coupled to the underlying operator and to the admissible polygon families. In the present work, the focus remains on finite-strain hyperelastic VEM, where the consistent contribution is built from polynomial projections and the unresolved modes in $\missingkernel$ require explicit control at the level of the current tangent energy. In this context, the above stabilization-free developments serve primarily as conceptual motivation: they confirm that the projection channel can be engineered to improve stability properties, but they do not remove the need to design a physically consistent, dev/vol-decoupled kernel-control mechanism for nonlinear elasticity.

\paragraph{}
The stabilization-free formulations in \cite{Berrone2024,Berrone2025,Berrone2026} demonstrate that, for linear elliptic operators, stability can be recovered by enriching and redesigning the projection channel so that coercivity holds without an explicit kernel-control term. This perspective is valuable in the present context because it confirms that the stability mechanism in VEM can be shifted from a separate stabilization bilinear form to suitably chosen computable projections. At the same time, the constructions in \cite{Berrone2024,Berrone2025} rely on operator-specific coercivity arguments and on projection spaces whose definition and degree may need to be adapted locally to geometry and coefficients, with stability ensured through conditions tailored to the linear setting.

\paragraph{}
In finite-strain hyperelasticity, the situation differs in two essential respects. First, the target form is the state-dependent tangent operator arising from the current deformation, so the stability requirement concerns spectral equivalence with the current tangent energy on $\missingkernel$ at each Newton iterate, rather than coercivity of a fixed linear operator. Second, the physically relevant structure is a deviatoric/volumetric split whose limiting behavior as $\nu \to 1/2$ must be respected at the discrete level. In this regime, enforcing stability solely through a modified projection would require a deformation- and material-dependent redesign of the projection spaces and of the associated coercivity proofs, effectively transferring to the projection channel the same burdens that stabilization is intended to handle: robustness across polygon shapes, uniform control in the nearly incompressible limit, and compatibility with consistent linearization. For these reasons, the present work retains an explicit stabilization term and focuses on its design as a pure kernel-control mechanism, constructed to be submesh-free, dev/vol-decoupled, and uniformly spectrally equivalent to the element's current tangent energy on $\missingkernel$.

\paragraph{}The main contributions of this paper are:
\begin{enumerate}
    \item \textbf{A spectral viewpoint for finite-strain VEM stabilization on $\missingkernel$.}
    The stabilization requirement is formulated in terms of spectral equivalence between the stabilization bilinear form and the element Newton tangent energy restricted to $\missingkernel$. This yields a basis-independent characterization in terms of generalized Rayleigh quotients and generalized eigenvalues, providing both a precise stability yardstick and a practical diagnostic for assessing how a stabilization assigns stiffness to unresolved deformation patterns.

    \item \textbf{A kernel-only, deviatoric/volumetric decoupled stabilization that is submesh-free and exactly linearizable.}
    A new stabilization is proposed that depends only on the projector residual at the degrees of freedom, and therefore vanishes on $\left[\mathbb{P}_1(E)\right]^2$ by construction. The stabilization is split into a deviatoric channel scaled solely by a shear measure $\mu_E$ and a volumetric channel scaled by an independent bulk measure $\kappa_E$, without allowing bulk-dependent proxies to inflate the deviatoric penalty. The construction avoids internal sub-triangulations and yields closed-form residuals and tangents, enabling consistent Newton linearization.

    \item \textbf{Mode-aware directional scaling for anisotropic polygonal elements.}
    Motivated by the interpretation of stabilization as an energy assignment on missing modes, the proposed deviatoric stabilization incorporates inexpensive geometry-driven principal directions and bounded anisotropy weights. This allows the stabilization to redistribute stiffness across unresolved directions in a controlled manner on anisotropic or distorted polygons, where scalar stabilization parameters are known to mis-scale kernel modes.

    \item \textbf{Uniform-in-$\nu$ scaling results for the deviatoric stabilization on $\missingkernel$.}
    Under standard polygon regularity assumptions, it is shown that the proposed deviatoric stabilization is bounded above and below by $\mu_E|\cdot|_{1,E}^2$ on $\missingkernel$, with constants independent of $h_E$ and independent of $\nu$. This provides a theory-first justification that the deviatoric stabilization remains shear-scaled in the nearly incompressible limit, and it clarifies how the volumetric channel may be capped or suppressed to avoid residual compressibility on $\missingkernel$.

    \item \textbf{Element-level and benchmark evidence isolating and resolving the volumetric-proxy mechanism.}
    A single-element isochoric kernel-mode diagnostic is introduced to test whether stabilization assigns bulk-driven stiffness to volume-preserving missing modes as $\nu\to 1/2$. The diagnostic shows that classical surrogate-based stabilizations scale isochoric kernel energies with an inflated effective shear parameter, whereas the proposed decoupled stabilization remains shear-scaled. Kernel-restricted modal spectra further confirm the intended kernel-only property, directional stiffness redistribution, and deviatoric/volumetric separation. Finally, Cook's membrane in the nearly incompressible regime demonstrates that the proposed stabilization yields a coherent refinement response across mesh families and mitigates locking-like behavior observed under classical scaling choices.
\end{enumerate}

\paragraph{}The paper is organized as follows. Section~\ref{sec:vem_finite_elasticity} introduces the finite-strain hyperelastic model, the lowest-order conforming virtual element space on polygonal meshes, and the VEM energy split into a computable consistency contribution and a stabilization contribution acting on $\missingkernel$. Section~\ref{sec:stabilization_scaling} reviews classical finite-strain stabilization constructions and introduces the proposed kernel-only deviatoric/volumetric decoupled stabilization, together with scaling results showing shear-consistent behavior on $\missingkernel$ and controlled volumetric contributions in the nearly incompressible regime. Section~\ref{sec:spectral_analysis} develops a spectral viewpoint for stabilization on $\missingkernel$, relating the canonical VEM stability inequality to generalized Rayleigh quotients and generalized eigenvalue bounds, thereby providing a basis-independent diagnostic for missing-mode energy assignment. Section~\ref{sec:numerical_experiments} reports numerical experiments, including element-level diagnostics (single-element isochoric kernel-mode tests and kernel-restricted modal spectra) and the Cook's membrane benchmark under near incompressibility, to assess robustness across element geometries and parameter regimes. Concluding remarks are given in Section~\ref{sec:conclusions}. For completeness, the functional-analytic notation and the main Sobolev and fractional Sobolev space definitions used in the analysis are summarized in Appendix~\ref{ap:functional_analysis}.

\section{The virtual element method applied to finite elasticity}
\label{sec:vem_finite_elasticity}

\paragraph{}This section summarizes the finite-strain virtual element formulation adopted in the present work and introduces the notation used in the subsequent stabilization analysis. The governing hyperelastic problem is first recalled in variational form, together with the state-dependent Newton linearization and the associated tangent bilinear form. The virtual element discretization is then constructed in the lowest-order conforming setting ($k=1$) on general polygonal meshes: the local space is specified through boundary degrees of freedom and harmonic extension, and the computable component of the method is obtained via polynomial projection operators.

\paragraph{}A central object is the projector $\projector$, which maps a virtual field to $\left[\mathbb{P}_1(E)\right]^2$ by prescribing its constant gradient. This induces the canonical decomposition of each local space into a polynomial image and a complementary kernel subspace. The consistency contribution to the discrete energy and tangent depends only on the projected component and therefore vanishes on the kernel. As a consequence, whenever $\dim(\missingkernel)>0$ (i.e., for polygonal elements with $N_E>3$), stabilization is required to control the missing modes while preserving polynomial consistency. The section concludes by stating the standard VEM stability requirement: the stabilization must vanish on polynomial fields and must be uniformly equivalent to the target tangent energy on $\missingkernel$. This requirement provides the reference scaling criterion for the stabilization designs examined in the subsequent sections.

\subsection{Governing equation for finite elasticity}

\paragraph{}Let $\Omega \subset \mathbb{R}^2$ denote the reference configuration of a hyperelastic body, with boundary $\partial\Omega$ decomposed into disjoint parts $\Gamma_D$ and $\Gamma_N$ such that $\partial\Omega=\overline{\Gamma_D}\cup\overline{\Gamma_N}$ and $\Gamma_D\cap\Gamma_N=\emptyset$. The unknown displacement field is $\mathbf{u}:\Omega\rightarrow \mathbb{R}^2$, and the deformation mapping is
\begin{equation}
    \bm{\varphi}(\mathbf{X}) = \mathbf{X} + \mathbf{u}(\mathbf{X}).
\end{equation}
The deformation gradient and the right Cauchy--Green tensor are defined by
\begin{equation}
    \mathbf{F}(\mathbf{u}) = \nabla \bm{\varphi} = \mathbf{I} + \nabla \mathbf{u},
    \qquad
    \mathbf{C}(\mathbf{u}) = \mathbf{F}(\mathbf{u})^{\mathsf{T}}\mathbf{F}(\mathbf{u}),
    \qquad
    J(\mathbf{u}) = \det(\mathbf{F}(\mathbf{u})).
\end{equation}

\paragraph{}A material is termed hyperelastic if there exists a strain-energy density $\psi=\psi(\mathbf{F})$ such that the first Piola--Kirchhoff stress tensor satisfies
\begin{equation}\label{eq:first_piola_kirchhoff}
    \mathbf{P}(\mathbf{F}) = \frac{\partial \psi}{\partial \mathbf{F}}(\mathbf{F}).
\end{equation}
Equivalently, by expressing $\psi$ as a function of $\mathbf{C}$, the second Piola--Kirchhoff stress tensor is
\begin{equation}\label{eq:second_piola_kirchhoff}
    \mathbf{S}(\mathbf{C}) = 2\,\frac{\partial \psi}{\partial \mathbf{C}}(\mathbf{C}),
    \qquad
    \mathbf{P}(\mathbf{F})=\mathbf{F}\mathbf{S}(\mathbf{C}).
\end{equation}
For isotropic materials, $\psi$ may be expressed in terms of invariants of $\mathbf{C}$ (and $J$). A compressible neo-Hookean model is adopted as a representative example,
\begin{equation}\label{eq:neo_hookean_energy}
    \psi(\mathbf{F}) = \frac{\mu}{2}\left(I_1 - 2\log(J) - 2\right) + \frac{\lambda}{2}\log(J)^2,
\end{equation}
where $I_1=\mathrm{tr}(\mathbf{C})$, and $\mu$ and $\lambda$ are the Lam\'e parameters,
\begin{equation}
    \mu = \frac{E_Y}{2(1+\nu)},
    \qquad
    \lambda = \frac{E_Y\,\nu}{(1+\nu)(1-2\nu)}.
\end{equation}
The additive constant in (\ref{eq:neo_hookean_energy}) is immaterial and may be chosen such that $\psi(\mathbf{I})=0$.

\paragraph{}The equilibrium equations in the reference configuration read
\begin{equation}\label{eq:strong_form_hyperelastic}
    -\nabla\cdot \mathbf{P}(\mathbf{F}(\mathbf{u})) = \rho^0\mathbf{b}
    \quad \text{in }\Omega,
    \qquad
    \mathbf{u}=\overline{\mathbf{u}}
    \quad \text{on }\Gamma_D,
    \qquad
    \mathbf{P}(\mathbf{F}(\mathbf{u}))\mathbf{n}=\overline{\mathbf{t}}
    \quad \text{on }\Gamma_N,
\end{equation}
where $\rho^0$ is the reference mass density, $\mathbf{b}$ is the prescribed body force per unit reference volume, $\overline{\mathbf{t}}$ is the prescribed traction per unit reference area, and $\mathbf{n}$ is the outward unit normal.

\paragraph{}The problem may be posed variationally by minimizing the total potential energy
\begin{equation}\label{eq:total_potential_energy}
    \Pi(\mathbf{u}) = \int_{\Omega} \psi(\mathbf{F}(\mathbf{u}))\, d\Omega
    - \int_{\Omega} \rho^0 \mathbf{b}\cdot \mathbf{u}\, d\Omega
    - \int_{\Gamma_N} \overline{\mathbf{t}}\cdot \mathbf{u}\, dS,
\end{equation}
over the admissible set $\mathcal{V}_{\overline{\mathbf{u}}}=\{\mathbf{w}\in[H^1(\Omega)]^2:\ \mathbf{w}=\overline{\mathbf{u}}\ \text{on }\Gamma_D\}$. The corresponding weak form is: find $\mathbf{u}\in \mathcal{V}_{\overline{\mathbf{u}}}$ such that, for all $\mathbf{v}\in \mathcal{V}_0=\{\mathbf{w}\in[H^1(\Omega)]^2:\ \mathbf{w}=\mathbf{0}\ \text{on }\Gamma_D\}$,
\begin{equation}\label{eq:weak_form_hyperelastic}
    \int_{\Omega} \mathbf{P}(\mathbf{F}(\mathbf{u})):\nabla \mathbf{v}\, d\Omega
    =
    \int_{\Omega} \rho^0 \mathbf{b}\cdot \mathbf{v}\, d\Omega
    + \int_{\Gamma_N} \overline{\mathbf{t}}\cdot \mathbf{v}\, dS.
\end{equation}

\paragraph{}For Newton-type methods, the linearization of (\ref{eq:weak_form_hyperelastic}) at a given displacement field $\mathbf{u}^*$ introduces the (state-dependent) bilinear form
\begin{equation}\label{eq:tangent_bilinear_form}
    a_{\Omega}(\mathbf{w},\mathbf{v}; \mathbf{u}^*) =
    \int_{\Omega} \mathbb{A}(\mathbf{F}(\mathbf{u}^*))\big[\nabla \mathbf{w}\big]:\nabla \mathbf{v}\, d\Omega,
\end{equation}
where $\mathbb{A}(\mathbf{F})=\partial \mathbf{P}/\partial \mathbf{F}(\mathbf{F})$ denotes the fourth-order material tangent, and $\mathbb{A}(\mathbf{F})[\mathbf{H}]$ is the second-order tensor obtained by applying $\mathbb{A}$ to $\mathbf{H}\in\mathbb{R}^{2\times 2}$. The associated residual functional is
\begin{equation}\label{eq:residual_functional}
    R(\mathbf{u};\mathbf{v}) =
    \int_{\Omega} \mathbf{P}(\mathbf{F}(\mathbf{u})):\nabla \mathbf{v}\, d\Omega
    - \int_{\Omega} \rho^0 \mathbf{b}\cdot \mathbf{v}\, d\Omega
    - \int_{\Gamma_N} \overline{\mathbf{t}}\cdot \mathbf{v}\, dS.
\end{equation}

\subsection{Basic concepts}
Let $\Omega \subset \mathbb{R}^2$ be a bounded Lipschitz polygonal domain with boundary $\partial \Omega$, and let $\{\mathcal{T}_h\}_h$ be a family of polygonal partitions of $\Omega$ into simple polygons $E$. Denote by $h_E:=\mathrm{diam}(E)$ the element diameter, by $|E|$ the element area, by $N_E$ the number of vertices of $E$, and by $e\subset\partial E$ a generic edge with length $|e|$. Let $h:=\max_{E\in\mathcal{T}_h} h_E$ be the discretization parameter. The following regularity assumptions are adopted throughout.
\begin{enumerate}\label{assumption:regularity}
    \item[(M1)] \emph{Uniform star-shapedness:} there exists $\rho_0\in(0,1]$, independent of $h$, such that every element $E\in\mathcal{T}_h$ is star-shaped with respect to a ball of radius $\rho_0 h_E$, i.e., there exists $\mathbf{x}_E\in E$ such that
    \[
        B(\mathbf{x}_E,\rho_0 h_E)\subset E.
    \]

    \item[(M2)] \emph{No vanishing edges:} there exists $c_e>0$, independent of $h$, such that every edge $e\subset\partial E$ satisfies $|e|\ge c_e\, h_E$.

    \item[(M3)] \emph{Uniformly bounded number of edges:} there exists $N_{\max}$ such that $N_E\le N_{\max}$ for all $E\in\mathcal{T}_h$.

    \item[(M4)] \emph{Area scaling:} as a consequence of (M1), there exist constants $c_A,C_A>0$, depending only on $\rho_0$, such that
    \[
        c_A\, h_E^2 \le |E| \le C_A\, h_E^2
        \qquad \forall E\in\mathcal{T}_h .
    \]
\end{enumerate}
Unless stated otherwise, quasi-uniformity is not assumed.

\paragraph{}For the lowest-order conforming formulation ($k=1$), the local virtual element space is defined by:
\begin{equation}\label{eq:virtual_element_space}
    V_{h,E} = \left\{ \mathbf{v} \in \left[ H^1(E)\right]^2:\ \mathbf{v}|_{\partial E}\in \left[C^0(\partial E)\right]^2,\ \mathbf{v}|_e \in \left[\mathbb{P}_1(e)\right]^2,\ \Delta \mathbf{v}=\mathbf{0}\ \text{in }E,\ \forall e\subset \partial E \right\}.
\end{equation}
Therefore, functions in $V_{h,E}$ are vector-valued harmonic functions in $E$ whose trace is piecewise linear along $\partial E$. In the case $k=1$, the degrees of freedom are the vertex values; consequently,
\begin{equation}
    \dim(V_{h,E}) = 2N_E,
\end{equation}
where $N_E$ is the number of vertices of $E$. Since $\left[H^1(\Omega)\right]^2$ is a Hilbert space, its standard inner product
\begin{equation}
    \langle \mathbf{u},\mathbf{v} \rangle_{H^1(\Omega)} = \int_{\Omega} \mathbf{u}\cdot \mathbf{v}\, d\Omega + \int_{\Omega} \nabla \mathbf{u}:\nabla \mathbf{v}\, d\Omega,
\end{equation}
for all $\mathbf{u}, \mathbf{v} \in \left[ H^1(\Omega)  \right]^2$, induces the norm
\begin{equation}
    \| \mathbf{u} \|_{H^1(\Omega)} = \left(\langle \mathbf{u},\mathbf{u} \rangle_{H^1(\Omega)}\right)^{1/2}.
\end{equation}
The global conforming space associated with $\mathcal{T}_h$ is defined by enforcing elementwise membership and homogeneous Dirichlet boundary conditions:
\begin{equation}
    V_h = \left\{ \mathbf{v} \in \left[H^1_0(\Omega)\right]^2:\ \mathbf{v}|_E \in V_{h,E}\ \ \forall E\in \mathcal{T}_h \right\}.
\end{equation}

\paragraph{}The projection operator $\projector$ maps virtual functions to the polynomial space $\left[\mathbb{P}_1(E)\right]^2$ by prescribing its (constant) gradient. Specifically,
\begin{equation}
    \projector: V_{h,E} \longrightarrow \left[\mathbb{P}_1(E)\right]^2,
\end{equation}
such that the gradient satisfies
\begin{equation}\label{eq:projection}
    \nabla\left(\projector \mathbf{u}_h\right) = \frac{1}{|E|} \int_{E} \nabla \mathbf{u}_h\, dE.
\end{equation}
In addition, a normalization condition is imposed to fix the translational part of $\projector \mathbf{u}_h$ (e.g., by matching the mean value at the vertices). This definition yields the decomposition
\begin{equation}
    \mathbf{u}_h = \projector \mathbf{u}_h + \left(\mathbf{u}_h - \projector \mathbf{u}_h\right),
\end{equation}
\begin{equation}\label{eq:disp_field_split}
    V_{h,E} = \projector V_{h,E} \oplus \ker(\projector),
\end{equation}
where 
\begin{equation}
  \ker(\projector)=\{\mathbf{v}\in V_{h,E}:\ \projector\mathbf{v}=\mathbf{0}\}
\end{equation}
denotes the kernel of the projector. Alternative (equivalent) projections are discussed in \cite{Ahmad2013}.

For $k=1$, the trace $\mathbf{u}_h|_{\partial E}$ is completely determined by the vertex degrees of freedom, since $\mathbf{u}_h$ is linear on each edge. Moreover, the projected gradient in (\ref{eq:projection}) can be computed using only boundary data. By the divergence theorem,
\begin{equation}\label{eq:grad_integral_projection}
    \nabla\left(\projector \mathbf{u}_h\right) = \frac{1}{|E|}\int_{\partial E} \mathbf{u}_h \otimes \mathbf{n}\, dS
    = \frac{1}{|E|} \sum_{i=1}^{N_E} \int_{e_i} \mathbf{u}_h \otimes \mathbf{n}_i\, dS,
\end{equation}
where $\mathbf{n}$ denotes the outward unit normal on $\partial E$, and $\mathbf{n}_i$ is its restriction to the edge $e_i$. Since $\mathbf{u}_h|_{e_i}\in \left[\mathbb{P}_1(e_i)\right]^2$, each edge integral in (\ref{eq:grad_integral_projection}) is exactly computable from the endpoint values.

\subsection{Virtual element approximation}

\paragraph{}The virtual element approximation of the finite-strain problem is formulated as a discrete counterpart of the weak form (\ref{eq:weak_form_hyperelastic}). Let $V_h$ be the conforming virtual element space defined previously, and define the subspace of admissible test functions
\begin{equation}
    V_{h,0} = \left\{ \mathbf{v}_h \in V_h:\ \mathbf{v}_h=\mathbf{0}\ \text{on }\Gamma_D\right\}.
\end{equation}
Moreover, let $\overline{\mathbf{u}}_h$ be a suitable virtual element approximation of the prescribed boundary data $\overline{\mathbf{u}}$ on $\Gamma_D$, and define the affine trial space $V_{h,\overline{\mathbf{u}}}=\overline{\mathbf{u}}_h+V_{h,0}$. The discrete problem is: find $\mathbf{u}_h\in V_{h,\overline{\mathbf{u}}}$ such that
\begin{equation}\label{eq:vem_discrete_weak_form}
    R_h(\mathbf{u}_h;\mathbf{v}_h)=0 \qquad \forall \mathbf{v}_h\in V_{h,0},
\end{equation}
where $R_h$ is a computable approximation of the continuous residual in (\ref{eq:residual_functional}).

\paragraph{}In the present setting, $R_h$ is obtained from a discrete total potential energy. Specifically, a discrete functional $\Pi_h:V_{h,\overline{\mathbf{u}}}\rightarrow \mathbb{R}$ is defined by an elementwise decomposition
\begin{equation}\label{eq:vem_discrete_energy}
    \Pi_h(\mathbf{u}_h) = \sum_{E\in\mathcal{T}_h} \Pi_h^E(\mathbf{u}_h)
    \qquad \text{with} \qquad
    \Pi_h^E(\mathbf{u}_h) = U_{h,E}^c(\mathbf{u}_h^E) + U_{h,E}^s(\mathbf{u}_h^E) - \ell_h^E(\mathbf{u}_h^E),
\end{equation}
where $U_{h,E}^c$ is the consistency contribution, $U_{h,E}^s$ is the stabilization contribution, and $\ell_h^E$ represents the external work. For reference, the exact element internal energy is denoted by
\begin{equation}\label{eq:exact_element_internal_energy}
    U_E(\mathbf{u}) = \int_E \psi\left(\mathbf{F}(\mathbf{u})\right)\, dE,
\end{equation}
which is generally not computable in closed form within the virtual element setting. The discrete residual is defined by the derivative,
\begin{equation}\label{eq:vem_discrete_residual_from_energy}
    R_h(\mathbf{u}_h;\mathbf{v}_h) = D\Pi_h(\mathbf{u}_h)[\mathbf{v}_h] \qquad \forall \mathbf{v}_h\in V_{h,0}.
\end{equation}

\paragraph{}The consistency term is computed by replacing the deformation gradient by its polynomial projection. For each element $E$, define the projected deformation gradient
\begin{equation}
    \mathbf{F}_E(\mathbf{u}_h) = \mathbf{I} + \nabla\left(\projector \mathbf{u}_h^E\right),
    \qquad \mathbf{u}_h^E := \mathbf{u}_h|_E,
\end{equation}
which is constant over $E$ by construction. The consistency energy is then given by
\begin{equation}\label{eq:vem_consistency_energy}
    U_{h,E}^c(\mathbf{u}_h^E) = \int_E \psi\left(\mathbf{F}_E(\mathbf{u}_h)\right)\, dE
    = |E|\,\psi\left(\mathbf{F}_E(\mathbf{u}_h)\right).
\end{equation}

\paragraph{}The stabilization term provides control on the kernel component $\mathbf{u}_h^E-\projector\mathbf{u}_h^E\in \ker(\projector)$, that is, on the part of the discrete field not represented by the polynomial projection. In the present framework, the element internal energy approximation is written as
\begin{equation}\label{eq:vem_element_energy_split}
    U_{h,E}(\mathbf{u}_h^E) = U_{h,E}^c(\mathbf{u}_h^E) + U_{h,E}^s(\mathbf{u}_h^E),
\end{equation}
where $U_{h,E}^c$ is the computable consistency contribution defined in (\ref{eq:vem_consistency_energy}) and $U_{h,E}^s$ is a stabilization contribution. The stabilization is required to satisfy two fundamental principles:
\begin{equation}\label{eq:vem_stabilization_principles}
    U_{h,E}^s(\mathbf{u}_h^E) = U_{h,E}^s\!\left(\mathbf{u}_h^E-\projector\mathbf{u}_h^E\right),
    \qquad
    U_{h,E}^s(\mathbf{p})=0 \ \ \forall \mathbf{p}\in \left[\mathbb{P}_1(E)\right]^2,
\end{equation}
namely, it acts only on $\ker(\projector)$ and vanishes on polynomial fields (polynomial consistency).

\paragraph{}At the level of Newton linearization, the discrete tangent operators are obtained from the second variation of the discrete energy. Let $\ustar\in V_{h,\overline{\mathbf{u}}}$ denote the current Newton state. For each element $E$, define the elementwise discrete bilinear forms associated with the consistency and stabilization energies by
\begin{equation}\label{eq:vem_element_discrete_bilinear_forms}
    a_{h,E}^c(\mathbf{v}_h,\mathbf{w}_h;\ustar)
    :=
    D^2 U_{h,E}^c(\ustar)[\mathbf{v}_h,\mathbf{w}_h],
    \qquad
    a_{h,E}^s(\mathbf{v}_h,\mathbf{w}_h;\ustar)
    :=
    D^2 U_{h,E}^s(\ustar)[\mathbf{v}_h,\mathbf{w}_h],
\end{equation}
for all $\mathbf{v}_h,\mathbf{w}_h\in V_{h,E}$. The corresponding element tangent bilinear form is the sum
\begin{equation}\label{eq:vem_element_discrete_tangent_split}
    a_{h,E}(\mathbf{v}_h,\mathbf{w}_h;\ustar)
    :=
    a_{h,E}^c(\mathbf{v}_h,\mathbf{w}_h;\ustar)
    +
    a_{h,E}^s\!\left(\mathbf{v}_h-\projector\mathbf{v}_h,\mathbf{w}_h-\projector\mathbf{w}_h;\ustar\right).
\end{equation}
In the notation introduced previously,
\begin{equation}
    \abilinear{E}{c}{\mathbf{v}_h}{\mathbf{w}_h}=a_{h,E}^c(\mathbf{v}_h,\mathbf{w}_h;\ustar),
    \qquad
    \abilinear{E}{s}{\mathbf{v}_h}{\mathbf{w}_h}=a_{h,E}^s\!\left(\mathbf{v}_h-\projector\mathbf{v}_h,\mathbf{w}_h-\projector\mathbf{w}_h;\ustar\right).
\end{equation}

\paragraph{}Let $\{\bm{\phi}_i\}_{i=1}^{2N_E}$ be a basis of $V_{h,E}$ associated with the vertex degrees of freedom. The element residual vector $\mathbf{R}_{h,E}(\ustar)\in \mathbb{R}^{2N_E}$ and the element stiffness matrix $\mathbf{K}_{h,E}(\ustar)\in \mathbb{R}^{2N_E\times 2N_E}$ are defined by
\begin{equation}\label{eq:vem_element_residual_and_stiffness}
    \big(\mathbf{R}_{h,E}(\ustar)\big)_i := D\Pi_h^E(\ustar)[\bm{\phi}_i],
    \qquad
    \big(\mathbf{K}_{h,E}(\ustar)\big)_{ij} := a_{h,E}(\bm{\phi}_i,\bm{\phi}_j;\ustar),
\end{equation}
so that $\mathbf{K}_{h,E}(\ustar)=\mathbf{K}_{h,E}^c(\ustar)+\mathbf{K}_{h,E}^s(\ustar)$ with the contributions induced by $a_{h,E}^c$ and $a_{h,E}^s$. The global residual vector and global stiffness matrix are obtained by standard assembly over the mesh. At each Newton step, the increment of the global degrees of freedom is computed from the linear system
\begin{equation}\label{eq:newton_linear_system}
    \mathbf{K}_h(\ustar)\,\Delta \mathbf{a} = -\mathbf{R}_h(\ustar).
\end{equation}

\paragraph{}For stability and robustness, the stabilization contribution must reproduce the correct scaling of the target (projected-state) tangent energy on the missing modes. Define the element tangent bilinear form
\begin{equation}\label{eq:vem_element_tangent_form}
    a_E(\mathbf{v},\mathbf{w};\ustar)
    =
    \int_E \mathbb{A}\left(\mathbf{F}_E(\ustar)\right)\big[\nabla \mathbf{v}\big]:\nabla \mathbf{w}\, dE,
\end{equation}
for sufficiently regular $\mathbf{v},\mathbf{w}$. The classical VEM stability requirement presented in \cite{beirao2013vem} is the existence of constants $C_0,C_1>0$, independent of $E$ and $h$, such that
\begin{equation}\label{eq:vem_stabilization_inequality}
    C_0\, a_E(\mathbf{v}_h,\mathbf{v}_h;\ustar)
    \le
    \abilinear{E}{s}{\mathbf{v}_h}{\mathbf{v}_h}
    \le
    C_1\, a_E(\mathbf{v}_h,\mathbf{v}_h;\ustar)
    \qquad
    \forall \mathbf{v}_h\in V_{h,E}\ \text{with}\ \projector \mathbf{v}_h=\mathbf{0}.
\end{equation}
This inequality guarantees that the stabilization is uniformly equivalent to the target tangent energy on $\ker(\projector)$, thereby ensuring robustness with respect to element shape and mesh size.

\paragraph{}In particular, (\ref{eq:vem_stabilization_principles}) implies that the stabilization bilinear form depends only on the kernel components:
\begin{equation}\label{eq:vem_stabilization_kernel_action}
    \abilinear{E}{s}{\mathbf{v}_h}{\mathbf{w}_h}
    =
    \abilinear{E}{s}{\mathbf{v}_h-\projector\mathbf{v}_h}{\mathbf{w}_h-\projector\mathbf{w}_h}.
\end{equation}

\paragraph{}The external work term is approximated in a computable manner. In particular, for $k=1$, the boundary integral on $\Gamma_N$ is computable since $\mathbf{u}_h$ is piecewise linear on each boundary edge. For body forces, a common choice is to replace $\mathbf{u}_h$ by a computable polynomial projection. For instance, a constant projection on $E$ may be defined by vertex averaging as
\begin{equation}
    \Pi^0_E \mathbf{v}_h = \frac{1}{N_E}\sum_{i=1}^{N_E}\mathbf{v}_h(\mathbf{x}_i)\in \left[\mathbb{P}_0(E)\right]^2,
\end{equation}
so that $\int_E \rho^0\mathbf{b}\cdot \Pi^0_E\mathbf{v}_h\, dE$ is computable. Accordingly, a convenient elementwise definition of the external work functional is
\begin{equation}\label{eq:vem_discrete_load_functional}
    \ell_h^E(\mathbf{v}_h) =
    \int_E \rho^0 \mathbf{b}\cdot \Pi^0_E \mathbf{v}_h\, dE
    + \int_{\partial E\cap \Gamma_N} \overline{\mathbf{t}}\cdot \mathbf{v}_h\, dS.
\end{equation}

\paragraph{}The linearization of (\ref{eq:vem_discrete_weak_form}) yields the discrete (state-dependent) bilinear form used in Newton-type methods. Let $\ustar\in V_{h,\overline{\mathbf{u}}}$ denote the current Newton state. The global tangent bilinear form is defined by
\begin{equation}\label{eq:vem_discrete_tangent_bilinear_form}
    a_{h}(\mathbf{v}_h,\mathbf{w}_h;\ustar)
    =
    DR_h(\ustar)[\mathbf{v}_h,\mathbf{w}_h]
    =
    D^2\Pi_h(\ustar)[\mathbf{v}_h,\mathbf{w}_h],
\end{equation}
for all $\mathbf{v}_h,\mathbf{w}_h\in V_{h,0}$. The elementwise decomposition (\ref{eq:vem_discrete_energy}) implies the splitting
\begin{equation}
    a_{h}(\mathbf{v}_h,\mathbf{w}_h;\ustar)
    =
    \sum_{E\in\mathcal{T}_h}\left(\abilinear{E}{c}{\mathbf{v}_h}{\mathbf{w}_h}+\abilinear{E}{s}{\mathbf{v}_h}{\mathbf{w}_h}\right),
\end{equation}
where $\abilinear{E}{c}{\cdot}{\cdot}$ and $\abilinear{E}{s}{\cdot}{\cdot}$ denote the elementwise contributions associated with the second variations of $U_{h,E}^c$ and $U_{h,E}^s$, respectively. In particular, the consistency contribution admits the explicit expression
\begin{equation}\label{eq:vem_consistency_bilinear_form}
    \abilinear{E}{c}{\mathbf{v}_h}{\mathbf{w}_h}
    =
    |E|\,
    \mathbb{A}\left(\mathbf{F}_E(\ustar)\right)\big[\nabla(\projector \mathbf{v}_h^E)\big]
    :\nabla(\projector \mathbf{w}_h^E),
\end{equation}
where $\mathbb{A}(\mathbf{F})=\partial\mathbf{P}/\partial\mathbf{F}(\mathbf{F})$ is the material tangent introduced in (\ref{eq:tangent_bilinear_form}).  

\subsection{A discussion about missing modes}

\paragraph{}In the virtual element method, the local space $V_{h,E}$ is not handled through explicit interior basis functions. Instead, it is specified by its degrees of freedom and by a set of computable polynomial projection operators. As a consequence, all contributions built solely from projected quantities depend only on the polynomial component of a virtual field. The complementary components, i.e.\ those lying in the kernel of the projector employed in the consistency construction, do not contribute to the consistent part and are referred to as missing modes.

\paragraph{}In the present low-order setting, the consistency term is constructed through the projection $\projector=\Pi_E^\nabla$, which maps $V_{h,E}$ onto $\left[\mathbb{P}_1(E)\right]^2$. For any $\mathbf{w}_h\in V_{h,E}$, define the projected component and the remainder by
\begin{equation}
    \mathbf{w}_h^{\mathrm{p}} := \projector \mathbf{w}_h \in \left[\mathbb{P}_1(E)\right]^2,
    \qquad
    \mathbf{w}_h^{\mathrm{r}} := \left(\mathbf{I}-\projector\right)\mathbf{w}_h \in V_{h,E}.
\end{equation}
Since $\projector$ is a projector, $(\projector)^2=\projector$, and therefore
\begin{equation}
    \projector \mathbf{w}_h^{\mathrm{r}}
    =
    \projector\left(\mathbf{I}-\projector\right)\mathbf{w}_h
    =
    \left(\projector-(\projector)^2\right)\mathbf{w}_h
    =
    \mathbf{0}.
\end{equation}
Hence $\mathbf{w}_h^{\mathrm{r}}\in\missingkernel$ and $\mathbf{w}_h=\mathbf{w}_h^{\mathrm{p}}+\mathbf{w}_h^{\mathrm{r}}$. In particular, the decomposition induced by $\projector$ yields
\begin{equation}
    V_{h,E} = \left[\mathbb{P}_1(E)\right]^2 \oplus \missingkernel,
\end{equation}
where $\missingkernel$ collects the components of $V_{h,E}$ that are not represented by the polynomial projection used in the consistency term.

\paragraph{}The existence of a nontrivial projector kernel follows from a dimension argument. For the standard lowest-order VEM for vector-valued displacements in two dimensions, the degrees of freedom are the displacement values at the $N_E$ vertices of $E$, so that $\dim(V_{h,E})=2N_E$. On the other hand, $\dim\!\left(\left[\mathbb{P}_1(E)\right]^2\right)=6$. Therefore,
\begin{equation}
    \dim(\missingkernel)=\dim(V_{h,E})-\dim\!\left(\left[\mathbb{P}_1(E)\right]^2\right)=2N_E-6.
\end{equation}
For triangular elements ($N_E=3$), $\dim(\missingkernel)=0$ and the projection accounts for the full space. For polygonal elements with $N_E>3$, $\dim(\missingkernel)>0$ and the kernel must be controlled to ensure well-posedness of the discrete problem at the element level.

\paragraph{}In low-order VEM, the consistent contribution is built exclusively from projected quantities. Schematically,
\begin{equation}
    \ahE[c] (\mathbf{u}_h, \mathbf{v}_h; \mathbf{u}^*) = a_E (\projector \mathbf{u}_h, \projector \mathbf{v}_h; \mathbf{u}^*).
\end{equation}
Let $\mathbf{w}_h \in \missingkernel$. Then $\projector \mathbf{w}_h = \mathbf{0}$, and for any $\mathbf{v}_h \in V_{h,E}$,
\begin{equation}
    \ahE[c] (\mathbf{w}_h, \mathbf{v}_h; \mathbf{u}^*)
    =
    a_E (\projector \mathbf{w}_h, \projector \mathbf{v}_h; \mathbf{u}^*)
    =
    a_E(\mathbf{0}, \projector \mathbf{v}_h; \mathbf{u}^*)
    =
    0.
\end{equation}
Consequently, every $\mathbf{w}_h \in \missingkernel$ produces zero contribution to the element energy and to the element tangent associated with the consistent part. Equivalently, $\missingkernel$ is contained in the nullspace of the consistent bilinear form.

\paragraph{}It follows that, whenever $\dim(\missingkernel)>0$, the consistent term alone cannot yield a coercive element contribution on $V_{h,E}$, and the resulting element stiffness operator is rank deficient. Stabilization is therefore required to supply control on $\missingkernel$ while preserving polynomial consistency.

\paragraph{}The standard requirement is that the stabilization vanishes on the polynomial subspace and is spectrally equivalent to the target bilinear form on the projector kernel. In particular, for all $\mathbf{w}_h\in\missingkernel$ the stability inequality (\ref{eq:vem_stabilization_inequality}) is required to hold. This condition expresses that the stabilization provides the correct scaling of the unresolved modes without altering the polynomially consistent response.

\section{Stabilization scaling for finite-strain}
\label{sec:stabilization_scaling}
\paragraph{}This section addresses the scaling of stabilization terms for finite-strain virtual element formulations. In the hyperelastic VEM split, the computable consistency contribution is built from polynomial projections, whereas the stabilization acts only on the unresolved subspace $\missingkernel=\ker(\projector)$. As discussed in Section~\ref{sec:spectral_analysis}, robustness requires the stabilization bilinear form to be spectrally equivalent to the target (state-dependent) Newton tangent energy on $\missingkernel$, with equivalence constants uniform over admissible polygon geometries and stable in the nearly incompressible limit. The purpose of this section is twofold: first, to summarize how classical finite-deformation VEM stabilizations realize this mechanism through surrogate energies evaluated on auxiliary sub-triangulations and through effective Lam\'e-type scalings; second, to introduce and analyze a kernel-only stabilization that decouples deviatoric and volumetric contributions, avoids internal tessellations, and admits mesh-robust scaling driven by shear and bulk measures without bulk-dependent modifications of the deviatoric channel.

\paragraph{}To this end, Section~\ref{sec:classic_stab} reviews the standard construction based on the identity (\ref{eq:classical_strain_energy_identity}) and on stabilization energies evaluated over an internal triangulation, together with commonly adopted bounded volumetric proxies and geometry factors used to regularize the near-incompressible regime and distorted polygonal elements. Section~\ref{sec:new_stab} then proposes a stabilization built directly from the projector residual at the degrees of freedom. The deviatoric part is scaled by a shear measure and enhanced by geometry-adapted anisotropy weights, while the volumetric part is defined as a boundary-only penalty on normal residual components with an explicit bulk scaling that can be capped or suppressed as $\nu\to 1/2$. The subsequent results establish the structural properties of the proposed terms and derive uniform bounds showing that the deviatoric stabilization scales like $\mu_E|\cdot|_{1,E}^2$ on $\missingkernel$ with constants independent of $h_E$ and of $\nu$, thereby providing the stability foundation required for the modal diagnostics and numerical assessments reported later.

\subsection{Classic formulation of stabilization term}
\label{sec:classic_stab}

Classical finite-deformation virtual element formulations augment the computable consistency contribution by a stabilization term acting on $\ker(\projector)$, as discussed, e.g., in \cite{Wriggers2017,vanHuyssteen2020,vanHuyssteen2021}. In the lowest-order setting, the projection $\projector$ defined by (\ref{eq:projection}) is constant over each element and determines a projected deformation gradient
\begin{equation}
    \mathbf{F}_E(\mathbf{u}_h)=\mathbf{I}+\nabla(\projector \mathbf{u}_h^E),
    \qquad \mathbf{u}_h^E:=\mathbf{u}_h|_E.
\end{equation}
Accordingly, the element internal energy approximation is written in the split form
\begin{equation}\label{eq:classical_element_energy_split}
    U_{h,E}(\mathbf{u}_h^E)=U_{h,E}^c(\mathbf{u}_h^E)+U_{h,E}^s(\mathbf{u}_h^E),
\end{equation}
where $U_{h,E}^c$ is the consistency contribution defined in (\ref{eq:vem_consistency_energy}), and $U_{h,E}^s$ supplies control on $\ker(\projector)$.

\paragraph{}Following \cite{Wriggers2017,vanHuyssteen2020}, the stabilization is constructed by introducing a modified strain-energy density $\hat{\psi}$ and employing the classical VEM identity
\begin{equation}\label{eq:classical_strain_energy_identity}
    \psi\!\left(\mathbf{F}(\mathbf{u}_h)\right)
    =
    \psi\!\left(\mathbf{F}_E(\mathbf{u}_h)\right)
    +\hat{\psi}\!\left(\mathbf{F}(\mathbf{u}_h)\right)
    -\hat{\psi}\!\left(\mathbf{F}_E(\mathbf{u}_h)\right),
\end{equation}
in which the first term is computable through the projection, whereas the correction is confined to the stabilization channel. A representative choice consistent with (\ref{eq:neo_hookean_energy}) is
\begin{equation}\label{eq:classical_modified_energy}
    \hat{\psi}(\mathbf{F}) = \frac{\hat{\mu}}{2}\left(I_1 - 2\log(J) - 2\right) + \frac{\hat{\lambda}}{2}\log(J)^2,
\end{equation}
where $I_1=\mathrm{tr}(\mathbf{F}^{\mathsf{T}}\mathbf{F})$, $J=\det(\mathbf{F})$, and $\hat{\mu}$ and $\hat{\lambda}$ are effective Lam\'e parameters used exclusively to scale the stabilization.

\paragraph{}Since $\mathbf{u}_h$ is only explicitly known on $\partial E$ for $k=1$, the stabilization part is evaluated by means of an auxiliary triangulation $\mathcal{T}_E$ of $E$ and a piecewise affine extension $\mathbf{u}_h^{\mathrm{tri}}$ matching the vertex degrees of freedom. A common choice is a fan triangulation from an internal point (e.g., the centroid), which introduces no additional degrees of freedom, as illustrated in Figure \ref{fig:triangular_decomposition}.
\begin{figure}[htb]
    \centering
    \includegraphics[width=9cm]{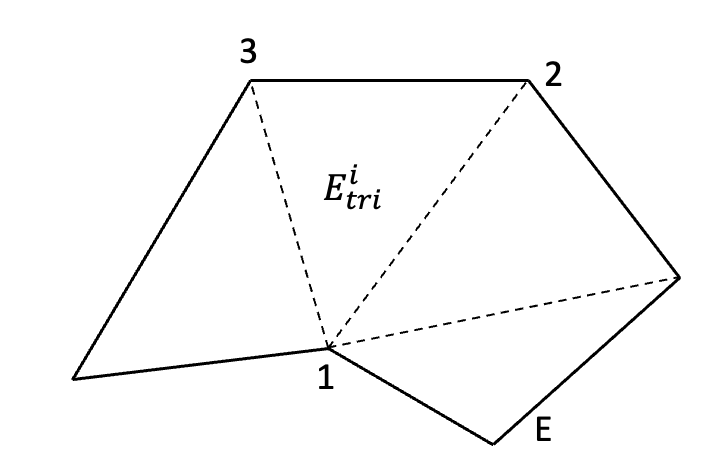}
    \caption{\label{fig:triangular_decomposition}Example of an auxiliary triangulation $\mathcal{T}_E$ used to evaluate the stabilization contribution without introducing additional degrees of freedom.}
\end{figure}
On each triangle $T\in\mathcal{T}_E$, the deformation gradient $\mathbf{F}_T(\mathbf{u}_h)$ is constant and defined by
\begin{equation}
    \mathbf{F}_T(\mathbf{u}_h)=\mathbf{I}+\nabla\!\left(\mathbf{u}_h^{\mathrm{tri}}|_T\right).
\end{equation}
The classical stabilization energy is then defined by
\begin{equation}\label{eq:classical_stabilization_energy}
    U_{h,E}^s(\mathbf{u}_h^E)
    =
    \sum_{T\in\mathcal{T}_E} |T|\left(\hat{\psi}\!\left(\mathbf{F}_T(\mathbf{u}_h)\right)-\hat{\psi}\!\left(\mathbf{F}_E(\mathbf{u}_h)\right)\right).
\end{equation}
By construction, $U_{h,E}^s$ depends only on the missing component $\mathbf{u}_h^E-\projector\mathbf{u}_h^E$ and vanishes when $\mathbf{u}_h^E\in [\mathbb{P}_1(E)]^2$.

\paragraph{}The element load potential is approximated as described in (\ref{eq:vem_discrete_load_functional}). In particular, for $k=1$, the Neumann contribution is computable since $\mathbf{u}_h$ is piecewise linear on boundary edges, whereas the body-force term is evaluated by replacing the virtual field with a computable constant projection on $E$.

\paragraph{}A representative strategy is the introduction of a bounded approximation of the volumetric Lam\'e parameter $\lambda=\lambda(\nu)$ through a truncated Taylor expansion \cite{vanHuyssteen2020}. For a fixed expansion point $\nu_0=-0.25$, a fifth-order approximation is defined by
\begin{equation}\label{eq:taylor_lambda}
    T_5(\lambda) = \lambda(\nu_0) + \sum_{k=1}^{5} \frac{\partial^k \lambda}{\partial \nu^k}\bigg|_{\nu_0}\frac{(\nu-\nu_0)^k}{k!}.
\end{equation}
Since $T_5(\lambda)$ remains bounded as $\nu \rightarrow 1/2$, the volumetric stabilization scaling based on $T_5(\lambda)$ does not diverge in the nearly incompressible limit.

\paragraph{}Element geometry effects are incorporated by means of a geometric factor that depends on an aspect ratio $\mathcal{R}$. Let $R_i$ and $R_o$ denote, respectively, the minor and major semi-axes of the minimal-area ellipse enclosing $E$, and define
\begin{equation}\label{eq:ellipse_aspect_ratio}
    \mathcal{R} = \frac{R_o}{R_i}.
\end{equation}
\begin{figure}[htb]
    \centering
    \includegraphics[width=10cm]{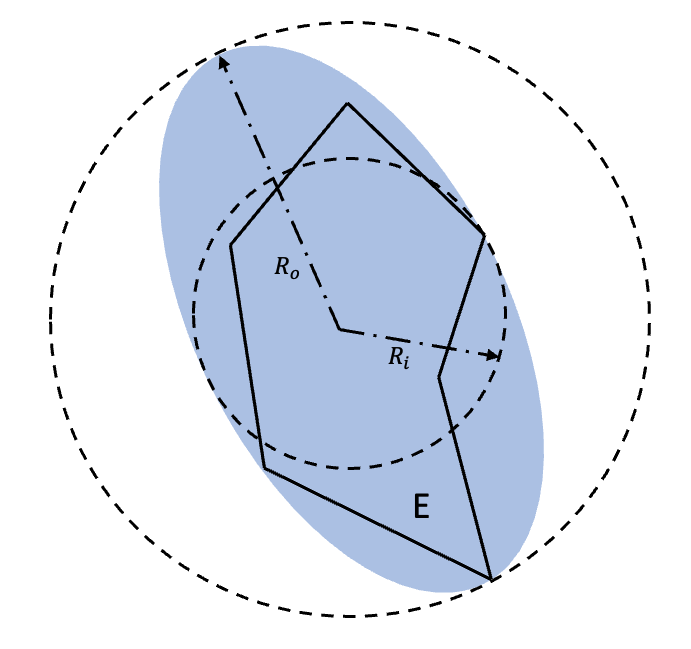}
    \caption{\label{fig:aspect_ratio}Schematic definition of an ellipse-based aspect ratio for a polygonal element.}
\end{figure}
Using $\mathcal{R}$, the geometric factor is defined as
\begin{equation}\label{eq:theta_geometric_factor}
    \Theta = \frac{2(1+\nu)}{\mathcal{R}},
\end{equation}
and the stabilization factor is set as
\begin{equation}\label{eq:phi_stabilization_factor}
    \Phi = \frac{\Theta}{\Theta+1}.
\end{equation}
The modified volumetric Lam\'e parameter is then given by
\begin{equation}\label{eq:lambda_hat_classical}
    \hat{\lambda} = \Phi\,T_5(\lambda).
\end{equation}

\paragraph{}In addition, a modified shear scaling is introduced through an incompressibility indicator
\begin{equation}\label{eq:alpha_incomp_classical}
    \alpha = \frac{T_5(\lambda)}{E_Y},
\end{equation}
leading to the effective shear Lam\'e parameter
\begin{equation}\label{eq:mu_hat_classical}
    \hat{\mu} = (1+\alpha)^2\,\Phi\,\mu.
\end{equation}
These effective parameters are subsequently employed to scale the stabilization channel, with the aim of mitigating locking effects while retaining robustness for distorted polygonal elements.

\subsection{Kernel-only deviatoric/volumetric decoupled stabilization term}
\label{sec:new_stab}

\paragraph{}Let $\{ \mathbf{x}_i \}_{i=1}^{N_E}$ be the collection of the $N_E$ vertices of polygon $E$ with area $|E|$ and polygonal diameter $h_E$. Let the discrete displacement unknowns be the vertex values $\mathbf{u}_i = \mathbf{u}_h(\mathbf{x}_i) \in \mathbb{R}^2$. Define the vertex residuals (kernel components) by:
\begin{equation}\label{eq:residual_kernel_component}
    \mathbf{r}_i (\mathbf{u}_h) = \mathbf{u}_h (\mathbf{x}_i) - \left( \projector \mathbf{u}_h \right)(\mathbf{x}_i) \in \mathbb{R}^2, \; i=1,..., N_E,
\end{equation}
and stack them into
\begin{equation}\label{eq:stacked_residual_kernel}
    \mathbf{r}(\mathbf{u}_h) = [\mathbf{r}_1; ...; \mathbf{r}_{N_E}] \in \mathbb{R}^{2N_E}.
\end{equation} 
This is the residual operator. By construction,
\begin{equation}
    \mathbf{r}(\mathbf{p}) = \mathbf{0}, \; \forall \mathbf{p} \in \left[ \mathbb{P}_1(E) \right]^2,
\end{equation}
so any stabilization depending only on $\mathbf{r}$ vanishes on the polynomial space and acts only on $\missingkernel$.

\paragraph{}Let $\mathbf{x}_C$ be the centroid of $E$. Define the geometry-only second moment matrix:
\begin{equation}
    \mathbf{M}_E = \sum \limits^{N_E}_{i=1} (\mathbf{x}_i - \mathbf{x}_C)(\mathbf{x}_i - \mathbf{x}_C)^T \in \mathbb{R}^{2\times 2}.
\end{equation}
Also, let $(\mathbf{q}_1, \mathbf{q}_2)$ be an orthonormal eigenbasis of $\mathbf{M}_E$ with eigenvalues $\xi_1 \geq \xi_2 > 0$, and let $\mathbf{Q}_E = [\mathbf{q}_1, \mathbf{q}_2] \in \mathbb{R}^{2\times 2}$. Define the aspect indicator as:
\begin{equation}
    \label{eq:aspect_indicator}
    r_E = \sqrt{\frac{\xi_1}{\xi_2}} \geq 1.
\end{equation}
Choose a bounded anisotropy map $g: [1, \infty) \longrightarrow (0,\infty)$:
\begin{equation}
    \label{eq:anisotropy_map}
    g(r) = \min \{ r^\beta, g_{\max} \}, \quad 0 < \beta \leq 1,\quad g_{\max} > 1,
\end{equation}
and define
\begin{equation}
    \tau_1 = g(r_E), \quad \tau_2 = \frac{1}{g(r_E)}.
\end{equation}
Rotate the residual into the principal frame:
\begin{equation}
    \hat{\mathbf{r}}_i = \mathbf{Q}_E^T \mathbf{r}_i = \left[\begin{array}{c}
        \hat{\mathbf{r}}_{i,1} \\
        \hat{\mathbf{r}}_{i,2}
    \end{array}\right].
\end{equation}

\paragraph{}Let $\mu_E > 0$ be a shear scale. For neo-Hookean materials, $\mu_E = \mu$ is adopted (no bulk-dependent proxy). The deviatoric stabilization form is defined as
\begin{equation}
    \label{eq:deviatoric_stab}
    \ahE[s][\text{dev}] (\mathbf{u}_h, \mathbf{v}_h) = \frac{\mu_E |E|}{h_E^2} \sum \limits^{N_E}_{i=1} \tau_1 \hat{\mathbf{r}}_{i,1}(\mathbf{u}_h)\hat{\mathbf{r}}_{i,1}(\mathbf{v}_h) + \tau_2 \hat{\mathbf{r}}_{i,2}(\mathbf{u}_h)\hat{\mathbf{r}}_{i,2}(\mathbf{v}_h).
\end{equation}
Equivalently, in matrix form:
\begin{equation}
    \label{eq:deviatoric_stab_matrix_form}
    \ahE[s][\text{dev}] (\mathbf{u}_h, \mathbf{v}_h) = \frac{\mu_E |E|}{h_E^2} \sum \limits^{N_E}_{i=1} \mathbf{r}_i(\mathbf{u}_h)^T (\mathbf{Q}_E \text{diag}(\tau_1 \tau_2) \mathbf{Q}_E^T) \mathbf{r}_i(\mathbf{v}_h),
\end{equation}
where
\begin{equation}
    \text{diag}(\tau_1 \tau_2) = \left[\begin{array}{cc}
        \tau_1 & 0 \\
        0 & \tau_2
    \end{array}\right].
\end{equation}

\paragraph{}Regarding the volumetric term, let $\kappa_E \geq 0$ be a bulk scale. In the nearly incompressible regime, the recommended choice is $\kappa_E = 0$ (or a cap independent of $\nu$) to avoid residual compressibility on $\missingkernel$. In the compressible regime, one may take $\kappa_E = \kappa$. For each boundary edge $\mathbf{e}=(\mathbf{x}_i, \mathbf{x}_{i+1}) \subset \partial E$, let $|\mathbf{e}|$ denote the length and $\mathbf{n}_e$ a fixed outward unit normal. Define the edge-averaged residual as:
\begin{equation}
    \mathbf{r}_e (\mathbf{u}_h) = \frac{\mathbf{r}_i(\mathbf{u}_h) + \mathbf{r}_{i+1}(\mathbf{u}_h)}{2}.
\end{equation}
The volumetric stabilization bilinear form is given by:
\begin{equation}
    \label{eq:volumetric_stab_matrix_form}
    \ahE[s][\text{vol}] (\mathbf{u}_h, \mathbf{v}_h) = \frac{\kappa_E}{h_E} \sum \limits_{e \subset \partial E} |\mathbf{e}| (\mathbf{r}_e (\mathbf{u}_h)\cdot \mathbf{n}_e)(\mathbf{r}_e (\mathbf{v}_h)\cdot \mathbf{n}_e).
\end{equation}
This penalizes spurious normal expansion or compression in the kernel component using only boundary information. Using (\ref{eq:deviatoric_stab}) and (\ref{eq:volumetric_stab_matrix_form}), the stabilization energy is given by:
\begin{equation}
    U_E^s(\mathbf{u}_h) = \frac{1}{2} \ahE[s][\text{dev}] (\mathbf{u}_h, \mathbf{v}_h) + \frac{1}{2} \ahE[s][\text{vol}] (\mathbf{u}_h, \mathbf{v}_h).
\end{equation}

\paragraph{}The following result establishes the basic structural properties of the proposed deviatoric/volumetric stabilization. Since the stabilization is constructed from the residual operator $\mathbf{r}(\mathbf{u}_h)=\mathbf{u}_h-(\projector\mathbf{u}_h)$ evaluated at the degrees of freedom, the first point verifies kernel consistency, namely that the stabilization vanishes on the polynomial image space and therefore acts exclusively on the unresolved subspace $\missingkernel=\ker(\projector)$. This property is essential for preserving the VEM consistency mechanism and for ensuring that patch-test modes are not altered by the stabilization channel. The second point confirms symmetry, which is required for the stabilization to be the Hessian of an underlying quadratic energy and, consequently, for the discrete Newton tangent to remain symmetric when the constitutive tangent is symmetric.

\paragraph{}The third point addresses coercivity on $\missingkernel$. In the spectral framework of Section~\ref{sec:spectral_analysis}, stability requires the stabilization bilinear form to be positive definite on $\missingkernel$ in order to prevent spurious zero-energy (hourglass) modes and to guarantee that the generalized Rayleigh quotient on $\missingkernel$ is well-defined. The theorem shows that, under an injectivity condition excluding nontrivial kernel functions with vanishing residual at all vertices (modulo admissibility constraints), the deviatoric term defines an SPD form on $\missingkernel$. In addition, the volumetric contribution is shown to be symmetric positive semidefinite for $\kappa_E\ge 0$, so that it can be added without compromising coercivity while remaining a separate, bulk-scaled channel, consistent with the intended deviatoric/volumetric decoupling.

\begin{theo}
    The deviatoric stabilization form in (\ref{eq:deviatoric_stab}) satisfies the following properties:
    \begin{enumerate}
        \item Kernel consistency: $\ahE[s][\text{dev}](\mathbf{p}, \mathbf{v}_h) = 0$ for all $\mathbf{p} \in [\mathbb{P}_1(E)]^2$ and all $\mathbf{v}_h \in V_{h,E}$. In particular, $\ahE[s][\text{dev}]$ depends only on the kernel components $(\text{I}- \projector)\mathbf{u}_h$, i.e. only on $\mathbf{r}$;
        \item Symmetry: $\ahE[s][\text{dev}](\mathbf{u}_h, \mathbf{v}_h) = \ahE[s][\text{dev}](\mathbf{v}_h, \mathbf{u}_h)$, for all $\mathbf{u}_h, \mathbf{v}_h \in V_{h,E}$;
        \item Positive definiteness on $\missingkernel$: assume that the vertex-evaluation map is injective on $\missingkernel$, that is,
        \begin{equation}
            \label{eq:missing_kernel_implication}
            \mathbf{w}_h \in \missingkernel \text{ and } \mathbf{w}_h(\mathbf{x}_i) = \mathbf{0} \text{ for all } i = 1,\ldots,N_E \;\Rightarrow\; \mathbf{w}_h = \mathbf{0}.
        \end{equation}
        Then $\ahE[s][\text{dev}]$ is symmetric positive definite on $\missingkernel$. Moreover, the volumetric term $\ahE[s][\text{vol}]$ (with $\kappa_E \geq 0$) is symmetric and positive semidefinite, so that $\ahE[s] = \ahE[s][\text{dev}] + \ahE[s][\text{vol}]$ is also symmetric positive definite on $\missingkernel$ under (\ref{eq:missing_kernel_implication}).
    \end{enumerate}
\end{theo}

\begin{proof}
    Let $\mathbf{p} \in [\mathbb{P}_1(E)]^2$. Since $\projector$ is a projector onto $[\mathbb{P}_1(E)]^2$,
    \begin{equation}
        \projector \mathbf{p} = \mathbf{p}.
    \end{equation}
    Hence, for every vertex $\mathbf{x}_i$:
    \begin{equation}
        \mathbf{r}_i (\mathbf{p}) = \mathbf{p}(\mathbf{x}_i) - \left( \projector \mathbf{p} \right)(\mathbf{x}_i) = \mathbf{0}.
    \end{equation}
    Define
    \begin{equation}
        \label{eq:support_we_matrix}
        \mathbf{W}_E = \mathbf{Q}_E \text{diag}(\tau_1 \tau_2) \mathbf{Q}_E^T.
    \end{equation}
    Therefore, 
    \begin{equation}
        \ahE[s][\text{dev}](\mathbf{p}, \mathbf{v}_h) = \frac{\mu_E |E|}{h_E^2} \sum \limits^{N_E}_{i=1} \mathbf{r}_i(\mathbf{p})^T \mathbf{W}_E \mathbf{r}_i(\mathbf{v}_h) = 0.
    \end{equation}

    \paragraph{}Since $\mathbf{W}_E$ is symmetric and the scalar products commute:
    \begin{equation}
        \mathbf{r}_i(\mathbf{u}_h)^T \mathbf{W}_E \mathbf{r}_i(\mathbf{v}_h) = \mathbf{r}_i(\mathbf{v}_h)^T \mathbf{W}_E \mathbf{r}_i(\mathbf{u}_h).
    \end{equation}
    Hence, summing over $i$:
    \begin{equation}
        \ahE[s][\text{dev}](\mathbf{u}_h, \mathbf{v}_h) = \ahE[s][\text{dev}](\mathbf{v}_h, \mathbf{u}_h).
    \end{equation}

    \paragraph{} By the definition of $\missingkernel$,
    \begin{equation}
        \mathbf{w}_h \in \missingkernel \Rightarrow \projector \mathbf{w}_h = \mathbf{0},
    \end{equation}
    so that at every vertex
    \begin{equation}
        \mathbf{r}_i(\mathbf{w}_h) = \mathbf{w}_h(\mathbf{x}_i) - (\projector \mathbf{w}_h)(\mathbf{x}_i) = \mathbf{w}_h(\mathbf{x}_i),
        \qquad i = 1,\ldots,N_E.
    \end{equation}
    Moreover, since $\tau_1,\tau_2 > 0$ and $\mathbf{Q}_E$ is orthonormal, the matrix $\mathbf{W}_E$ is symmetric positive definite on $\mathbb{R}^2$. Hence, for any nonzero vector $\mathbf{a} \in \mathbb{R}^2$,
    \begin{equation}
        \mathbf{a}^T\mathbf{W}_E \mathbf{a} > 0.
    \end{equation}

    \paragraph{}Then,
    \begin{equation}
        \ahE[s][\text{dev}](\mathbf{w}_h, \mathbf{w}_h) = \frac{\mu_E |E|}{h_E^2} \sum \limits^{N_E}_{i=1} \mathbf{r}_i(\mathbf{w}_h)^T \mathbf{W}_E \mathbf{r}_i(\mathbf{w}_h) \geq 0,
    \end{equation}
    and equality holds if and only if $\mathbf{r}_i(\mathbf{w}_h) = \mathbf{0}$ for all $i = 1,\ldots,N_E$, that is,
    \begin{equation}
        \label{eq:aux_injective_assumption_2}
        \ahE[s][\text{dev}](\mathbf{w}_h, \mathbf{w}_h) = 0 \;\Rightarrow\; \mathbf{w}_h(\mathbf{x}_i) = \mathbf{0}, \quad i = 1,\ldots,N_E.
    \end{equation}
    Combining (\ref{eq:aux_injective_assumption_2}) with the injectivity assumption (\ref{eq:missing_kernel_implication}) yields $\mathbf{w}_h = \mathbf{0}$, contradicting $\mathbf{w}_h \neq \mathbf{0}$. Therefore,
    \begin{equation}
        \label{eq:aux_injective_assumption_1}
        \ahE[s][\text{dev}](\mathbf{w}_h, \mathbf{w}_h) > 0
        \qquad \forall \mathbf{w}_h \in \missingkernel \setminus \{\mathbf{0}\},
    \end{equation}
    which proves positive definiteness on $\missingkernel$.

    \paragraph{}Finally, $\ahE[s][\text{vol}]$ is a sum of squared edge-normal residual components with prefactor $\kappa_E / h_E \ge 0$, so it is symmetric and positive semidefinite.
\end{proof}

\paragraph{}The next lemma provides a boundary counterpart of the discrete-to-continuous norm equivalences used later in the stability analysis. Since the proposed volumetric stabilization is built from edge-wise quantities involving vertex (or edge-averaged) residuals on $\partial E$, it is necessary to relate such discrete edge-difference measures to an intrinsic boundary seminorm that is stable under polygonal shape variations. The lemma establishes that, for functions that are affine on each edge, the scaled Gagliardo seminorm on $\partial E$ is equivalent to a weighted sum of squared nodal jumps along consecutive vertices. The constants in this equivalence depend only on the mesh regularity parameters $(\rho_0,c_e,N_{\max})$, and are therefore uniform with respect to $h_E$ and independent of material parameters.

\paragraph{}In the present context, this estimate is used as a technical tool to control boundary contributions of kernel components by computable edge-based expressions. In particular, it enables replacing boundary integrals by vertex-difference terms when deriving bounds for stabilization energies and when proving that the proposed edge-based penalties scale like $H^1$-type quantities on $\missingkernel$. Although stated for scalar functions, the same estimate applies to vector-valued functions by applying the inequality to each component and summing the resulting bounds.

\begin{lemma} \label{lemma:difference_inequality}
    Let $g \in C^0(\partial E)$ be affine on each edge such that $g|_{e_i} \in \mathbb{P}_1(e_i)$, with $e_i \in \partial E$ an edge of the polygon $E$. Denote nodal values by $g_i = g(\mathbf{x}_i)$. Define the intrinsic Gagliardo seminorm:
    \begin{equation}
        |g|^2_{\fracsobolev} = \int \limits_{\partial E} \int \limits_{\partial E} \frac{|g(\mathbf{p}) - g(\mathbf{q})|^2}{|\mathbf{p}-\mathbf{q}|^2} ds_\mathbf{p} ds_\mathbf{q}.
    \end{equation}
    Also, define the scaled seminorm:
    \begin{equation}
        |g|^2_{\fracsobolev, h_E} = \frac{1}{h_E} |g|^2_{\fracsobolev}.
    \end{equation}
    Then, there exists constants $C_2, C_3 > 0$, depending only on $\rho_0$, $c_e$ and $N_{\max}$ such that
    \begin{equation}
        \label{eq:difference_inequality}
        C_2 \sum \limits^{N_E}_{i=1} \frac{1}{|e_i|} |g_{i+1} - g_i|^2 \leq |g|^2_{\fracsobolev, h_E} \leq C_3  \sum \limits^{N_E}_{i=1} \frac{1}{|e_i|} |g_{i+1} - g_i|^2.
    \end{equation}
    Equivalently (multiplying by $h_E$),
    \begin{equation}
        \label{eq:difference_inequality_2}
        |g|^2_{\fracsobolev}  \simeq \sum \limits^{N_E}_{i=1} \frac{1}{|e_i|} |g_{i+1} - g_i|^2.
    \end{equation}
\end{lemma}

\begin{proof}
    See Appendix \ref{ap:proofs}.
\end{proof}

\paragraph{}The next result provides a discrete Poincaré--Korn type equivalence on the projector kernel that is central to the scaling analysis of the proposed stabilization. The deviatoric stabilization in (\ref{eq:deviatoric_stab}) is expressed in terms of vertex residuals and therefore naturally yields a vertex-based quadratic form on $\missingkernel$. In order to compare this discrete quantity to the target tangent energy, which is controlled by $H^1$-type seminorms of the displacement field, it is necessary to relate sums of squared nodal values to the $H^1$ seminorm on the element. Theorem~\ref{theo:poincare_korn} provides precisely this bridge: on $\missingkernel$, the vertex-based norm scaled by $h_E^{-2}$ is equivalent to $|\cdot|_{1,E}^2$ with constants depending only on the admissible polygon regularity parameters.

\paragraph{}The restriction to $\missingkernel$ is essential. While a bound of this type cannot hold uniformly on the full space due to the presence of rigid-body and polynomial modes, those modes are eliminated by the kernel constraint $\projector\mathbf{w}=\mathbf{0}$, which enforces a vanishing polynomial content without implying $\nabla\mathbf{w}=\mathbf{0}$. Consequently, kernel functions may carry nontrivial strains while still being orthogonal to the polynomial space, and the theorem ensures that their amplitude at the degrees of freedom is controlled by, and controls, their $H^1$ seminorm. This equivalence is used subsequently to show that the vertex-based deviatoric stabilization scales like $\mu_E|\mathbf{w}|_{1,E}^2$ on $\missingkernel$, yielding mesh-robust constants and supporting the spectral-equivalence viewpoint developed in Section~\ref{sec:spectral_analysis}.

\begin{theo}\label{theo:poincare_korn}
    Under the regularity assumptions (M1)--(M4) (see Assumption~\ref{assumption:regularity}), there exist constants $C_4, C_5>0$ depending only on $\rho_0$, $c_e$ and $N_{\max}$ such that
    \begin{equation}
        \label{eq:poincare_korn}
        C_4 |\mathbf{w}|_{1,E}^2 \leq \frac{1}{h_E^2} \sum \limits_{i=1}^{N_E} |\mathbf{w}(\mathbf{x}_i)|^2 \leq C_5 |\mathbf{w}|_{1,E}^2,
    \end{equation}
    for all $\mathbf{w} \in \missingkernel$.
\end{theo}

\begin{proof}
    See Appendix \ref{ap:theo_poincare_korn}.
\end{proof}

\paragraph{}The following theorem establishes the central stability property of the proposed deviatoric stabilization on the projector kernel. In the finite-strain setting, and in view of the spectral framework of Section~\ref{sec:spectral_analysis}, the stabilization is required to scale like the target tangent energy on $\missingkernel$ with constants that are uniform with respect to admissible polygon shapes and that do not deteriorate in the nearly incompressible limit. Since the deviatoric stabilization is defined through vertex residuals weighted in a geometry-adapted principal frame, the relevant question is whether this vertex-based quadratic form is equivalent to an $H^1$-type seminorm on $\missingkernel$ with robust constants. Theorem~\ref{theo:deviatoric_stab_condition} provides this equivalence.

\paragraph{}More precisely, under the boundedness of the anisotropy weights induced by the map $g$ and under the mesh regularity assumptions (M1)--(M4), the theorem shows that $\ahE[s][\mathrm{dev}]$ defines a coercive and bounded bilinear form on $\missingkernel$ that scales with the shear measure $\mu_E$ and the $H^1$ seminorm on the element. The constants $C_0^{\mathrm{dev}}$ and $C_1^{\mathrm{dev}}$ depend only on geometric regularity parameters and on the prescribed bounds for $(\tau_1,\tau_2)$, and are independent of $h_E$ and of the Poisson ratio. This result supplies the theoretical justification for using $\ahE[s][\mathrm{dev}]$ as a kernel-only stabilization channel that does not inject bulk-dependent scaling into the deviatoric response, thereby supporting robustness on anisotropic polygons and preventing artificial over-penalization of isochoric kernel modes in the nearly incompressible regime.

\begin{theo}\label{theo:deviatoric_stab_condition}
    Assume the anisotropy weights satisfy:
    \begin{equation}
        0 < \tau_{\min} \leq \tau_1, \tau_2 \leq \tau_{\max} < \infty,
    \end{equation}
    with $\tau_{\min}, \tau_{\max}$ depending only on the bounded map $g$. Define the matrix $\mathbf{W}_E$ as in (\ref{eq:support_we_matrix}):
    \begin{equation}
        \mathbf{W}_E = \mathbf{Q}_E \text{diag}(\tau_1 \tau_2) \mathbf{Q}_E^T, \: \mathbf{Q}_E^T \mathbf{Q}_E = \mathbf{I}.
    \end{equation}
    Define the stabilization term as in (\ref{eq:deviatoric_stab_matrix_form}). Then, there exists constants $C_0^{\text{dev}}, C_1^{\text{dev}} > 0$, depending only on $\rho_0$ and $\tau_{\min}, \tau_{\max}$ (hence, independent of $h_E$ and of $\nu$), such that for every $\mathbf{w}_h \in \missingkernel$:
    \begin{equation}
        C_0^{\text{dev}} \mu_E | \mathbf{w}_h |_{1,E}^2 \leq \ahE[s][dev] (\mathbf{w}_h, \mathbf{w}_h) \leq C_1^{\text{dev}} \mu_E | \mathbf{w}_h |_{1,E}^2.
    \end{equation}
\end{theo}

\begin{proof}
    If $\mathbf{w}_h \in \missingkernel$, then $\projector \mathbf{w}_h = \mathbf{0}$, and hence:
    \begin{equation}
        \mathbf{r}_i (\mathbf{w}_h) = \mathbf{w}_h (\mathbf{x}_i).
    \end{equation}
    Since $\mathbf{W}_E$ is SPD with eigenvalues $\tau_1$ and $\tau_2$, for all $\mathbf{a} \in \mathbb{R}^2$:
    \begin{equation}
        \label{eq:aux_stab_with_a}
        \tau_{\min} |\mathbf{a}|^2 \leq \mathbf{a}^T \mathbf{W}_E \mathbf{a} \leq \tau_{\max} |\mathbf{a}|^2. 
    \end{equation}
    By applying (\ref{eq:aux_stab_with_a}) with $\mathbf{a} = \mathbf{w}_h (\mathbf{x}_i)$ and summing over vertices:
    \begin{equation}
        \label{eq:aux_stab_with_a_expanded}
        \frac{\mu_E |E|}{h_E^2} \tau_{\min} \sum \limits^{N_E}_{i=1} |\mathbf{w}_h(\mathbf{x}_i)|^2\leq \ahE[s][\text{dev}] (\mathbf{w}_h, \mathbf{w}_h) \leq  \frac{\mu_E |E|}{h_E^2} \tau_{\max} \sum \limits^{N_E}_{i=1} |\mathbf{w}_h(\mathbf{x}_i)|^2
    \end{equation}
    So, it remains to relate the vertex based quantity $\frac{\mu_E |E|}{h_E^2} \tau_{\min} \sum \limits^{N_E}_{i=1} |\mathbf{w}_h(\mathbf{x}_i)|^2$ to $|\mathbf{w}_h|_{1,E}^2$ on $\missingkernel$.

    \paragraph{}Applying Lemma~\ref{lemma:difference_inequality} and Theorem~\ref{theo:poincare_korn}, and using the area $|E|$ under the same regularity assumptions, there exist constants $\overline{C}_0, \overline{C}_1 > 0$ such that
    \begin{equation}
        \label{eq:aux_intermediary_stab_ineq}
        \overline{C}_0 |\mathbf{w}_h|^2_{1,E} \leq \frac{|E|}{h_E^2} \sum \limits^{N_E}_{i=1} |\mathbf{w}_h(\mathbf{x}_i)|^2 \leq \overline{C}_1 |\mathbf{w}_h|^2_{1,E}.
    \end{equation}
    Combining (\ref{eq:aux_stab_with_a_expanded}) and (\ref{eq:aux_intermediary_stab_ineq}):
    \begin{equation}
        \mu_E \tau_{\min}  \overline{C}_0 |\mathbf{w}_h|^2_{1,E} \leq \ahE[s][\text{dev}] (\mathbf{w}_h, \mathbf{w}_h) \leq \mu_E \tau_{\max}  \overline{C}_1 |\mathbf{w}_h|^2_{1,E}.
    \end{equation}
    Setting
    \begin{equation}
        C_0^{\text{dev}} = \tau_{\min}  \overline{C}_0 \: \text{and} \: C_1^{\text{dev}} = \tau_{\max}  \overline{C}_1
    \end{equation}
    the theorem is proved.
\end{proof}

\paragraph{}The corollary complements Theorem~\ref{theo:deviatoric_stab_condition} by characterizing the basic stability properties of the volumetric kernel penalty. The form $\ahE[s][\mathrm{vol}]$ is constructed as a boundary-only quadratic functional of the normal component of the kernel residual, and therefore it is automatically nonnegative and symmetric for $\kappa_E\ge 0$. This guarantees that the volumetric channel can be added to the deviatoric stabilization without compromising symmetry or coercivity on $\missingkernel$, while keeping the bulk scaling confined to the volumetric part of the stabilization design.

\paragraph{}In addition, a trace estimate yields a bound of $\ahE[s][\mathrm{vol}]$ in terms of bulk norms of $\mathbf{w}_h$, showing that the volumetric penalty does not grow faster than $H^1$-type quantities on admissible polygon families. In the nearly incompressible regime, this observation supports selecting $\kappa_E$ independently of $\nu$ (including the choice $\kappa_E=0$) to avoid injecting a bulk-dependent penalty on $\missingkernel$, while still retaining a controlled volumetric contribution when desired.

\begin{corollary}\label{cor:vol_stab_upper_bound}
    Assume $\kappa_E \ge 0$. Then the volumetric stabilization
    \begin{equation}
        \ahE[s][\mathrm{vol}] (\mathbf{w}_h, \mathbf{w}_h)
        = \frac{\kappa_E}{h_E}\sum_{e \subset \partial E} |e|\,
        \big(\mathbf{r}_e(\mathbf{w}_h)\cdot \mathbf{n}_e\big)^2
    \end{equation}
    is symmetric and positive semidefinite. Moreover, there exists a constant $C^{\mathrm{vol}}>0$, depending only on the mesh-regularity parameters (in particular on $\rho_0$, and on $c_e$, $N_{\max}$ if these are used in the trace constant), such that for all $\mathbf{w}_h \in V_{h,E}$,
    \begin{equation}\label{eq:vol_upper_bound}
        \ahE[s][\mathrm{vol}] (\mathbf{w}_h, \mathbf{w}_h)
        \le C^{\mathrm{vol}}\,\kappa_E\,|\mathbf{w}_h|_{1,E}^2 .
    \end{equation}
\end{corollary}

\section{Spectral analysis of the stabilization term}
\label{sec:spectral_analysis}

\paragraph{}This section establishes a spectral framework for assessing and designing stabilization terms in finite-strain virtual element methods. The element formulation splits into a consistent contribution, computable through polynomial projections, and a stabilization contribution acting only on the unresolved subspace $\missingkernel$. Since robustness requires the stabilization to control exactly these missing modes without dominating the physical response, the analysis focuses on comparing the stabilization energy to the target (Newton) tangent energy restricted to $\missingkernel$.

\paragraph{}The first part recasts the canonical VEM stability requirement on $\missingkernel$ as a \emph{spectral equivalence} condition. In matrix form, this equivalence is characterized by the generalized eigenvalues of a stabilization matrix relative to the target tangent matrix, providing a basis-independent measure of how stiffness is distributed among missing deformation patterns. The second part links this spectral viewpoint to finite-strain Newton linearization: when the consistent energy depends on the displacement only through a projector, its second variation vanishes on $\missingkernel$, so kernel increments receive stiffness solely through stabilization. Taken together, these observations identify stabilization design as a problem of enforcing spectral equivalence on $\missingkernel$ with constants that remain robust under mesh distortion and in the nearly incompressible regime. The resulting eigenvalue diagnostics motivate the modal analyses reported later in Section~\ref{sec:modal_analysis}.

\subsection{Spectral equivalence}
\paragraph{}Fix a Newton state $\mathbf{u}^*$. To simplify notation, the state is omitted from the bilinear forms throughout this section. The stabilization form $\ahE[s](\cdot, \cdot)$ is said to be spectrally equivalent to the tangent $a_E(\cdot, \cdot)$ on $\missingkernel$ if there exist constants $C_0, C_1 > 0$ such that, for every $\mathbf{w}_h \in \missingkernel$, the inequality (\ref{eq:vem_stabilization_inequality}) holds. The lower bound ensures that $\ahE[s](\cdot, \cdot)$ controls every missing mode, ruling out spurious zero-energy (hourglass) modes and preserving coercivity and conditioning. The upper bound prevents the stabilization from being excessively stiff relative to the physics, thus avoiding artificial over-stiffness and stabilization-dominated responses.

\paragraph{}Upon restriction to $\missingkernel$, both bilinear forms become symmetric positive definite (SPD) matrices: $\mathbf{K}_E$ denotes the matrix of $a_E(\cdot, \cdot)$ and $\mathbf{S}_E$ the matrix of the stabilization, each restricted to $\missingkernel$. In this setting, inequality (\ref{eq:vem_stabilization_inequality}) is equivalent to requiring that every generalized eigenvalue $\lambda$ solving
\begin{equation}
    \label{eq:eigenproblem}
    \mathbf{S}_E \mathbf{x} = \lambda \mathbf{K}_E \mathbf{x}
\end{equation}
lies in the interval $[C_0, C_1]$. The next proposition formalizes this statement.
\begin{proposition}\label{prop:spectral_equivalence}
    Choose a basis $\{ \basis[i] \}^{m}_{i=1}$ of $\missingkernel$. Define the matrices:
    \begin{equation}
        \label{eq:componentwise_matrices}
        (\mathbf{K}_E)_{ij} = a_E(\basis[j], \basis[i]), \quad (\mathbf{S}_E)_{ij} = \ahE[s] (\basis[j], \basis[i]).
    \end{equation}
    Assume:
    \begin{enumerate}
        \item both linear forms are symmetric, so $\mathbf{K}_E$ and $\mathbf{S}_E$ are symmetric;
        \item $a_E(\mathbf{w}, \mathbf{w}) > 0$ for all $\mathbf{w} \in \missingkernel - \{\mathbf{0}\}$, and likewise $\ahE[s](\mathbf{w}, \mathbf{w})>0$, for all $\mathbf{w} \neq \mathbf{0}$ on $\missingkernel$.
    \end{enumerate}
    Then, $\mathbf{K}_E$ and $\mathbf{S}_E$ are symmetric positive definite on $\missingkernel$. The inequality (\ref{eq:vem_stabilization_inequality}) for all $\mathbf{w}\in \missingkernel$ is equivalent to say that all generalized eigenvalues $\lambda$ of the eigenpair $(\mathbf{K}_E, \mathbf{S}_E)$ solving (\ref{eq:eigenproblem}) lies in $[C_0, C_1]$.
\end{proposition}
\begin{proof}
    See Appendix \ref{ap:spectral_equivalence}.
\end{proof}

\paragraph{}This result shows that the canonical VEM stability requirement on the projector kernel, expressed by (\ref{eq:vem_stabilization_inequality}), is equivalent to bounding the generalized Rayleigh quotient
\[
    \rho(\mathbf{x})=\frac{\mathbf{x}^{\mathsf{T}}\mathbf{S}_E\mathbf{x}}{\mathbf{x}^{\mathsf{T}}\mathbf{K}_E\mathbf{x}},
\]
and, consequently, to requiring that every generalized eigenvalue of the pair $(\mathbf{S}_E,\mathbf{K}_E)$ remains in a fixed interval $[C_0,C_1]$, independently of the chosen basis of $\missingkernel$ (cf.\ (\ref{eq:eigenproblem}) and (\ref{eq:aux_rayleight_stab_ineq})). The inequality is therefore not merely an abstract coercivity statement, but a spectral equivalence condition between the stabilization-induced stiffness and the target tangent stiffness on the unresolved subspace.

\paragraph{}Two consequences of this result are relevant to the present work. First, it provides a practical diagnostic for stabilization choices: if the generalized eigenvalues spread significantly or drift outside an admissible interval, the stabilization mis-scales some missing deformation patterns, leading either to spurious low-energy modes (loss of coercivity, poor conditioning) or to excessive penalization (stabilization-dominated response). Second, in the finite-strain setting the matrices $\mathbf{K}_E$ and $\mathbf{S}_E$ depend on the current Newton state $\mathbf{u}^*$, so the same spectral criterion applies locally in deformation and can be monitored along the loading path. This motivates the modal analyses in Section~\ref{sec:modal_analysis}, where generalized eigenpairs restricted to $\missingkernel$ are used to assess how candidate stabilizations distribute stiffness among unresolved modes, with emphasis on robustness under polygonal distortion and in the nearly incompressible regime.

\subsection{Spectral analysis}

\paragraph{}The spectral equivalence results above are not only theoretical stability statements, but also provide a basis-independent way to \emph{diagnose} a stabilization choice. Indeed, once both the target tangent form and the stabilization form are restricted to the missing-mode space $\missingkernel$, the VEM stability yardstick reduces to bounds on a generalized Rayleigh quotient and, equivalently, to bounds on generalized eigenvalues. This observation turns the abstract requirement
\[
C_0\,a_E(\delta\mathbf{u},\delta\mathbf{u}) \le \ahE[s](\delta\mathbf{u},\delta\mathbf{u}) \le C_1\,a_E(\delta\mathbf{u},\delta\mathbf{u})
\qquad \forall\,\delta\mathbf{u}\in\missingkernel
\]
into a concrete \emph{mode-by-mode} statement: each generalized eigenvector represents an unresolved deformation pattern, and its eigenvalue measures the ratio between stabilization-induced stiffness and target tangent stiffness along that pattern.

\paragraph{}This modal viewpoint has been emphasized in the lowest-order ($k=1$) setting by \cite{Fujimoto2024}, where eigenpairs of the element stiffness decomposition are used to interpret stabilization as an assignment of energy to deformation patterns not reproduced by the consistent part. In polygonal elements, and especially for anisotropic or distorted shapes, such a mode-resolved interpretation is particularly informative: a single scalar scaling may distribute stiffness unevenly across unresolved patterns, producing spurious soft modes (hourglass behavior) or excessive penalization of specific deformation families.

\paragraph{}In the present finite-strain setting, the same diagnostic can be applied by replacing the linear-elastic stiffness spectrum with the spectrum of the \emph{current Newton tangent} restricted to $\missingkernel$. In view of Proposition~\ref{prop:kernel_tangent_from_stabilization}, this restriction is also the natural object for analyzing stabilization effects: kernel increments are invisible to the consistent contribution, yet may carry non-trivial strains, so their tangent energy must be supplied and correctly scaled by the stabilization channel. The generalized eigenvalues on $\missingkernel$ therefore provide a direct quantitative measure of whether a given stabilization achieves the intended energy scaling at the current deformation state, and whether this scaling remains robust with respect to polygon geometry and (near-)incompressibility. Section~\ref{sec:single_element_kernel_diagnostic} leverages this framework to compare stabilization choices through their missing-mode spectra.

\begin{proposition}\label{prop:kernel_tangent_from_stabilization}
    Assume that the consistency energy depends on $\mathbf{u}$ only through $\projector \mathbf{u}$, i.e.
    \begin{equation}
        U^c_E(\mathbf{u}) = \overline{U}^c_E(\mathbf{u}) (\projector \mathbf{u})
    \end{equation}
    for some functional $\overline{U}^c_E(\mathbf{u})$ defined on the polynomial image space. For every Newton state and every $\increment{u}, \increment{v} \in \missingkernel$,
    \begin{equation}
        D^2U_E^c (\ustar)[\increment{u}, \increment{v}] = 0 \Rightarrow \mathbf{K}_E (\mathbf{u}) |_{\missingkernel} = \mathbf{K}_E^s (\ustar)|_{\missingkernel}.
    \end{equation}
    In other words, on the projector kernel, the entire tangent comes solely from stabilization.
\end{proposition}
\begin{proof}
    See Appendix \ref{ap:kernel_tangent_from_stabilization}.
\end{proof}

\begin{remark}
    Let $Z$ be a real Banach space and let $f:Z \longrightarrow \mathbb{R}$ be Fréchet differentiable at $\mathbf{z}$. By definition, there exists a bounded linear functional $Df(\mathbf{z}) \in Z^*$, where $Z^*$ is the dual space of $Z$, such that:
    \begin{equation}
        f(\mathbf{z}+\mathbf{y}) = f(\mathbf{z}) + Df(\mathbf{z})[\mathbf{y}] + r_f(\mathbf{y}), \; \frac{|r_f(\mathbf{y})|}{\| \mathbf{y}\|_Z}  \longrightarrow \; \text{as} \; \|\mathbf{y} \|_Z \longrightarrow 0.
    \end{equation}
    Here, the direction $\mathbf{y}$ already appears in the linear term as the evaluation $Df(\mathbf{z})[\mathbf{y}] \in \mathbb{R}$.

    If, moreover, $f$ is twice Fréchet differentiable, then the map $Df:Z \longrightarrow Z^*$ is Fréchet differentiable at $\mathbf{z}$. Hence, there exists a bounded linear operator $D^2f(\mathbf{z})\in \mathcal{L}(Z,Z^*)$ such that:
    \begin{equation}
        Df(\mathbf{z} + \mathbf{y}) = Df(\mathbf{z}) + D^2f(\mathbf{z})[\mathbf{y}]+r_{Df}(\mathbf{y}), \; \frac{\|r_{Df}(\mathbf{y})\|_{Z^*}}{\|\mathbf{y}\|_Z} \longrightarrow 0 \; \text{as} \; \| \mathbf{y} \|_Z \longrightarrow 0.
    \end{equation}
    Note that $\mathcal{L}(Z,Z^*)$ is the space of bounded linear operators from $Z$ to $Z^*$. In this second-order expansion, applying $D^2f(\mathbf{z})$ to $\mathbf{y}$ is necessary because $D^2f(\mathbf{z})$ is a linear map with values in $Z^*$: for each increment $\mathbf{y} \in Z$, $D^2f(\mathbf{z})[\mathbf{y}]$ is a functional in $Z^*$. If one wants to use the usual Hessian bilinear action, one evaluates once more on $\bm{\eta}\in Z$, obtaining $D^2f(\mathbf{z})[\mathbf{y}][\bm{\eta}]\in \mathbb{R}$.
\end{remark}

\begin{proposition}\label{prop:quadratic_form_spectral_equivalence}
    Let 
    \begin{equation}
        q^s (\increment{u}) = \ahE[s](\increment{u}, \increment{u})
    \end{equation}
    be the stabilization quadratic form on $\missingkernel$, and let 
    \begin{equation}
        q^*(\increment{u}) = a_E (\increment{u}, \increment{u}; \ustar)
    \end{equation}
    be the target tangent quadratic form on $\missingkernel$. Assume that both are quadratic forms on $\missingkernel$, with $q^*(\increment{u})>0$ for $\increment{u} \neq \mathbf{0}$ (i.e., positive definite on $\missingkernel$). Suppose that there exist constants $0 < C_* \leq C^*$ such that, for all $\increment{u} \in \missingkernel$,
    \begin{equation}
        \label{eq:stability_criteria_eigenproblem}
        C_* q^*(\increment{u}) \leq q^s(\increment{u}) \leq C^* q^*(\increment{u}).
    \end{equation}
    If $\mathbf{K}^s$ and $\mathbf{K}^*$ are the matrices representing $q^s$ and $q^*$, respectively, in any basis of $\missingkernel$, then every generalized eigenvalue $\lambda$ of
    \begin{equation}
        \mathbf{K}^s \mathbf{x} = \lambda \mathbf{K}^* \mathbf{x}
    \end{equation}
    satisfies
    \begin{equation}
        C_* \leq \lambda \leq C^*.
    \end{equation}
\end{proposition}

\begin{proof}
    See Appendix \ref{ap:quadratic_form_spectral_equivalence}.
\end{proof}

\begin{remark}
    In the hyperelastic VEM split, the consistent part is often constructed as a composition with the projector, e.g.,
    \begin{equation}
        U_E^c(\mathbf{u}) = \overline{U}_E^c (\projector \mathbf{u}).
    \end{equation}
    As a consequence, its Newton tangent satisfies
    \begin{equation}
        D^2U_E^c (\mathbf{u})[\increment{u}, \increment{u}] = 0, \; \forall \increment{u} \in \missingkernel,
    \end{equation}
    i.e., the consistent tangent energy vanishes on the projector kernel. This, however, does not imply that the physical tangent energy vanishes on $\missingkernel$. Indeed, $\increment{u}$ means only that the polynomial content is zero, not that $\nabla \increment{u} = \mathbf{0}$. In general $\projector \increment{u} = \mathbf{0}$ does not imply $\nabla \increment{u} = \mathbf{0}$, so kernel increments may still carry non-trivial strains and, thus, nonzero tangent energy. For this reason, in the spectral-equivalence condition one takes $q^*(\increment{u}) = a_E^*(\increment{u}, \increment{u})$ as the tangent physical energy restricted to $\missingkernel$ (or a computable surrogate equivalent to it), rather than $D^2 U_E^c(\mathbf{u})[\increment{u}, \increment{u}]$, which would be identically zero on $\missingkernel$. Also, if a rigid-body increment lies in $\missingkernel$, one restricts to the admissible subspace (e.g., via boundary conditions) so that $q^*(\increment{u}) > 0$ for every $\increment{u} \neq \mathbf{0}$ in that subspace. 
\end{remark}

\begin{remark}
    In the generalized eigenvalue framework, the SPD property of $\mathbf{K}^*$ is imposed as a starting assumption, since $q^s(\increment{u}) > 0$, for all $\increment{u} \neq \mathbf{0}$ (on the admissible subspace) defines a coercive quadratic form and the generalized Rayleigh quotient is well-defined. By contrast, one cannot conclude that $\mathbf{K}^s$ is positive definite (or even symmetric) without additional hypotheses on the stabilization. Indeed, if no lower bound is assumed for $q^s$, a poorly scaled or incomplete stabilization may admit zero-energy modes in $\missingkernel$, i.e. there exists $\increment{u} \neq \mathbf{0}$ such that $q^s(\increment{u}) = \mathbf{0}$, in which case $\mathbf{K}^s$ is only positive semi-definite (hence singular) on $\missingkernel$. Also, if the stabilization tangent is not exact Hessian of an underlying scalar stabilization energy (e.g., due to frozen parameters or ad-hoc linearization), then $\mathbf{K}^s$ may even fail to be symmetric. 
\end{remark}

The results of this section place the hyperelastic VEM stabilization term in a precise spectral framework. Proposition~\ref{prop:kernel_tangent_from_stabilization} identifies the spectrum on $\missingkernel$ as the natural object for stabilization design in the finite-strain setting. Since the consistent energy is typically constructed as a composition with the projector, its second variation vanishes for kernel increments, so that on $\missingkernel$ the discrete Newton tangent is supplied entirely by the stabilization channel. This makes explicit the energy-assignment interpretation underlying the modal viewpoint: unresolved deformation patterns are invisible to the consistent part and receive stiffness only through $\ahE[s](\cdot,\cdot)$. Any mis-scaling in the stabilization therefore translates immediately into mis-scaled energies for missing modes, which may manifest as spurious softness (hourglass modes), artificial over-stiffness, or locking-like behavior, particularly on anisotropic or distorted polygons and in nearly incompressible regimes.

Proposition~\ref{prop:quadratic_form_spectral_equivalence} rephrases the same requirement at the quadratic-form level: if $q^s$ denotes the quadratic form induced by the stabilization tangent on $\missingkernel$ and $q^*$ denotes the target tangent energy restricted to $\missingkernel$ (or a computable surrogate equivalent to it), then spectral equivalence of these forms is equivalent to uniform bounds on the generalized eigenvalues of their matrix representatives. Propositions~\ref{prop:spectral_equivalence}--\ref{prop:quadratic_form_spectral_equivalence} thus reduce the stabilization design objective to the following principle: $\ahE[s](\cdot,\cdot)$ should be designed so that, at each Newton state, its restriction to $\missingkernel$ is spectrally equivalent to the restriction of the target tangent energy, with constants robust with respect to element geometry and material parameters, and in particular without degradation as $\nu \to 1/2$.

This principle provides both a theory-first criterion for robustness (uniform eigenvalue bounds on $\missingkernel$) and a practical diagnostic tool: by computing the generalized eigenpairs on $\missingkernel$, the stabilization-induced stiffness can be inspected mode by mode. Section~\ref{sec:modal_analysis} exploits this equivalence to perform an element-level modal analysis of candidate stabilization strategies, highlighting how different choices distribute energy across missing modes and how this distribution depends on polygon shape and (near-)incompressibility.

\section{Numerical experiments}
\label{sec:numerical_experiments}

\paragraph{}This section presents numerical experiments designed to validate the stabilization design principles developed in Sections~\ref{sec:spectral_analysis} and~\ref{sec:new_stab}. The experiments target two requirements of finite-strain VEM stabilization. First, the stabilization must act only on the unresolved subspace $\missingkernel$ and must not pollute the polynomial (patch-test) content. Second, restricted to $\missingkernel$, the stabilization must scale like the relevant tangent energy with constants that remain robust under admissible polygonal distortions and do not deteriorate as $\nu\to 1/2$. Since the proposed construction decouples deviatoric and volumetric channels and avoids bulk-dependent modifications of the shear scaling, the numerical assessments focus on detecting bulk-driven contamination of isochoric missing modes and on quantifying the mode-by-mode stiffness distribution induced by stabilization.

\paragraph{}The experimental program proceeds from element-level diagnostics to a representative nonlinear benchmark:
\begin{enumerate}
    \item Section~\ref{sec:single_element_kernel_diagnostic} introduces a single-element kernel-mode test that isolates an isochoric deformation pattern in $\missingkernel$ and measures the scaling of stabilization energy as $\nu$ approaches the incompressible limit, thereby directly probing volumetric-proxy leakage into the stabilization channel.
    \item Section~\ref{sec:modal_analysis} performs a kernel-restricted spectral analysis on polygonal elements to assess directional stiffness allocation, deviatoric/volumetric separation, and conditioning trends in a basis-independent manner.
    \item The Cook's membrane benchmark evaluates the macroscopic response under large deformation and shear-dominated loading across multiple mesh families, emphasizing refinement behavior and robustness in the nearly incompressible regime.
\end{enumerate}

\subsection{Single-element kernel mode diagnostic}
\label{sec:single_element_kernel_diagnostic}

A single-element diagnostic is introduced to isolate one mechanism that is critical in the nearly incompressible regime: whether the stabilization term assigns bulk-driven stiffness to deformation patterns that are isochoric. Since stabilization acts exclusively on $\missingkernel=\ker(\projector)$, a contamination of the deviatoric stabilization channel by volumetric proxies may be detected already at the element level, independently of global boundary-value effects and mesh interactions. The diagnostic therefore targets unresolved modes that (i) belong to $\missingkernel$ and (ii) preserve volume, and measures how the stabilization energy assigned to such modes varies as $\nu\to 1/2$.

The test is carried out on a single unit square element $E=[0,1]^2$ with vertices $(0,0)$, $(1,0)$, $(1,1)$, and $(0,1)$, yielding $N_E=4$ vertices and $2N_E=8$ displacement degrees of freedom collected in $\mathbf{u}\in\mathbb{R}^{8}$. Let $\mathbf{P}_E\in\mathbb{R}^{8\times 8}$ denote the matrix representation (in the vertex basis) of the $k=1$ VEM projector $\projector$ defined in (\ref{eq:projection}), so that $\mathbf{P}_E\mathbf{u}$ represents the vertex values of the projected polynomial field in $[\mathbb{P}_1(E)]^2$. Since $\dim([\mathbb{P}_1(E)]^2)=6$, the projector kernel has dimension $\dim\ker(\mathbf{P}_E)=8-6=2$. A kernel deformation mode is generated by sampling a vector $\mathbf{v}\in\mathbb{R}^{8}$ and projecting it onto the kernel,
\begin{equation}\label{eq:kernel_mode_construction}
\mathbf{u}^{\ker} := (\mathbf{I}-\mathbf{P}_E)\mathbf{v}.
\end{equation}
Kernel membership is then verified a posteriori through the relative projector residual
\begin{equation}\label{eq:kernel_membership_check}
\frac{\|\mathbf{P}_E\mathbf{u}^{\ker}\|}{\|\mathbf{u}^{\ker}\|}\approx 0,
\end{equation}
confirming that the consistent part does not detect the mode.

The volumetric character of $\mathbf{u}^{\ker}$ is assessed by considering the piecewise-affine extension $\mathbf{u}^{\ker,\mathrm{tri}}$ associated with the auxiliary fan triangulation $\mathcal{T}_E$ used in the classical stabilization (cf.\ Section~\ref{sec:classic_stab}). On each $T\in\mathcal{T}_E$, define $\mathbf{F}_T(\mathbf{u}^{\ker})=\mathbf{I}+\nabla(\mathbf{u}^{\ker,\mathrm{tri}}|_T)$ and $J_T(\mathbf{u}^{\ker})=\det(\mathbf{F}_T(\mathbf{u}^{\ker}))$. The mode is classified as isochoric when the induced volumetric change remains negligible, for instance in the sense that
\begin{equation}\label{eq:isochoric_check}
\max_{T\in\mathcal{T}_E}\big|J_T(\alpha\,\mathbf{u}^{\ker})-1\big| \ll 1
\end{equation}
for a fixed small amplitude $\alpha$ (here $\alpha=10^{-2}$). This construction yields a deformation pattern of hourglass type: a non-polynomial, unresolved mode that does not induce appreciable volume change and is therefore expected to be controlled by a shear-consistent stabilization.

The stabilization energy is then evaluated on the scaled mode $\mathbf{u}_h(\alpha)=\alpha\,\mathbf{u}^{\ker}$ for both the classical stabilization of Section~\ref{sec:classic_stab} and the proposed decoupled stabilization of Section~\ref{sec:new_stab}, while sweeping the Poisson ratio over $\nu\in\{0.3,\,0.4,\,0.45,\,0.49,\,0.499,\,0.4999\}$. Since the tested mode is isochoric, the physically consistent stiffness scale is the shear modulus $\mu$, and the stabilization energy is reported in the normalized form
\begin{equation}\label{eq:stab_normalization_isochoric}
\widehat{E}_{\mathrm{stab}}(\nu)
:=
\frac{U_{h,E}^s(\alpha\,\mathbf{u}^{\ker};\nu)}{\mu\,\alpha^2}.
\end{equation}
A stabilization that assigns only shear-consistent energy to isochoric kernel modes yields $\widehat{E}_{\mathrm{stab}}(\nu)$ approximately constant as $\nu\to 1/2$, whereas growth of $\widehat{E}_{\mathrm{stab}}(\nu)$ indicates bulk-dependent stiffening of isochoric missing modes, which is the characteristic signature of volumetric-proxy leakage into the deviatoric stabilization channel.

\paragraph{}Figure~\ref{fig:vol_proxy_energy} reports the stabilization energy $U_{h,E}^s(\alpha\,\mathbf{u}^{\ker})$ of an isochoric kernel mode as the Poisson ratio approaches the incompressible limit. Since the tested mode satisfies $\mathbf{u}^{\ker}\in\ker(\projector)$ and induces negligible volume change in the auxiliary triangulation (i.e., $J_T(\alpha\,\mathbf{u}^{\ker})\approx 1$ for all $T\in\mathcal{T}_E$), the physically relevant stiffness scale for this mode is the shear modulus $\mu$. Accordingly, a stabilization that does not introduce volumetric contamination is expected to assign an energy proportional to $\mu\,\alpha^2$ (up to geometry-dependent constants), and therefore to remain bounded and nearly independent of $\nu$ when normalized by $\mu\,\alpha^2$. The raw trends in Figure~\ref{fig:vol_proxy_energy} are consistent with this expectation for the decoupled stabilization: the energy follows only the mild variation of $\mu$ with $\nu$ (for fixed $E_Y$), whereas the classical stabilization exhibits a systematic increase as $\nu\to 1/2$.

\paragraph{}The normalization tests in Figures~\ref{fig:vol_proxy_original}--\ref{fig:vol_proxy_comparison} make the underlying scaling explicit. In Figure~\ref{fig:vol_proxy_original}, the classical stabilization normalized by the physical shear scale, $U_{h,E}^s/(\mu\,\alpha^2)$, increases strongly with $\nu$, showing that the stabilization stiffness assigned to this isochoric kernel mode grows substantially faster than $\mu$. The same energy becomes approximately constant only after normalization by the effective shear parameter $\hat{\mu}$ used internally by the classical formulation, confirming that the assigned stiffness scales primarily with $\hat{\mu}$ rather than with $\mu$. Since $\hat{\mu}$ depends on volumetric proxies (through bounded approximations of $\lambda$ and the incompressibility indicator used to inflate the shear scaling), this provides element-level evidence of bulk-driven contamination of the shear-type kernel penalty. Figure~\ref{fig:vol_proxy_comparison} summarizes the comparison under the physically relevant normalization $U_{h,E}^s/(\mu\,\alpha^2)$: the decoupled stabilization remains essentially constant across $\nu$, while the classical stabilization increases markedly as $\nu\to 1/2$. This diagnostic isolates the mechanism underlying the locking-like behavior observed in global benchmarks, namely the artificial stiffening of unresolved isochoric deformation patterns induced by volumetric proxy leakage into the stabilization channel.

\begin{figure}[H]
    \centering
    \includegraphics[width=0.85\textwidth]{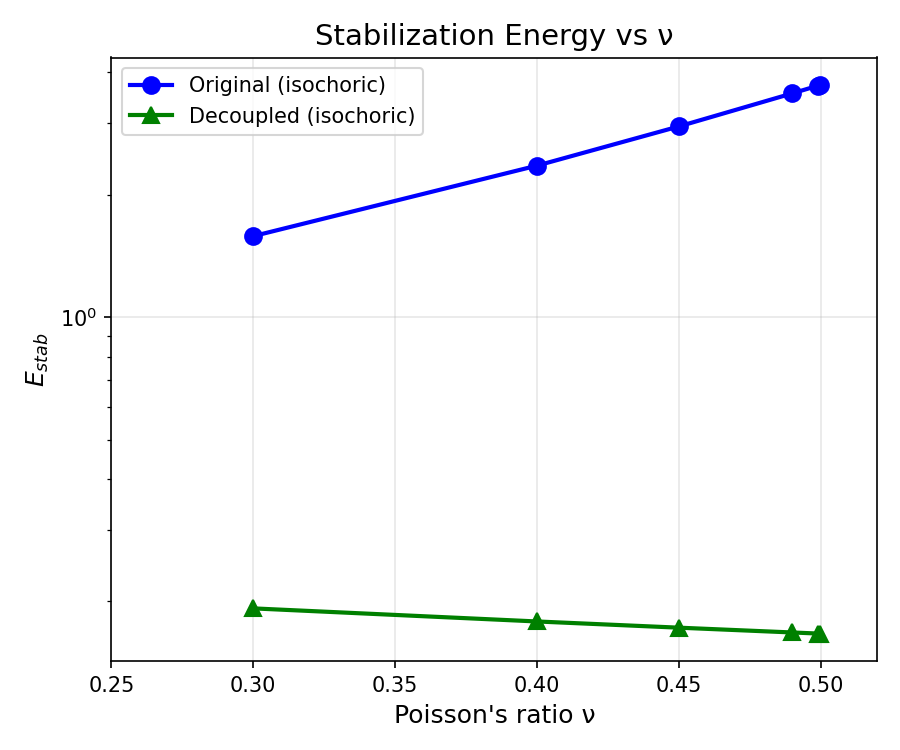}
    \caption{Single-element isochoric kernel-mode diagnostic: raw stabilization energy $U_{h,E}^s(\alpha\,\mathbf{u}^{\ker})$ as a function of Poisson ratio $\nu$ on the unit square. The same kernel mode $\mathbf{u}^{\ker}\in\ker(\projector)$ is used for all $\nu$, with fixed amplitude $\alpha=10^{-2}$ under a fixed Young modulus $E_Y$ (hence $\mu=E_Y/(2(1+\nu))$ varies mildly with $\nu$). The classical stabilization exhibits a monotone increase in energy as $\nu\to 1/2$, despite the mode being isochoric, whereas the decoupled stabilization follows only the shear-scale variation and remains bounded.}
    \label{fig:vol_proxy_energy}
\end{figure}

\begin{figure}[H]
    \centering
    \includegraphics[width=0.85\textwidth]{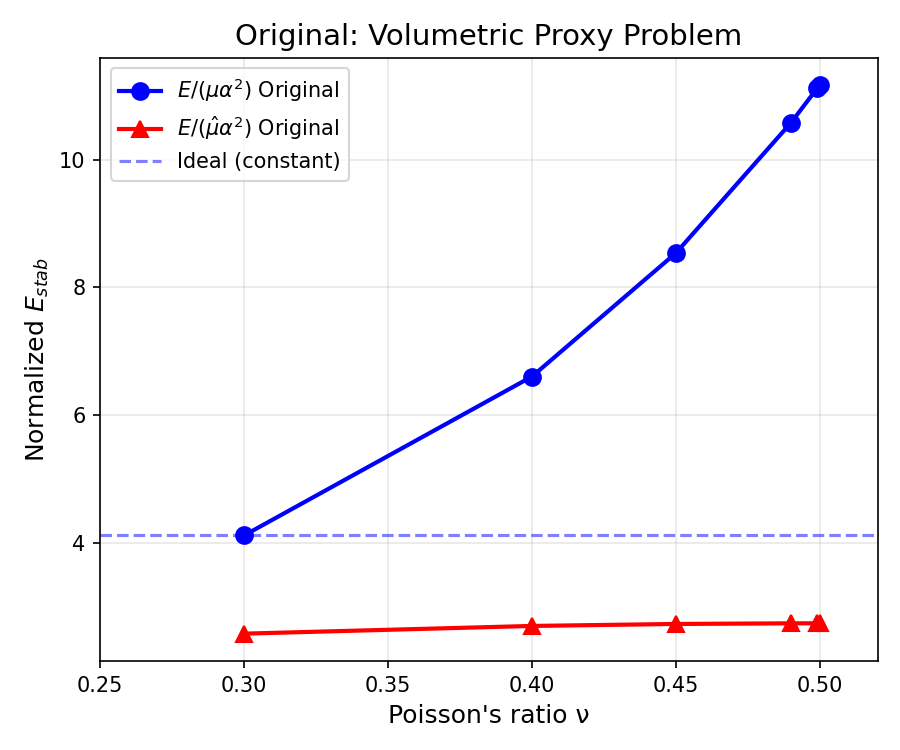}
    \caption{Single-element isochoric kernel-mode diagnostic for the classical stabilization: normalized stabilization energy versus $\nu$. The normalization by the physical shear scale, $U_{h,E}^s/(\mu\,\alpha^2)$, increases strongly as $\nu\to 1/2$, indicating bulk-driven stiffening of an isochoric kernel mode. The same energy normalized by the effective parameter used internally by the classical stabilization, $U_{h,E}^s/(\hat{\mu}\,\alpha^2)$, remains approximately constant, confirming that the classical stabilization scales this isochoric mode with $\hat{\mu}$ rather than with $\mu$.}
    \label{fig:vol_proxy_original}
\end{figure}

\begin{figure}[H]
    \centering
    \includegraphics[width=0.85\textwidth]{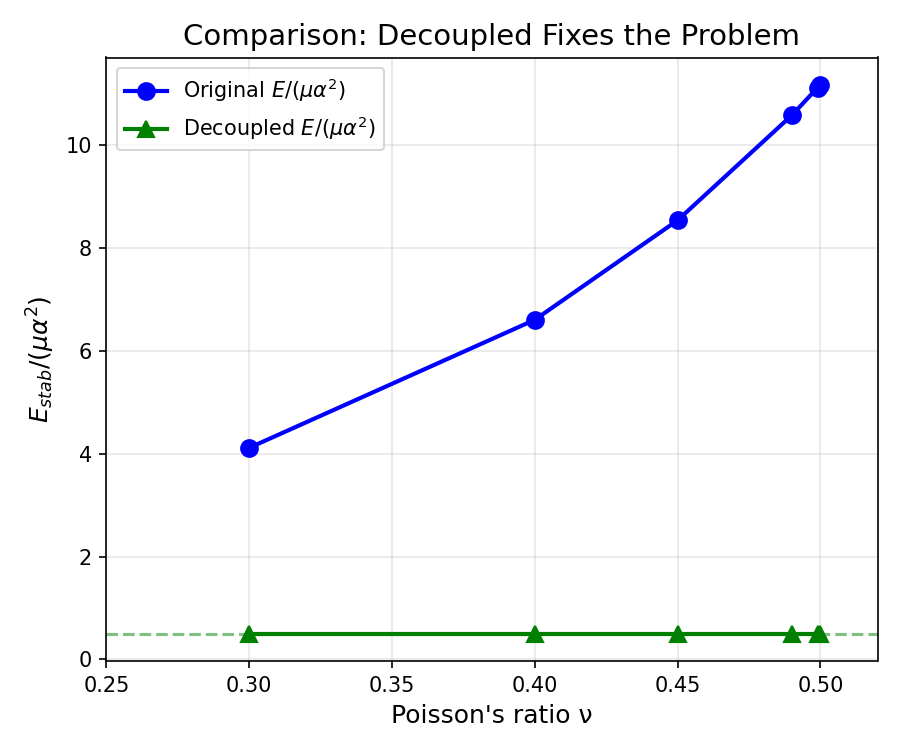}
    \caption{Comparison of classical and decoupled stabilization on an isochoric kernel mode under the physically relevant normalization $U_{h,E}^s/(\mu\,\alpha^2)$. The decoupled stabilization remains essentially constant over the full range of $\nu$, consistent with a shear-scaled penalty on isochoric missing modes. The classical stabilization grows substantially as $\nu\to 1/2$, providing element-level evidence of volumetric-proxy leakage into the shear-type stabilization channel.}
    \label{fig:vol_proxy_comparison}
\end{figure}

\subsection{Modal analysis}\label{sec:modal_analysis}

A single-element modal (spectral) analysis is employed to assess the proposed decoupled stabilization operator at the element level, independently of global boundary-value effects. Consider a polygonal element $E$ with $N_E$ vertices and $2N_E$ displacement degrees of freedom collected in the vector $\mathbf{u}\in\mathbb{R}^{2N_E}$. Let $\mathbf{P}_E\in\mathbb{R}^{2N_E\times 2N_E}$ denote the matrix representation (in the vertex basis) of the projector $\projector$ introduced in (\ref{eq:projection}), i.e., $\mathbf{P}_E\mathbf{u}$ corresponds to the vertex values of the projected polynomial field in $\left[\mathbb{P}_1(E)\right]^2$. The kernel component is measured by the residual operator
\begin{equation}
    \mathbf{r}(\mathbf{u})=(\mathbf{I}-\mathbf{P}_E)\mathbf{u},
\end{equation}
which is the discrete counterpart of (\ref{eq:stacked_residual_kernel}) and vanishes on $\left[\mathbb{P}_1(E)\right]^2$.

The stabilization operator is assembled in matrix form from the deviatoric and volumetric bilinear forms defined in (\ref{eq:deviatoric_stab_matrix_form}) and (\ref{eq:volumetric_stab_matrix_form}). The deviatoric contribution embeds the principal-frame weighting $\mathbf{Q}_E$ and directional factors $(\tau_1,\tau_2)$ driven by the aspect indicator $r_E$ and the anisotropy map $g$ in (\ref{eq:aspect_indicator})--(\ref{eq:anisotropy_map}), whereas the volumetric contribution couples neighbouring residual components through edge normals to penalize spurious normal expansion/compression in the kernel component. Denoting by $\mathbf{S}_E=\mathbf{S}_{E}^{\mathrm{dev}}+\mathbf{S}_{E}^{\mathrm{vol}}$ the corresponding stabilization matrix, missing-mode behavior is isolated by restricting $\mathbf{S}_E$ to $\ker(\mathbf{P}_E)$: a basis matrix $\mathbf{Z}$ for $\ker(\mathbf{P}_E)$ is computed and the reduced operator $\mathbf{Z}^{\mathsf{T}}\mathbf{S}_E\mathbf{Z}$ is formed. The eigendecomposition of this kernel-restricted matrix provides stabilization modes (eigenvectors) and their associated stiffnesses (eigenvalues) on $\missingkernel$.

The resulting spectra are consistent with the design goals. For the unit square ($2N_E=8$), exactly six eigenvalues are numerically zero, matching $\dim(\left[\mathbb{P}_1(E)\right]^2)=6$ and confirming that polynomial fields lie in the nullspace of the stabilization. The kernel dimension is therefore two, and the two kernel eigenvalues coincide (approximately $76.923$), as expected from square symmetry. When rectangles are used to vary the geometric aspect ratio $r_E$, the kernel eigenvalue ratio $\lambda_{\max}/\lambda_{\min}$ tracks the designed directional ratio $\tau_1/\tau_2=g(r_E)^2$ essentially exactly, demonstrating that the principal-frame weighting redistributes missing-mode stiffness in the intended anisotropic manner (Figure~\ref{fig:modal_anisotropy}). On a hexagon (for which $\dim\ker(\mathbf{P}_E)=6$), varying the bulk scale $\kappa_E$ yields a clean deviatoric/volumetric separation: a subset of kernel eigenvalues remains constant (purely deviatoric modes), while the remaining modes increase with $\kappa_E$ (modes with volumetric content), indicating the absence of bulk-driven contamination of the deviatoric channel (Figure~\ref{fig:modal_dev_vol}). In a near-incompressible sweep (varying $\nu$ up to $0.4999$), the kernel eigenvalues normalised by $\mu$ remain flat, supporting the claim that no $\lambda$-dependent proxy inflates the deviatoric stabilization. Conditioning trends are also consistent: regular polygons exhibit near-unit kernel condition numbers, while rectangles show predictable growth proportional to the imposed anisotropy (approximately $r_E^2$), which can be controlled by the cap $g_{\max}$ in (\ref{eq:anisotropy_map}) (Figure~\ref{fig:modal_conditioning}).

\begin{figure}[H]
    \centering
    \includegraphics[width=0.85\textwidth]{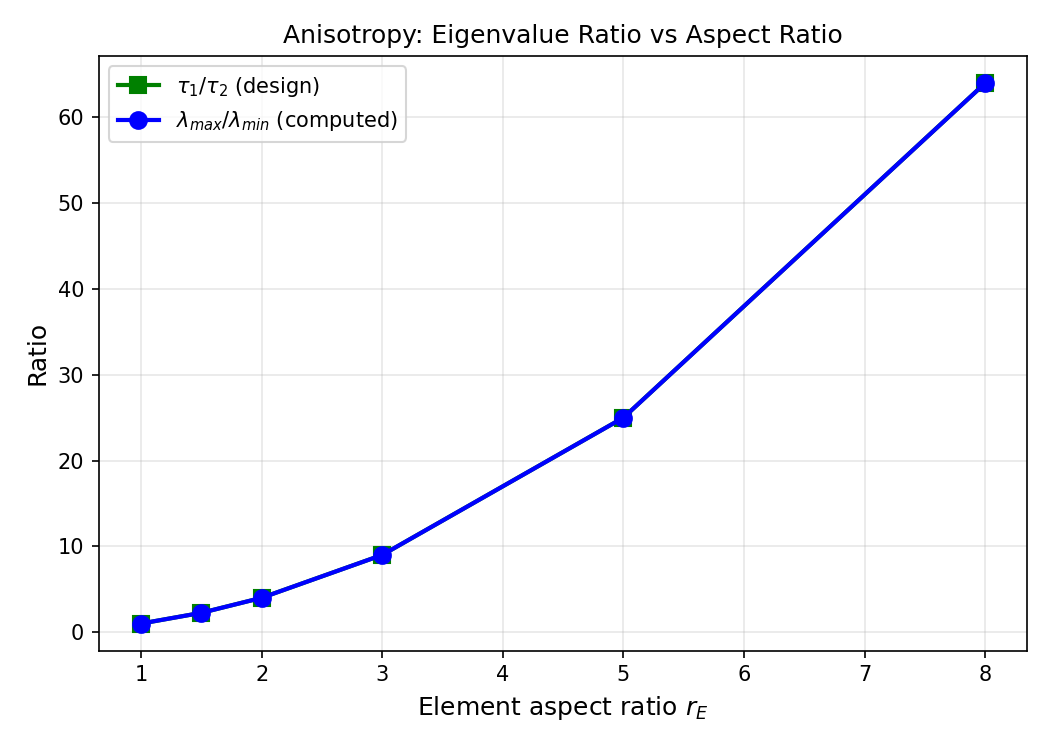}
    \caption{Anisotropic scaling of the kernel-only deviatoric stabilization. The designed directional ratio $\tau_1/\tau_2=g(r_E)^2$ is compared against the computed kernel eigenvalue ratio $\lambda_{\max}/\lambda_{\min}$ as the element aspect indicator $r_E$ varies.}
    \label{fig:modal_anisotropy}
\end{figure}

\begin{figure}[H]
    \centering
    \includegraphics[width=\textwidth]{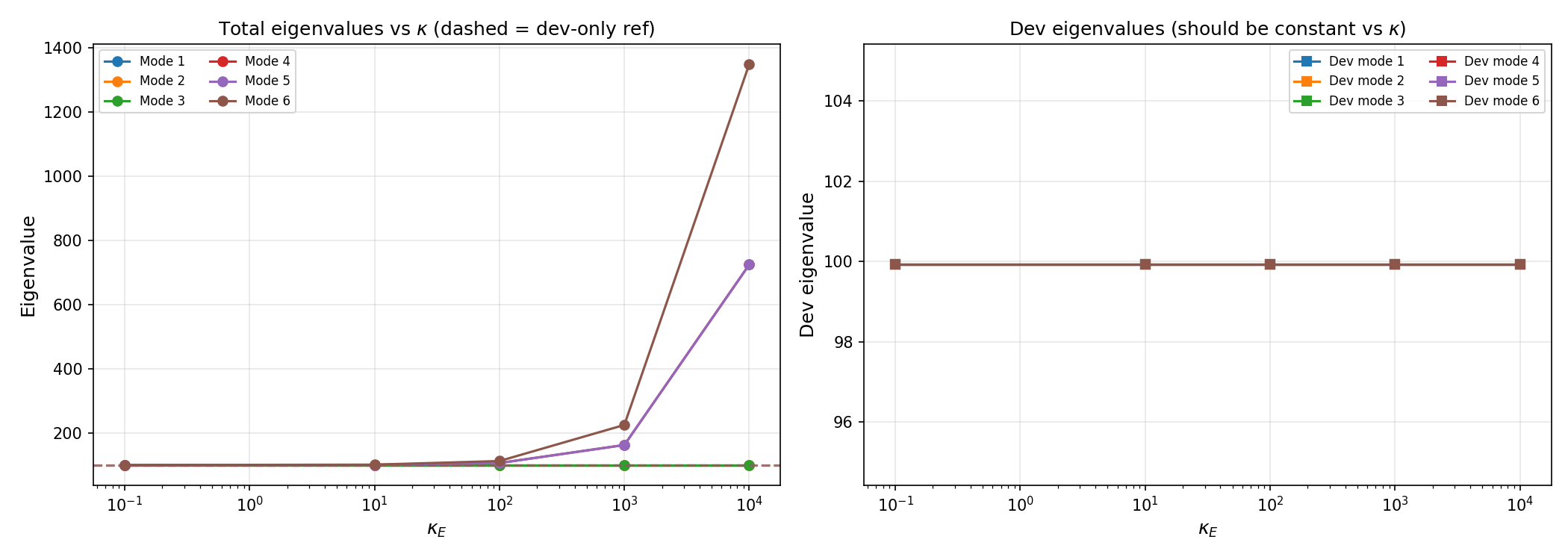}
    \caption{Deviatoric/volumetric separation on $\missingkernel$ under the decoupled stabilization. Kernel-restricted eigenvalues are shown as the bulk scale $\kappa_E$ is varied: a subset of modes remains constant (deviatoric content), whereas modes with volumetric content increase with $\kappa_E$.}
    \label{fig:modal_dev_vol}
\end{figure}

\begin{figure}[H]
    \centering
    \includegraphics[width=0.9\textwidth]{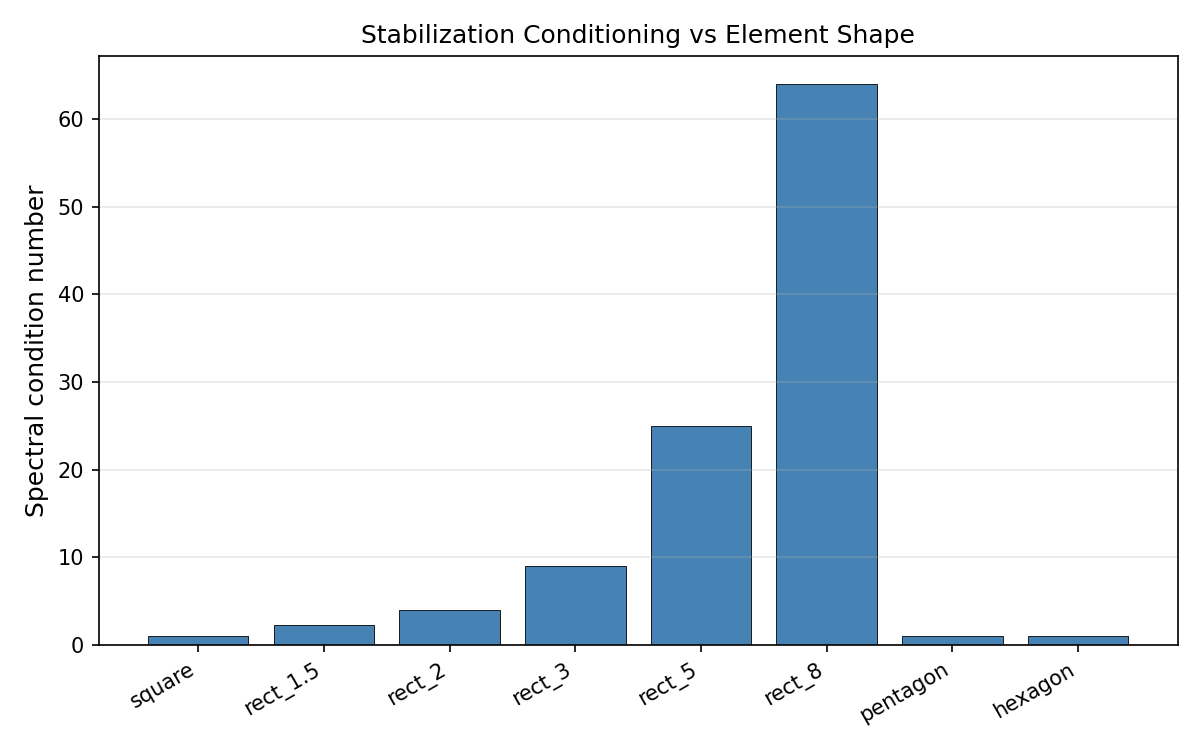}
    \caption{Conditioning of the kernel-restricted stabilization operator across element shapes. The spectral condition number of the reduced matrix $\mathbf{Z}^{\mathsf{T}}\mathbf{S}_E\mathbf{Z}$ increases predictably with rectangular anisotropy, while remaining near unity for regular polygons.}
    \label{fig:modal_conditioning}
\end{figure}

This element-level modal inspection is a concrete instance of the spectral viewpoint developed in Section~\ref{sec:spectral_analysis}. There, stability on $\missingkernel$ is formulated as the Rayleigh-quotient bound (\ref{eq:reduced_stability_criteria_eigenproblem}) induced by the inequality (\ref{eq:stability_criteria_eigenproblem}), equivalently as bounds on generalized eigenvalues of the pair $(\mathbf{K}^s,\mathbf{K}^*)$ restricted to $\missingkernel$. In the present setting, the reduced operator $\mathbf{Z}^{\mathsf{T}}\mathbf{S}_E\mathbf{Z}$ provides a basis-independent spectrum of the stabilization energy assigned to each missing-mode pattern, thereby exposing directly whether the kernel-only property is satisfied and how geometric anisotropy and the dev/vol split redistribute stiffness across modes. In this sense, the modal analysis complements the theoretical spectral criteria by making the mode-by-mode scaling of the stabilization explicit at the element level.

\subsection{Cook's membrane}

Cook's membrane is a standard benchmark for finite element formulations under large deformation and shear-dominated loading. The problem consists of a tapered cantilever panel subjected to a uniform shear traction on its free edge. The geometry, illustrated in Figure~\ref{fig:cooks_membrane_geometry}, is defined by the four corner points (in consistent length units) as in \cite{Wriggers2017}:
\begin{center}
\begin{tabular}{lcc}
\toprule
Corner & $x$ & $y$ \\
\midrule
Bottom-left & 0 & 0 \\
Bottom-right & 48 & 44 \\
Top-right & 48 & 60 \\
Top-left & 0 & 44 \\
\bottomrule
\end{tabular}
\end{center}
The panel width is 48, and the height varies from 44 on the left edge to 16 on the right edge, producing a skewed quadrilateral with non-parallel edges.

\begin{figure}[H]
    \centering
    \includegraphics[width=0.7\textwidth]{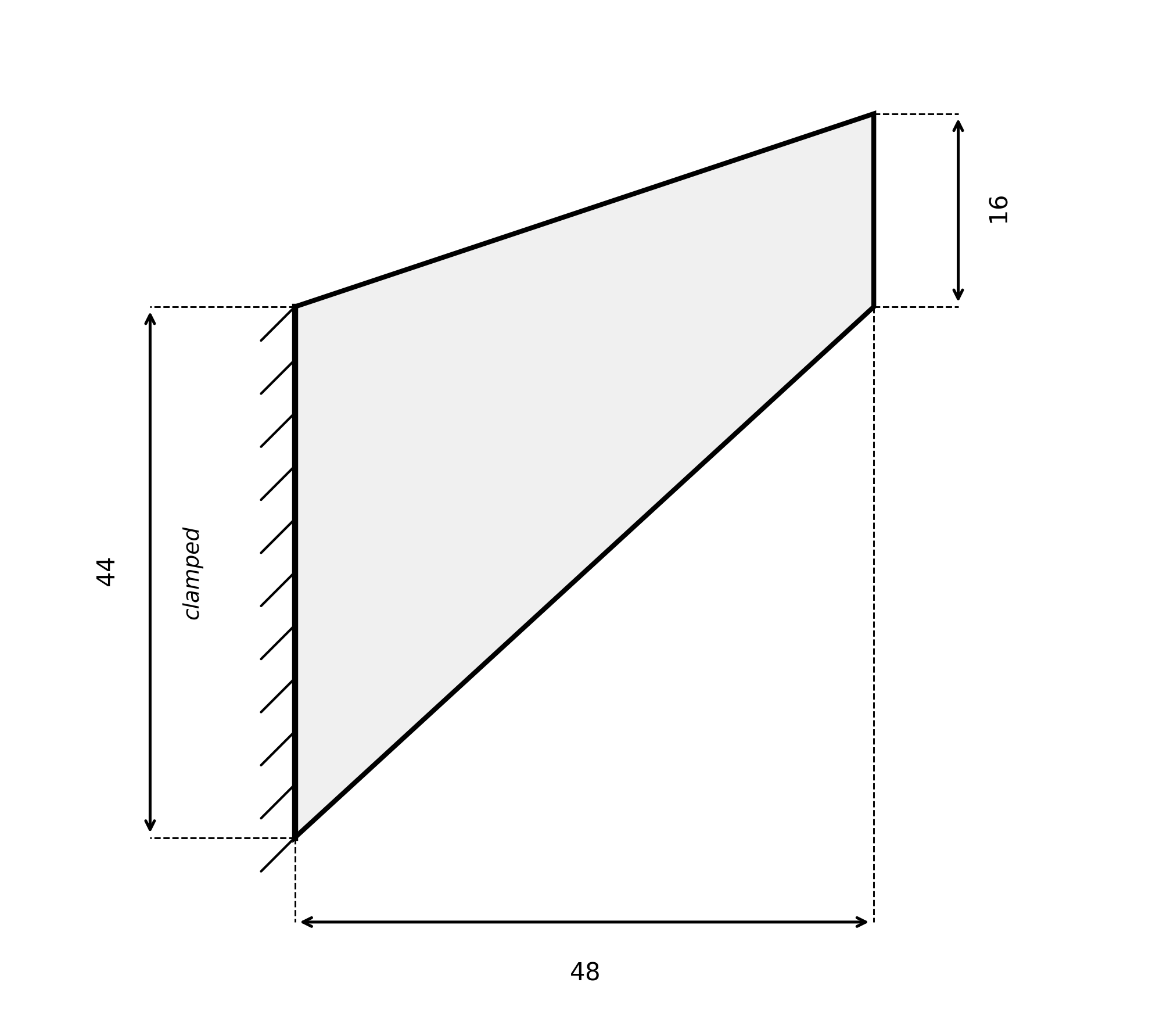}
    \caption{Geometry and boundary conditions for Cook's membrane benchmark.}
    \label{fig:cooks_membrane_geometry}
\end{figure}

The left edge ($x=0$, from $(0,0)$ to $(0,44)$) is fully clamped, i.e., $u_x = u_y = 0$ at every node. The right edge ($x=48$, from $(48,44)$ to $(48,60)$) is subjected to a constant distributed vertical traction $t_y = q_0 = 4$, yielding a total applied vertical force $F_y = q_0 \times 16 = 64$. The top and bottom edges are traction-free. The right-edge traction is integrated numerically using the midpoint rule along each boundary edge segment. The quantity of interest is the vertical displacement $u_y$ at the upper-right corner $(48,60)$.

The material is assumed to obey the compressible neo-Hookean constitutive law introduced in (\ref{eq:neo_hookean_energy}), under plane-strain conditions. The Lamé constants are set to $\mu = 40$ and $\lambda = 100$, and the distributed load is $q_0 = 4$.

Four distinct mesh families are tested, each at four refinement levels: a regular quadrilateral structured grid (\texttt{quad}), a moderately distorted quadrilateral mesh with smooth perturbations (\texttt{dist1}), a highly distorted quadrilateral mesh obtained via higher-order polygonal subdivision (\texttt{dist2}), and irregular polygonal cells generated via Voronoi tessellation (\texttt{voronoi}). Representative meshes at refinement level $h = 0.0625$ are illustrated in Figure~\ref{fig:mesh_families}. The normalized mesh parameter $h$ relates to the actual element size as $h_{\text{actual}} = h \times 48$; the refinement levels considered are summarized in Table~\ref{tab:cooks_refinement_levels}.

\begin{table}[H]
\centering
\caption{Refinement levels considered for the Cook's membrane benchmark. The normalized mesh parameter $h$ corresponds to an actual element size $h_{\text{actual}} = h \times 48$. The number of elements is reported separately for the structured/distorted quadrilateral families (\texttt{quad}, \texttt{dist1}, \texttt{dist2}) and the Voronoi family (\texttt{voronoi}).}
\label{tab:cooks_refinement_levels}
\begin{tabular}{cccc}
\toprule
$h$ (normalized) & $h_{\text{actual}}$ & Quad/dist elements & Voronoi elements \\
\midrule
0.5000 & 24.0 & 4 & 10 \\
0.2500 & 12.0 & 16 & 10 \\
0.1250 & 6.0 & 64 & 40 \\
0.0625 & 3.0 & 256 & 160 \\
\bottomrule
\end{tabular}
\end{table}

The total load is applied incrementally in 10 equal steps (load factors $0.1, 0.2, \ldots, 1.0$). At each step, the nonlinear equilibrium equations are solved via the Newton--Raphson method described in (\ref{eq:newton_linear_system}), with a maximum of 50 iterations per step, residual tolerance $\|\mathbf{R}\| < 10^{-8}$, and increment tolerance $\|\Delta\mathbf{u}\| < 10^{-10}$. Typical convergence is achieved in 4--7 Newton iterations per load step.

Two stabilization strategies are compared across all mesh types and refinement levels. The first is the classical formulation of \cite{Wriggers2017} and \cite{vanHuyssteen2020}, detailed in Section \ref{sec:classic_stab}, which evaluates a surrogate neo-Hookean energy over an internal fan triangulation with effective Lamé parameters $\hat{\mu}$ and $\hat{\lambda}$ given by (\ref{eq:mu_hat_classical}) and (\ref{eq:lambda_hat_classical}). The second is the proposed kernel-only deviatoric/volumetric decoupled stabilization introduced in Section \ref{sec:new_stab}, which decouples the shear and bulk stabilization channels and scales the deviatoric contribution exclusively by $\mu$, eliminating the dependence on $\lambda$ identified as the source of the volumetric proxy problem.

\begin{figure}[H]
    \centering
    \begin{subfigure}[b]{0.48\textwidth}
        \centering
        \includegraphics[width=\textwidth]{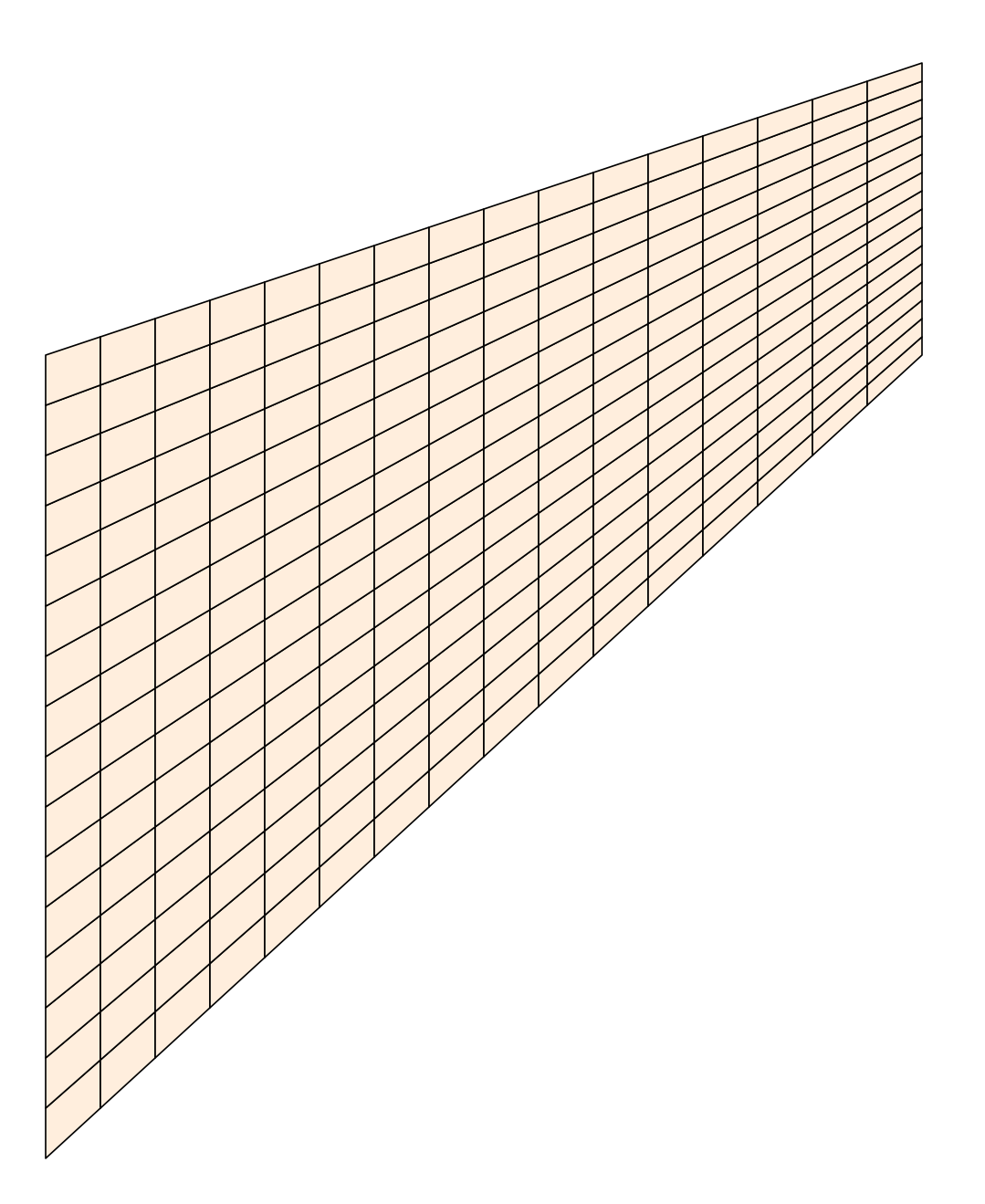}
        \caption{\texttt{Regular quadrilateral}}
    \end{subfigure}
    \hfill
    \begin{subfigure}[b]{0.48\textwidth}
        \centering
        \includegraphics[width=\textwidth]{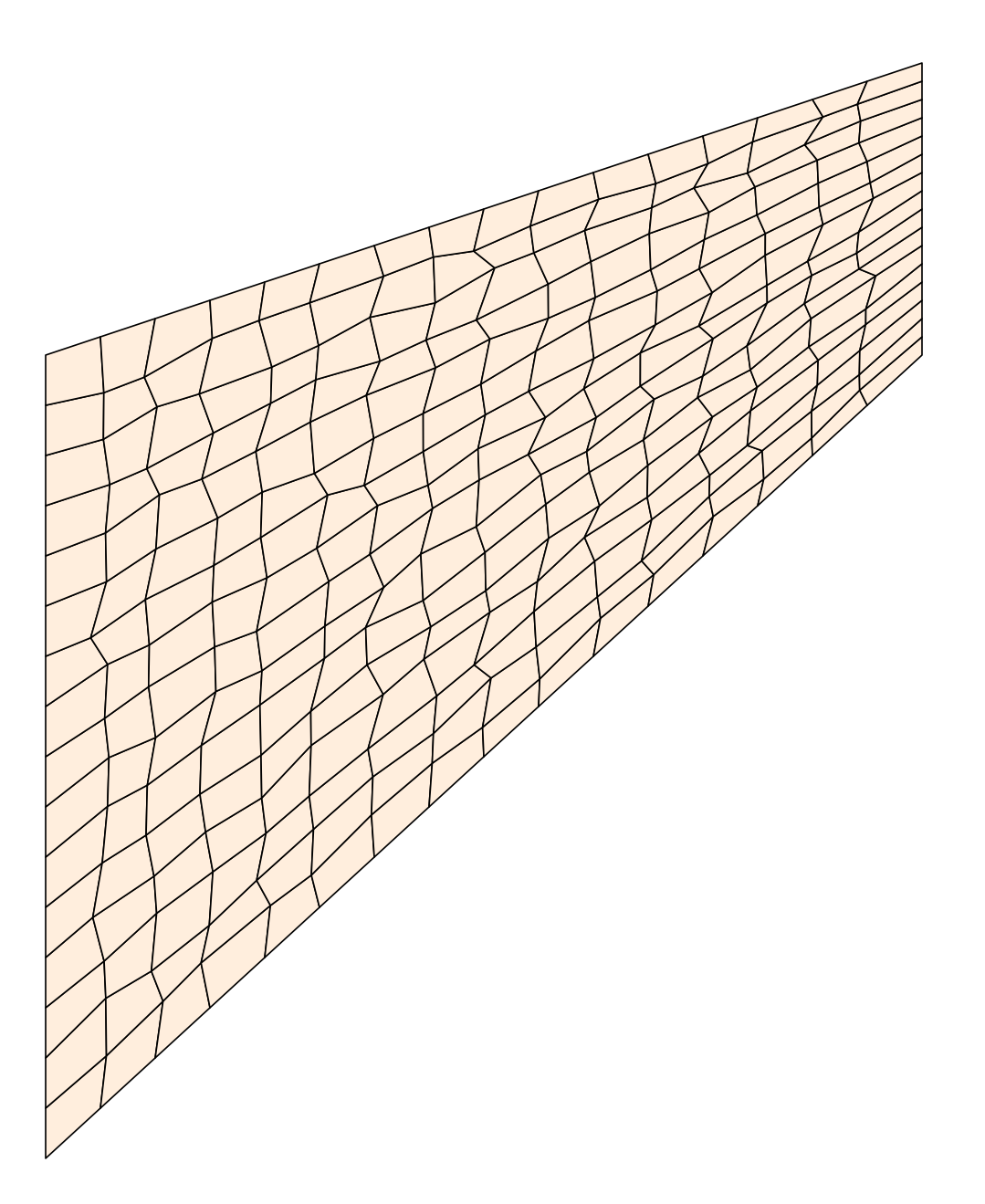}
        \caption{\texttt{Distorted quadrilateral}}
    \end{subfigure}
    \vskip\baselineskip
    \begin{subfigure}[b]{0.48\textwidth}
        \centering
        \includegraphics[width=\textwidth]{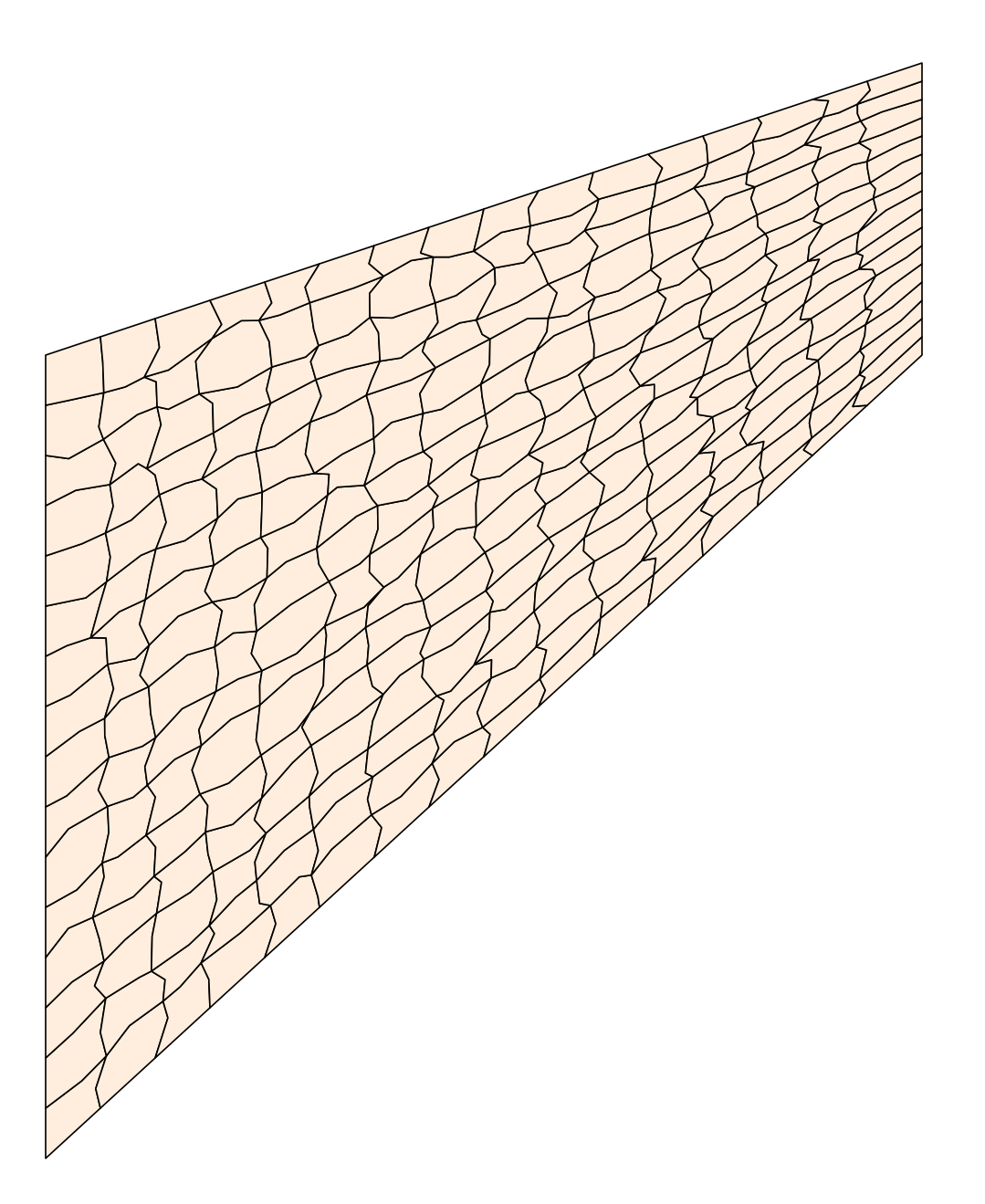}
        \caption{\texttt{Highly distorted quadrilateral}}
    \end{subfigure}
    \hfill
    \begin{subfigure}[b]{0.48\textwidth}
        \centering
        \includegraphics[width=\textwidth]{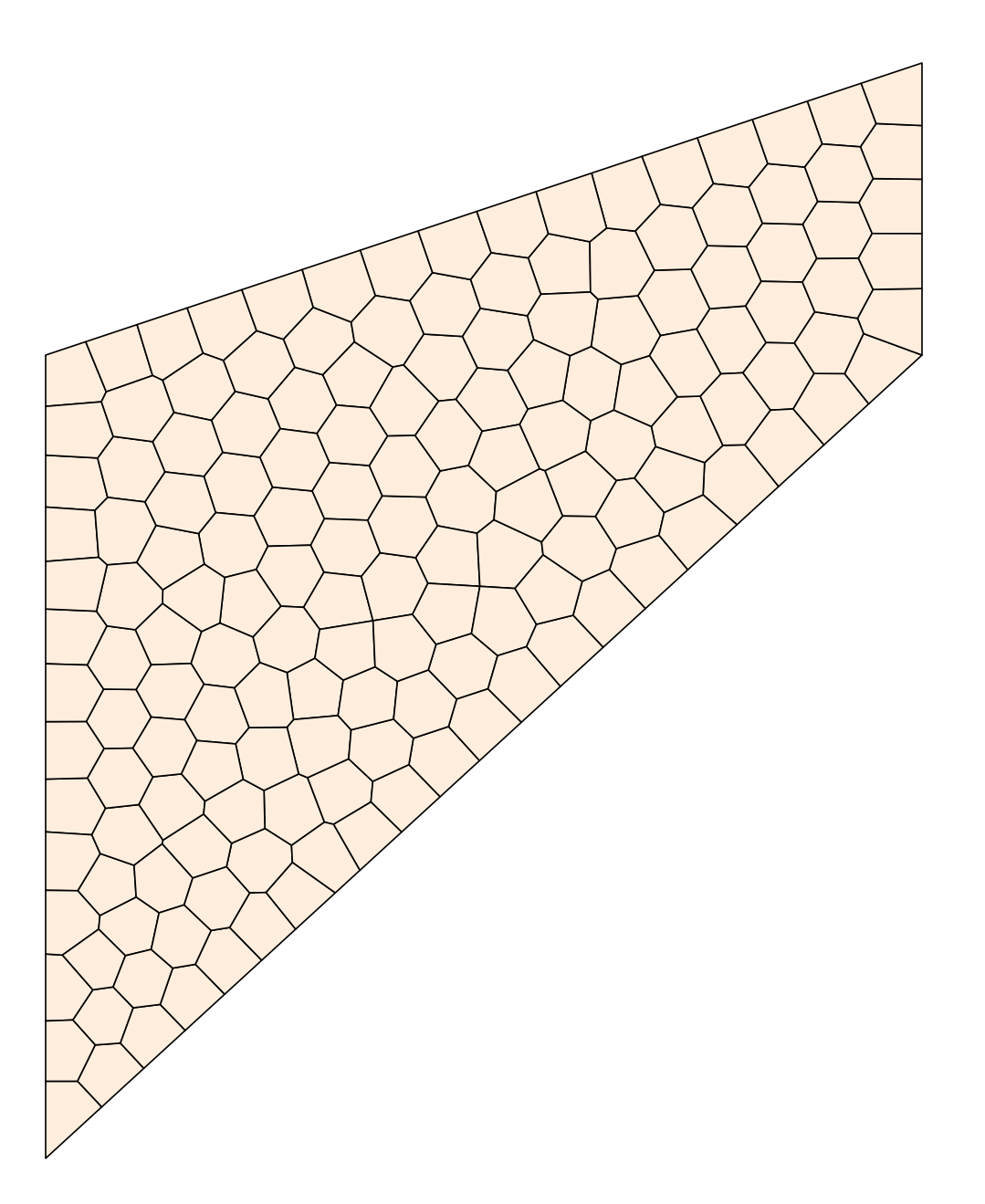}
        \caption{\texttt{Voronoi tessellation}}
    \end{subfigure}
    \caption{Representative meshes for Cook's membrane at refinement level $h = 0.0625$: (a) regular quadrilateral, (b) distorted quadrilateral (\texttt{dist1}), (c) highly distorted quadrilateral (\texttt{dist2}), and (d) Voronoi tessellation.}
    \label{fig:mesh_families}
\end{figure}

For $\nu=0.499$, the Cook's membrane results show a clear separation between a locking-free response under the decoupled term and a locking-prone response under the classic term, consistently across the four mesh families (\texttt{Regular quadrilateral}, \texttt{Distorted quadrilateral}, \texttt{Highly distorted quadrilateral}, and \texttt{Voronoi tessellation}). In all cases both formulations complete the prescribed loading history, so that the distinguishing feature is not solver robustness but the asymptotic load--displacement response and its mesh dependence in the nearly incompressible regime. Figures~\ref{fig:quad_comparison}, \ref{fig:dist1_comparison}, \ref{fig:dist2_comparison}, and~\ref{fig:voronoi_comparison} display the deformed state for both stabilization terms on each mesh family and for different mesh sizes.

Under the decoupled term, the tip displacement $u_y$ approaches a narrow band (approximately $8.5$) as the mesh is refined, indicating a stable nearly incompressible limit with mild mesh sensitivity. The structured families (\texttt{Regular quadrilateral} and \texttt{Distorted quadrilateral}) approach this value from below ($8.01 \to 8.48$ and $7.79 \to 8.51$, respectively), while the Voronoi family exhibits a small non-monotone transient on the coarsest meshes ($9.02 \to 8.33$) before settling and converging to the same fine-mesh value ($8.52 \to 8.53$). The highly distorted family (\texttt{Highly distorted quadrilateral}) approaches the same range from above ($10.67 \to 8.99$), which is consistent with the fact that strong geometric distortion can produce a comparatively compliant coarse response that is reduced as resolution increases and the deformation field is better represented. Overall, the decoupled term yields a coherent refinement picture: despite markedly different element geometries, the fine meshes across families cluster around a common response, with the expected coarse-mesh variability attributable to geometric irregularity and low-resolution tessellation effects.

In contrast, the classic term exhibits the characteristic signature of volumetric locking at $\nu=0.499$: the predicted tip displacement is severely reduced on coarse meshes and converges only slowly toward the decoupled response as the mesh is refined. For \texttt{Regular quadrilateral} and \texttt{Distorted quadrilateral}, the classic values increase from $u_y \approx 3$ on the coarsest mesh to $u_y \approx 7.3$--$7.5$ on the finest mesh, remaining substantially below the corresponding decoupled values ($\approx 8.5$). The Voronoi family follows the same trend ($5.11 \to 7.83$), with a larger coarse-mesh displacement than \texttt{Regular quadrilateral}/\texttt{Distorted quadrilateral} but still a persistent gap on refinement. The \texttt{Highly distorted quadrilateral} family shows the least severe discrepancy at the finest level ($8.51$ versus $8.99$), yet still displays the same slow approach from a markedly underpredicted coarse response ($3.66 \to 8.51$). These patterns indicate that, for the classic term, the nearly incompressible response is controlled by an artificial stiffness component that is strongly mesh-dependent and only gradually mitigated by refinement---precisely the behavior associated with locking.

The mechanism is consistent with the scaling embedded in the classic term, where a volumetric proxy is permitted to enter the deviatoric stabilization channel via an incompressibility factor $\alpha$ and an inflated effective shear modulus $\hat{\mu} = (1+\alpha)^2 \Phi \mu$ defined in (\ref{eq:mu_hat_classical}). In the nearly incompressible regime, $\lambda$ becomes very large relative to $\mu$, so any construction of the type $\alpha = \alpha(\lambda, \mu)$ yields a large multiplicative amplification in $\hat{\mu}$. As a consequence, the stabilization does not merely control the projector-kernel modes at a shear-consistent energy level; rather, it injects volumetric-driven stiffness into the shear-like kernel penalty, over-constraining the unresolved modes and suppressing deformation. The observed mesh dependence follows naturally: when the stabilization dominates, the computed response becomes sensitive to geometric irregularity and to how the missing-mode content manifests on each mesh family, producing the wide spread on coarse meshes and the slow, locking-controlled convergence toward the limiting response.

\begin{figure}[H]
    \centering
    \includegraphics[width=\textwidth]{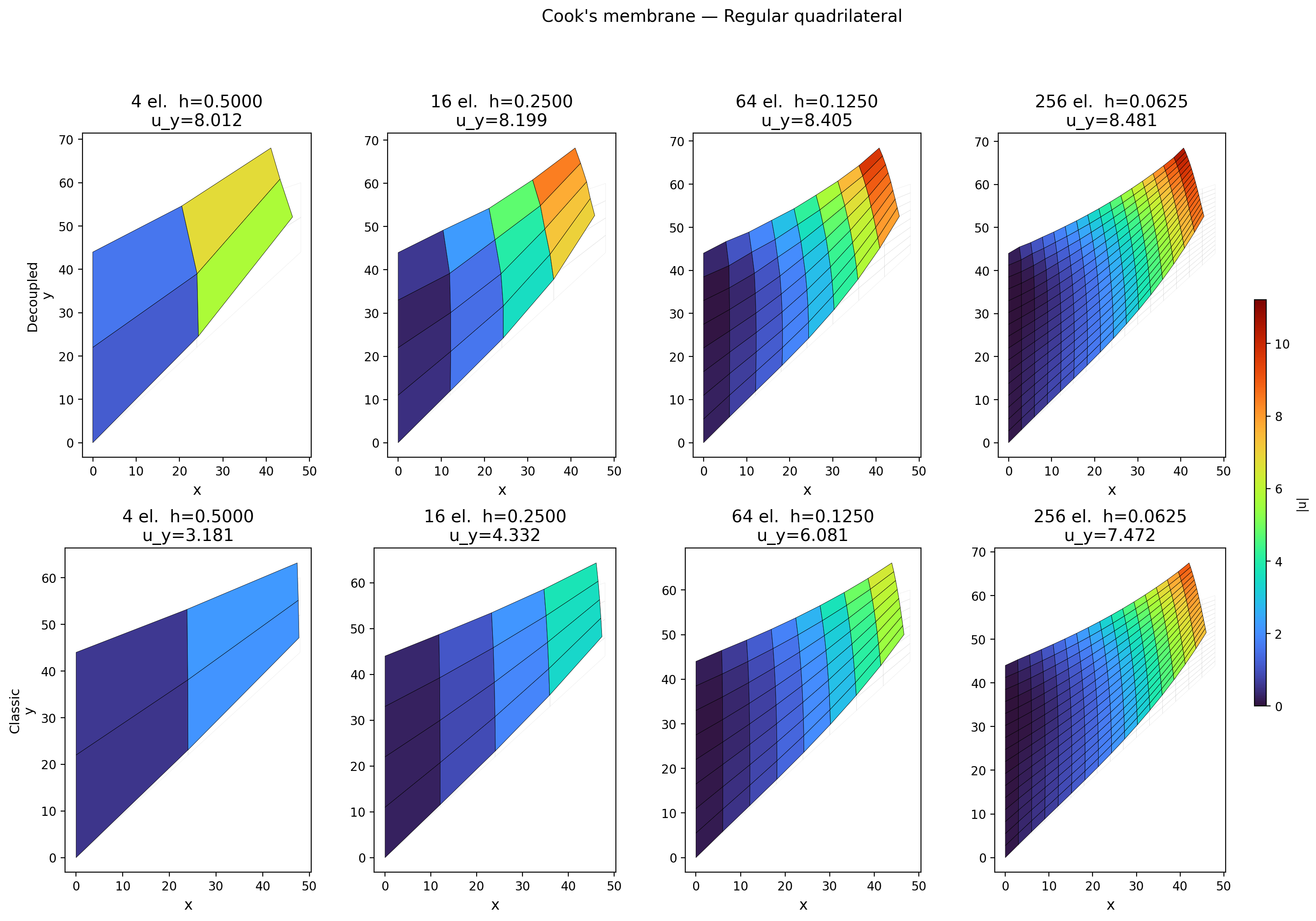}
    \caption{Cook's membrane on the \texttt{Regular quadrilateral} mesh family and nearly-incompressible regime ($\nu = 0.499$). Each column corresponds to a refinement level ($h = 0.5000,\, 0.2500,\, 0.1250,\, 0.0625$); the top row shows results with the proposed \emph{decoupled term} and the bottom row with the \emph{classic term}. The color field represents the displacement magnitude $|\mathbf{u}|$, and the tip displacement $u_y$ at the upper-right corner is reported above each panel. For this mesh family, the decoupled term exhibits a smooth increase toward the fine-mesh response ($u_y: 8.012 \to 8.481$), whereas the classic term yields markedly lower values on coarse meshes and approaches the decoupled response only gradually under refinement ($u_y: 3.181 \to 7.472$), consistent with pronounced locking effects in the nearly incompressible regime.}
    \label{fig:quad_comparison}
\end{figure}

\begin{figure}[H]
    \centering
    \includegraphics[width=\textwidth]{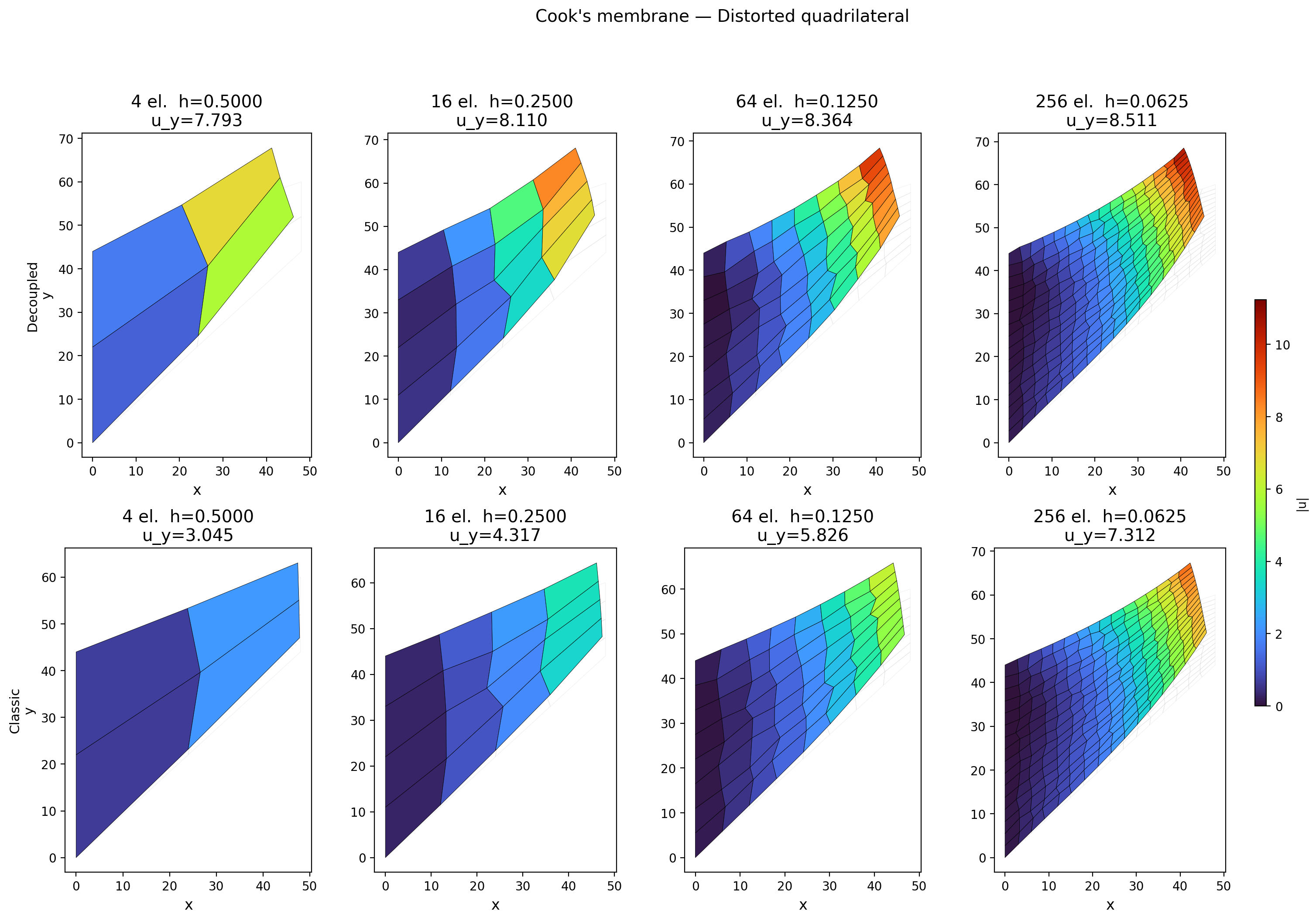}
    \caption{Cook's membrane on the \texttt{Distorted quadrilateral} mesh family and nearly-incompressible regime ($\nu = 0.499$). The decoupled term approaches a mesh-independent tip displacement under refinement ($u_y: 7.793 \to 8.511$). The classic term remains significantly lower on coarse meshes and converges slowly toward the decoupled response ($u_y: 3.045 \to 7.312$), indicating that moderate geometric distortion does not remove the locking footprint of the classic stabilization at $\nu=0.499$.}
    \label{fig:dist1_comparison}
\end{figure}

\begin{figure}[H]
    \centering
    \includegraphics[width=\textwidth]{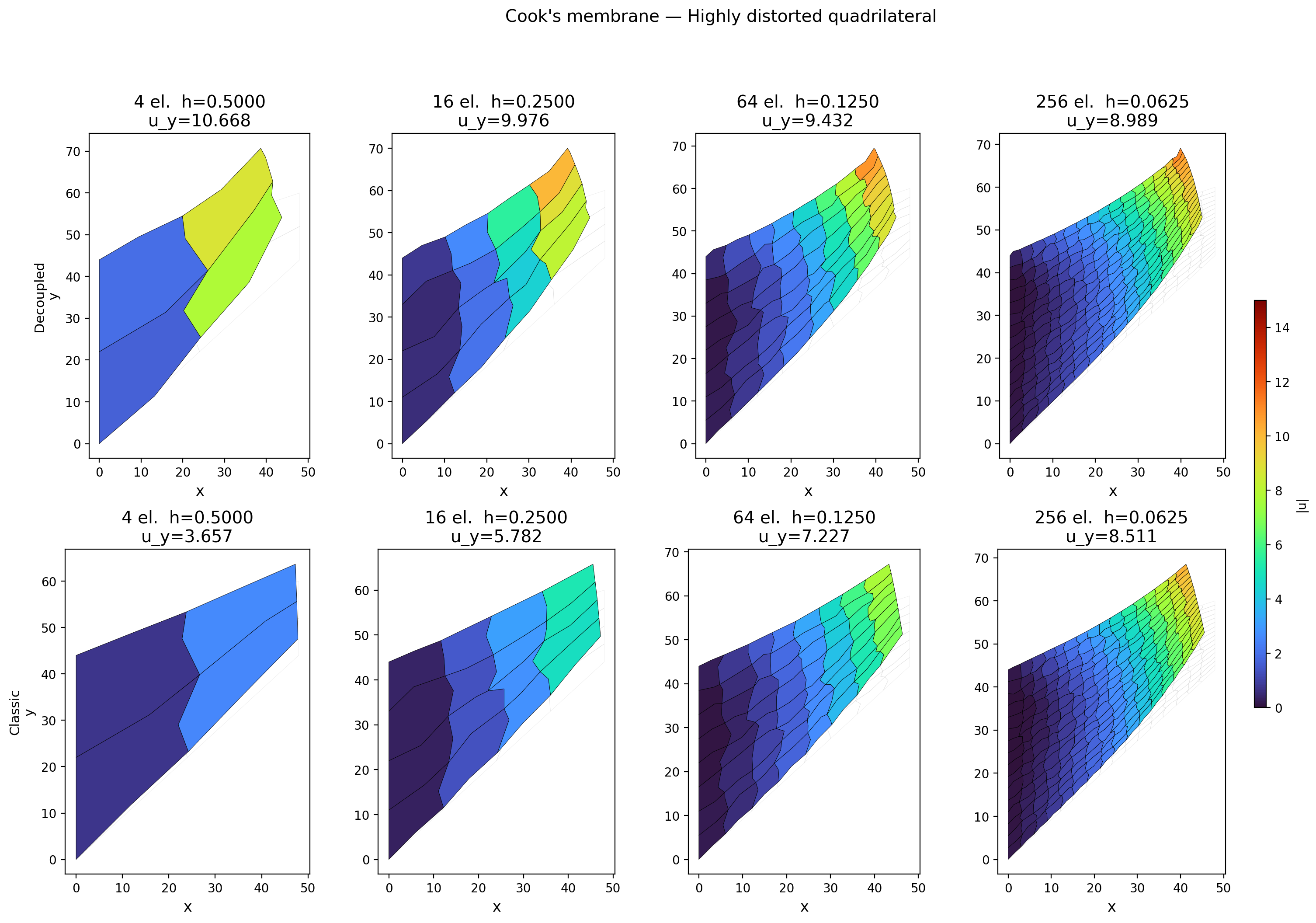}
    \caption{Cook's membrane on the \texttt{Highly distorted quadrilateral} mesh family and nearly-incompressible regime ($\nu = 0.499$). Due to the strong geometric distortion, the decoupled term exhibits an elevated coarse-mesh response that decreases with refinement and approaches the fine-mesh range ($u_y: 10.668 \to 8.989$). The classic term again underpredicts the tip displacement on coarse meshes and recovers only gradually under refinement ($u_y: 3.657 \to 8.511$), showing that severe distortion does not eliminate the locking-dominated behavior of the classic stabilization.}
    \label{fig:dist2_comparison}
\end{figure}

\begin{figure}[H]
    \centering
    \includegraphics[width=\textwidth]{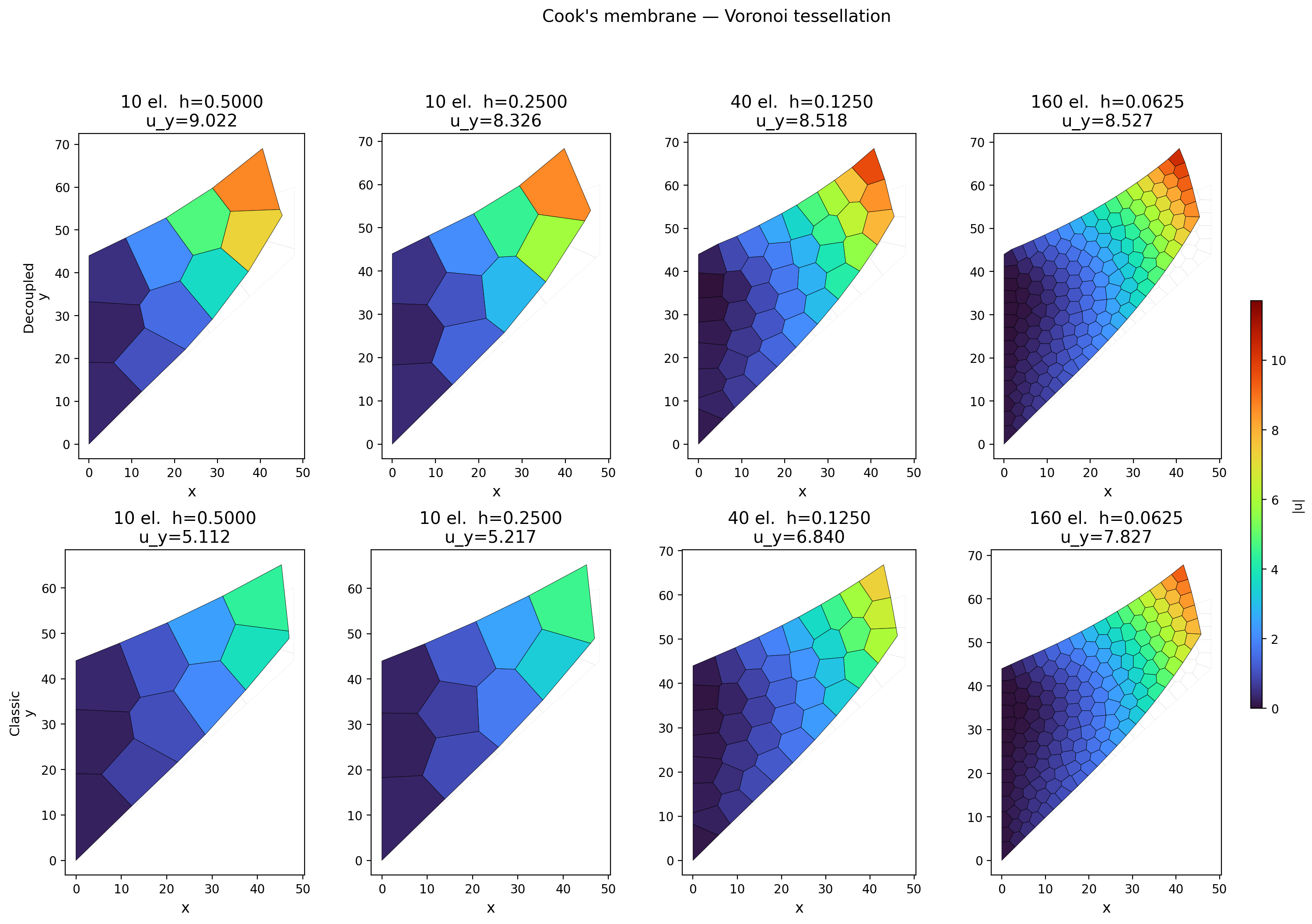}
    \caption{Cook's membrane on the \texttt{Voronoi tessellation} mesh family and nearly-incompressible regime ($\nu = 0.499$). The decoupled term exhibits a non-monotone transient on the coarsest levels ($u_y: 9.022 \to 8.326$) but converges to a consistent fine-mesh tip displacement ($u_y: 8.518 \to 8.527$). The classic term also completes the loading history on this mesh family, yet yields substantially smaller displacements on coarse meshes and remains below the decoupled response at comparable refinements ($u_y: 5.112 \to 7.827$), indicating a persistent locking-induced stiffening that is relieved only slowly by mesh refinement.}
    \label{fig:voronoi_comparison}
\end{figure}

The tip-displacement convergence plot in Figure~\ref{fig:cooks_convergence} provides a compact refinement summary of the nearly incompressible response at $\nu=0.499$ across all four mesh families. Using the dashed line $u_y^{\mathrm{ref}}=8.481$ as a common reference, the decoupled term exhibits a rapid collapse of the curves toward the reference level with only mild mesh-dependent transients (most notably on the coarsest Voronoi and highly distorted meshes), whereas the classic term displays a pronounced locking footprint: substantial underprediction on coarse meshes and a slow, strongly mesh-dependent recovery as $h$ decreases. The classic curves remain well below the reference over most of the refinement range and approach it only on the finest meshes, while the decoupled curves lie close to the reference already at moderate resolutions. This figure therefore complements the deformed-shape comparisons by isolating the principal effect of stabilization in the nearly incompressible regime---the difference in refinement behavior and asymptotic bias---without conflating it with solver termination or load-step completion.

\begin{figure}[H]
    \centering
    \includegraphics[width=\textwidth]{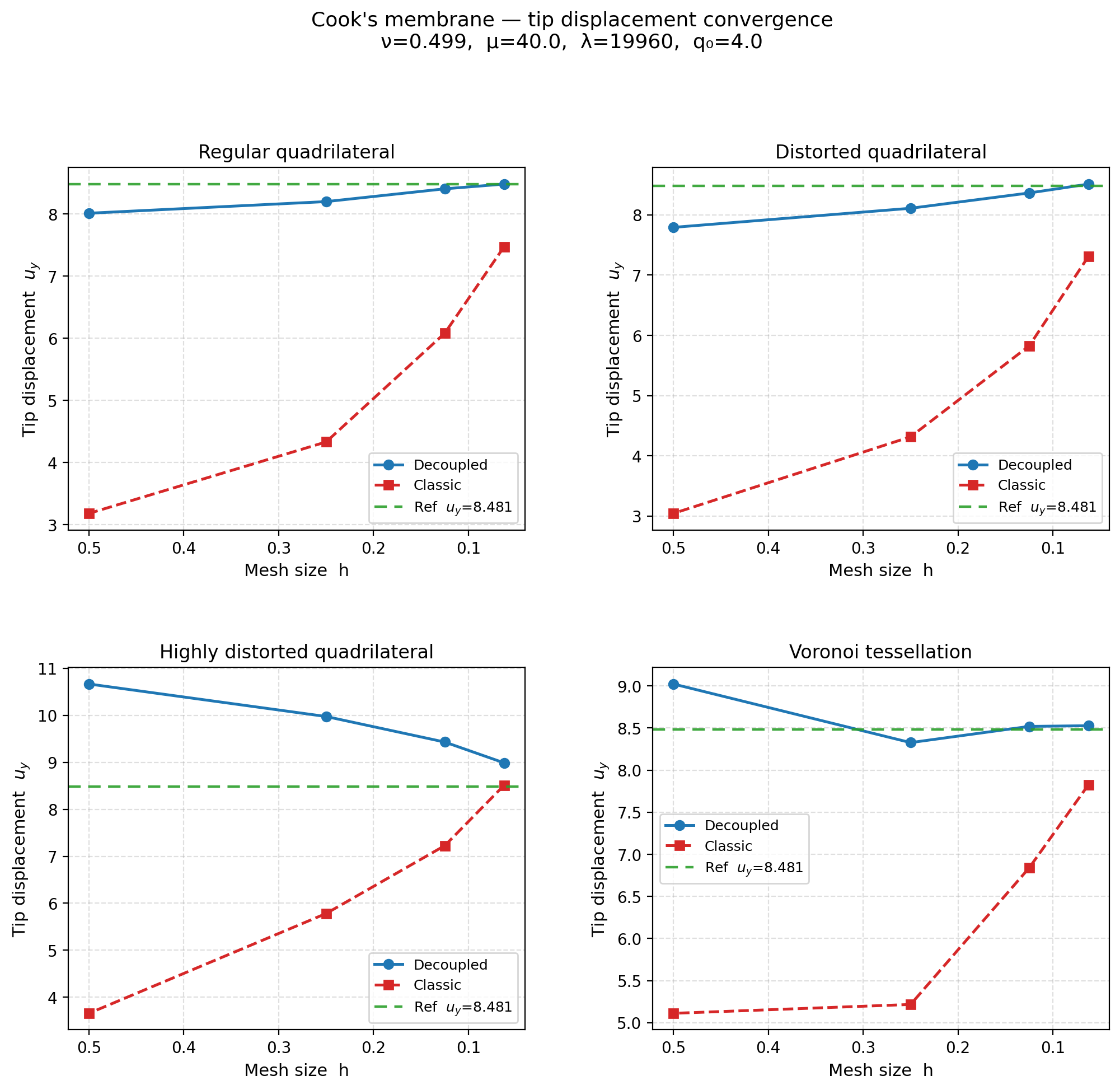}
    \caption{Cook's membrane—tip displacement convergence at \(\nu=0.499\) (\(\mu=40\), \(\lambda=19960\), \(q_0=4\)). Tip displacement \(u_y\) at the upper-right corner is reported versus mesh size \(h\) for the four mesh families. The \emph{decoupled term} (blue) clusters rapidly toward a common fine-mesh response, while the \emph{classic term} (red) underpredicts \(u_y\) on coarse meshes and approaches the limiting value only gradually under refinement, consistent with volumetric locking. The dashed line indicates the reference value \(u_y^{\mathrm{ref}}=8.481\) taken from the finest regular quadrilateral result.}

    \label{fig:cooks_convergence}
\end{figure}

\section{Conclusions}
\label{sec:conclusions}

This work addressed the role of stabilization in finite-strain virtual element methods (VEM), with emphasis on the requirement that the stabilization term act only on unresolved modes and scale like the relevant (state-dependent) tangent energy on $\missingkernel$. In the finite-deformation setting, common stabilization strategies evaluate nonlinear surrogate energies on auxiliary sub-triangulations and introduce effective Lam\'e-type scalings. While these constructions provide practical robustness, they may also couple deviatoric and volumetric effects through bulk-dependent proxies entering shear-type penalties and may retain non-vanishing volumetric penalties on $\missingkernel$ in the nearly incompressible limit. These effects motivate stabilization designs that are explicitly mode-aware and that separate deviatoric and volumetric mechanisms in the stabilization channel.

The first contribution was a spectral framing of stabilization requirements on $\missingkernel$. Spectral equivalence of the stabilization bilinear form and the target tangent energy on $\missingkernel$ was recast as a generalized Rayleigh-quotient bound and, equivalently, as bounds on generalized eigenvalues associated with kernel-restricted operators. This provides a basis-independent diagnostic for whether stabilization assigns appropriate energies to unresolved deformation patterns. In the finite-strain setting, the analysis clarifies why kernel-restricted spectra are the natural objects to monitor: when the consistent energy depends on the displacement only through the projector, its second variation vanishes on $\missingkernel$, so that the stabilization channel is responsible for assigning stiffness to missing modes at the Newton state.

The second contribution was the construction of a kernel-only deviatoric/volumetric decoupled stabilization that is submesh-free and admits closed-form tangents. The proposed stabilization is built from the projector residual evaluated at the degrees of freedom, ensuring kernel consistency by construction. The deviatoric part is scaled purely by a shear measure $\mu_E$ and incorporates a geometry-driven principal-frame weighting with bounded anisotropy factors, thereby enabling directional stiffness assignment without introducing bulk-dependent modifications of the shear scaling. The volumetric part is formulated as a boundary-only penalty on normal residual components with an explicit bulk scale $\kappa_E$, which can be capped or suppressed as $\nu\to 1/2$ to avoid residual compressibility on $\missingkernel$. Under standard polygon regularity assumptions, the main scaling result shows that the deviatoric stabilization is equivalent to $\mu_E|\cdot|_{1,E}^2$ on $\missingkernel$ with constants independent of $h_E$ and of $\nu$, providing the stability foundation required by the canonical VEM yardstick.

The numerical experiments supported these design objectives at both element and problem levels. A single-element isochoric kernel-mode diagnostic demonstrated that a classical surrogate-based stabilization assigns increasing energy to an isochoric missing mode as $\nu\to 1/2$ when energies are normalized by the physical shear scale, indicating bulk-driven stiffening of unresolved shear-type patterns. In contrast, the proposed decoupled stabilization remained essentially constant under the same normalization, consistent with shear-only scaling on isochoric kernel modes. A kernel-restricted modal analysis further confirmed the intended properties: exact nullity on polynomial modes, directional redistribution of kernel stiffness consistent with the prescribed anisotropy weights, and a clear deviatoric/volumetric separation under variations of $\kappa_E$. Finally, in Cook's membrane under nearly incompressible conditions, the decoupled stabilization produced a coherent refinement picture across diverse mesh families, whereas the classical stabilization exhibited a locking-like footprint with pronounced underprediction on coarse meshes and slow recovery under refinement.

Overall, the results indicate that stabilization design in finite-strain VEM benefits from treating stabilization as an energy-assignment mechanism on $\missingkernel$, guided by spectral equivalence to the target tangent energy and by explicit deviatoric/volumetric separation. The proposed kernel-only decoupled stabilization satisfies these principles while avoiding auxiliary sub-triangulations and bulk-dependent shear inflation, and the numerical evidence suggests improved robustness on distorted polygonal meshes in the nearly incompressible regime.

Several extensions are natural. The present analysis focused on the lowest-order setting and on polygonal elements under plane strain; extensions to higher-order VEM spaces, three-dimensional polyhedra, and more general hyperelastic models follow the same design philosophy but require additional projection operators and corresponding kernel characterizations. A further direction is to refine the mode-aware weighting using inexpensive geometric proxies tailored to extreme anisotropies, while preserving boundedness and uniform-in-$\nu$ stability constants. Finally, integrating the kernel spectral diagnostics into automated parameter selection may provide a practical route to stabilization tuning that is robust across element shapes and loading paths.

\newpage

\bibliographystyle{unsrt}  


\newpage

\appendix
\section{Functional analysis overview}
\label{ap:functional_analysis}

\paragraph{}This appendix summarizes the functional-analytic setting and notation used throughout the paper, including the geometric regularity assumptions required for trace operators and boundary fractional norms, and the Sobolev spaces employed in the stability and scaling estimates. Unless stated otherwise, all function spaces are defined over $\mathbb{R}$. Vector-valued spaces are understood componentwise; for instance, $[H^1(\Omega)]^2$ denotes the Cartesian product of two copies of $H^1(\Omega)$ equipped with the natural product norm and seminorm. Standard results and conventions are adopted, and the presentation follows the references \cite{evans10,brezis11}.

\subsection{Lipschitz domains and boundary regularity}

\paragraph{}Several functional-analytic tools used in this work (trace operators, Poincar\'e--type inequalities, and fractional Sobolev spaces on boundaries) require mild regularity assumptions on the geometry of the underlying domain. Throughout, $\Omega\subset\mathbb{R}^d$ denotes a bounded open set. The boundary regularity needed is expressed by the notion of a Lipschitz domain.

\paragraph{}A function $\phi:\mathbb{R}^{d-1}\to\mathbb{R}$ is called Lipschitz on an open set $U\subset\mathbb{R}^{d-1}$ if there exists a constant $L\ge 0$ such that
\begin{equation}
|\phi(\mathbf{x}')-\phi(\mathbf{y}')|
\le
L\,|\mathbf{x}'-\mathbf{y}'|
\qquad
\forall\,\mathbf{x}',\mathbf{y}'\in U,
\end{equation}
where $|\cdot|$ denotes the Euclidean norm in $\mathbb{R}^{d-1}$. Equivalently, $\phi$ is globally Lipschitz on $U$ if it has bounded slope in the metric sense; in particular, Rademacher's theorem implies that $\phi$ is differentiable almost everywhere and satisfies $|\nabla\phi|\le L$ almost everywhere on $U$.

\paragraph{}A bounded open set $\Omega\subset\mathbb{R}^d$ is called a Lipschitz domain if for every point $\mathbf{x}_0\in\partial\Omega$ there exist a radius $r>0$, a rigid motion of $\mathbb{R}^d$ (translation and rotation), and a Lipschitz function $\phi:U\to\mathbb{R}$ defined on an open set $U\subset\mathbb{R}^{d-1}$ such that, in the transformed coordinates $\mathbf{x}=(\mathbf{x}',x_d)\in\mathbb{R}^{d-1}\times\mathbb{R}$ with $\mathbf{x}_0$ mapped to the origin, the boundary portion $\partial\Omega\cap B(\mathbf{0},r)$ is represented as the graph
\begin{equation}
\partial\Omega\cap B(\mathbf{0},r)
=
\left\{
(\mathbf{x}',x_d)\in B(\mathbf{0},r)\,:\, x_d=\phi(\mathbf{x}')
\right\},
\end{equation}
and $\Omega$ lies on one side of this graph, namely
\begin{equation}
\Omega\cap B(\mathbf{0},r)
=
\left\{
(\mathbf{x}',x_d)\in B(\mathbf{0},r)\,:\, x_d>\phi(\mathbf{x}')
\right\}
\quad
\text{or}
\quad
\Omega\cap B(\mathbf{0},r)
=
\left\{
(\mathbf{x}',x_d)\in B(\mathbf{0},r)\,:\, x_d<\phi(\mathbf{x}')
\right\}.
\end{equation}
Thus, near each boundary point, $\partial\Omega$ can be written as the graph of a Lipschitz function, and $\Omega$ is locally the epigraph (or hypograph) of that function.

\paragraph{}In two dimensions, polygonal domains are Lipschitz. In particular, each polygonal element $E$ of a partition $\mathcal{T}_h$ is a bounded Lipschitz domain, and its boundary $\partial E$ is a one-dimensional Lipschitz manifold. This regularity suffices for the existence of a continuous trace operator $\gamma:H^1(\Omega)\to H^{1/2}(\partial\Omega)$ and for standard trace inequalities, which are invoked in the analysis of boundary-based stabilization terms.

\paragraph{}A family of polygonal meshes is termed shape-regular if it satisfies uniform geometric constraints such as (M1)--(M4) in Assumption~\ref{assumption:regularity}. These assumptions provide uniform control of constants in inequalities used throughout the paper (e.g., Poincar\'e, trace, and discrete norm equivalences), independently of the mesh size parameter $h$.

\subsection{Sobolev spaces}

\paragraph{}Let $\Omega\subset\mathbb{R}^d$ be a bounded Lipschitz domain with boundary $\partial\Omega$, and let $k\in\mathbb{N}_0$. The Sobolev space $H^k(\Omega)$ is defined as
\begin{equation}
H^k(\Omega)
=
\left\{
v\in L^2(\Omega)\,:\, D^\alpha v \in L^2(\Omega)\ \text{for all multi-indices }\alpha\ \text{with }|\alpha|\le k
\right\},
\end{equation}
where $D^\alpha$ denotes the weak derivative associated with the multi-index $\alpha$. The standard Sobolev norm and seminorm are given by
\begin{equation}
\|v\|_{H^k(\Omega)}^2
=
\sum_{|\alpha|\le k}\|D^\alpha v\|_{L^2(\Omega)}^2,
\qquad
|v|_{H^k(\Omega)}^2
=
\sum_{|\alpha|= k}\|D^\alpha v\|_{L^2(\Omega)}^2.
\end{equation}
In particular, for $k=1$,
\begin{equation}
\|v\|_{H^1(\Omega)}^2
=
\|v\|_{L^2(\Omega)}^2+\|\nabla v\|_{L^2(\Omega)}^2,
\qquad
|v|_{H^1(\Omega)}^2
=
\|\nabla v\|_{L^2(\Omega)}^2.
\end{equation}

\paragraph{}The space $H_0^1(\Omega)$ is defined as the closure of $C_0^\infty(\Omega)$ in $H^1(\Omega)$, equivalently as
\begin{equation}
H_0^1(\Omega)
=
\left\{
v\in H^1(\Omega)\,:\, \gamma(v)=0\ \text{on }\partial\Omega
\right\},
\end{equation}
where $\gamma:H^1(\Omega)\to H^{1/2}(\partial\Omega)$ denotes the trace operator, which is continuous for Lipschitz domains.

\paragraph{}In the analysis of elementwise stabilization terms, local Sobolev seminorms and boundary norms are employed. For a polygonal element $E\subset\mathbb{R}^2$ and a (scalar) function $v\in H^1(E)$, define the $L^2(E)$ norm and the $H^1(E)$ seminorm by
\begin{equation}
\|v\|_{0,E}^2
=
\int_E |v|^2\,dE,
\qquad
|v|_{1,E}^2
=
\int_E |\nabla v|^2\,dE.
\end{equation}
For a vector field $\mathbf{v}\in [H^1(E)]^2$, the corresponding quantities are defined componentwise,
\begin{equation}
\|\mathbf{v}\|_{0,E}^2
=
\int_E |\mathbf{v}|^2\,dE,
\qquad
|\mathbf{v}|_{1,E}^2
=
\int_E |\nabla \mathbf{v}|^2\,dE
=
\sum_{j=1}^2 \int_E |\nabla v_j|^2\,dE,
\end{equation}
where $|\cdot|$ denotes the Euclidean norm and $\nabla\mathbf{v}$ is the matrix of first derivatives.

\paragraph{}On the boundary $\partial E$, the $L^2(\partial E)$ norm is defined by
\begin{equation}
\|\mathbf{v}\|_{0,\partial E}^2
=
\int_{\partial E} |\mathbf{v}|^2\,ds.
\end{equation}
When required, the trace of $\mathbf{v}\in[H^1(E)]^2$ is understood in the sense of the continuous trace operator $\gamma:[H^1(E)]^2\to[H^{1/2}(\partial E)]^2$. In particular, trace inequalities of the form
\begin{equation}
\|\mathbf{v}\|_{0,\partial E}^2
\le
C\left(
h_E^{-1}\|\mathbf{v}\|_{0,E}^2
+
h_E\,|\mathbf{v}|_{1,E}^2
\right)
\end{equation}
are invoked, where $h_E=\mathrm{diam}(E)$ and $C$ depends only on the shape-regularity parameters of the admissible polygon family.

\subsection{Fractional Sobolev spaces on the boundary}

\paragraph{}Let $\Gamma\subset\mathbb{R}^d$ be a Lipschitz $(d-1)$-dimensional manifold (in particular, $\Gamma=\partial\Omega$ or $\Gamma=\partial E$ for a polygonal element $E$). The fractional Sobolev space $H^{1/2}(\Gamma)$ may be characterized by the Gagliardo seminorm. For $g\in L^2(\Gamma)$, define
\begin{equation}
|g|_{H^{1/2}(\Gamma)}^2
=
\int_{\Gamma}\int_{\Gamma}\frac{|g(\mathbf{p})-g(\mathbf{q})|^2}{|\mathbf{p}-\mathbf{q}|^{d}}\,ds_{\mathbf{p}}\,ds_{\mathbf{q}}.
\end{equation}
Then
\begin{equation}
H^{1/2}(\Gamma)
=
\left\{
g\in L^2(\Gamma)\,:\, |g|_{H^{1/2}(\Gamma)}<\infty
\right\},
\end{equation}
and a norm on $H^{1/2}(\Gamma)$ is given by
\begin{equation}
\|g\|_{H^{1/2}(\Gamma)}^2
=
\|g\|_{L^2(\Gamma)}^2+|g|_{H^{1/2}(\Gamma)}^2.
\end{equation}

\paragraph{}In the two-dimensional setting $d=2$ used in this work, $\Gamma$ is one-dimensional and the exponent in the Gagliardo seminorm reduces to $d=2$, yielding
\begin{equation}
|g|_{H^{1/2}(\Gamma)}^2
=
\int_{\Gamma}\int_{\Gamma}\frac{|g(\mathbf{p})-g(\mathbf{q})|^2}{|\mathbf{p}-\mathbf{q}|^{2}}\,ds_{\mathbf{p}}\,ds_{\mathbf{q}}.
\end{equation}
When $\Gamma=\partial E$ for a polygonal element $E$, this coincides with the intrinsic boundary seminorm employed in Lemma~\ref{lemma:difference_inequality}.

\paragraph{}The trace operator $\gamma$ maps $H^1(\Omega)$ continuously into $H^{1/2}(\partial\Omega)$, and there exists a constant $C>0$ such that
\begin{equation}
\|\gamma(v)\|_{H^{1/2}(\partial\Omega)}
\le
C\,\|v\|_{H^1(\Omega)}
\qquad \forall v\in H^1(\Omega).
\end{equation}
Moreover, the trace inequality provides the estimate
\begin{equation}
\|v\|_{L^2(\partial\Omega)}^2
\le
C\left(
h^{-1}\|v\|_{L^2(\Omega)}^2+h\,|v|_{H^1(\Omega)}^2
\right),
\end{equation}
where $h$ is a characteristic length scale (e.g., $h=h_E$ on an element $E$), and $C$ depends only on the shape-regularity of the domain.

\section{Proofs}
\label{ap:proofs}

\subsection{Proof of Proposition \ref{prop:spectral_equivalence}}
\label{ap:spectral_equivalence}
Writing any $\mathbf{w} \in \missingkernel$ as 
\begin{equation}
    \label{eq:aux_vector_basis_decomposition}
    \mathbf{w} = \sum^{m}_{i=1} x_i \basis[i],
\end{equation}
i.e. by its coefficient vector $\mathbf{x} \in \mathbb{R}^m$.
Then,
\begin{equation}
    a_E (\mathbf{w}, \mathbf{w}) = \mathbf{x}^T \mathbf{K}_E \mathbf{x}, \quad \ahE[s] (\mathbf{w}, \mathbf{w}) = \mathbf{x}^T \mathbf{S}_E \mathbf{x}.
\end{equation}
The inequality (\ref{eq:vem_stabilization_inequality}) becomes: 
\begin{equation}
    \label{eq:aux_matrix_form_stab_ineq}
    C_0 \mathbf{x}^T \mathbf{K}_E \mathbf{x} \leq  \mathbf{x}^T \mathbf{S}_E \mathbf{x} \leq C_1 \mathbf{x}^T \mathbf{K}_E \mathbf{x},
\end{equation}
for all $\mathbf{x} \neq \mathbf{0}$. Since $\mathbf{K}_E$ is SPD, it is possible to divide (\ref{eq:aux_matrix_form_stab_ineq}) by $\mathbf{x}^T \mathbf{K}_E \mathbf{x}$:
\begin{equation}
    \label{eq:aux_rayleight_stab_ineq}
    C_0 \leq \frac{\mathbf{x}^T \mathbf{S}_E \mathbf{x}}{\mathbf{x}^T \mathbf{K}_E \mathbf{x}} \leq C_1
\end{equation}
The fraction 
\begin{equation}
    \rho (\mathbf{x}) = \frac{\mathbf{x}^T \mathbf{S}_E \mathbf{x}}{\mathbf{x}^T \mathbf{K}_E \mathbf{x}}
\end{equation}
is the Rayleigh quotient.

\paragraph{} Because $\mathbf{K}_E$ is SPD, it admits an SPD square root $\mathbf{K}_E^{1/2}$ and an inverse $\mathbf{K}_E^{-1/2}$. Define
\begin{equation}
    \mathbf{B} = \mathbf{K}_E^{-1/2} \mathbf{S}_E \mathbf{K}_E^{-1/2}.
\end{equation}
The matrix $\mathbf{B}$ is symmetric, and for any $\mathbf{x} \neq \mathbf{0}$, letting $\mathbf{y} = \mathbf{K}_E^{1/2}\mathbf{x} \neq \mathbf{0}$,
\begin{equation}
    \rho (\mathbf{x}) = \frac{\left( \mathbf{K}_E^{-1/2}\mathbf{y} \right)^T \mathbf{S}_E \left( \mathbf{K}_E^{-1/2}\mathbf{y} \right)}{\mathbf{y}^T \mathbf{y}} = \frac{\mathbf{y}^T \mathbf{B} \mathbf{y}}{\mathbf{y}^T \mathbf{y}}.
\end{equation}
Hence, bounding $\rho (\mathbf{x})$ for all $\mathbf{x} \neq \mathbf{0}$ is equivalent to bounding the ordinary Rayleigh quotient of $\mathbf{B}$ for all $\mathbf{y} \neq \mathbf{0}$.

\paragraph{} For symmetric matrices, it holds that:
\begin{equation}
    \lambda_{\min} (\mathbf{B}) = \min \limits_{\mathbf{y} \neq \mathbf{0}}  \frac{\mathbf{y}^T \mathbf{B} \mathbf{y}}{\mathbf{y}^T \mathbf{y}}, \quad \lambda_{\max} (\mathbf{B}) = \max \limits_{\mathbf{y} \neq \mathbf{0}}  \frac{\mathbf{y}^T \mathbf{B} \mathbf{y}}{\mathbf{y}^T \mathbf{y}}.
\end{equation}
Therefore, if
\begin{equation}
    \label{eq:aux_rayleigh_bound}
    C_0 \leq  \frac{\mathbf{y}^T \mathbf{B} \mathbf{y}}{\mathbf{y}^T \mathbf{y}} \leq C_1,
\end{equation}
for all $\mathbf{y} \neq \mathbf{0}$. Then,
\begin{equation}
    C_0 \leq \lambda_{\min} (\mathbf{B}) \leq \lambda_{\max} (\mathbf{B}) \leq C_1.
\end{equation}
Hence, all eigenvalues of $\mathbf{B}$ lie in $[C_0, C_1]$. Conversely, if all eigenvalues of $\mathbf{B}$ lie in $[C_0, C_1]$, then for every $\mathbf{y} \neq \mathbf{0}$, (\ref{eq:aux_rayleigh_bound}) holds true, and 
\begin{equation}
    C_0 \leq \rho (\mathbf{x}) \leq C_1
\end{equation}
for all $\mathbf{x} \neq \mathbf{0}$, which is exactly (\ref{eq:aux_matrix_form_stab_ineq}).

\paragraph{} If 
\begin{equation}
    \mathbf{B} \mathbf{y} = \lambda \mathbf{y},
\end{equation}
set
\begin{equation}
    \mathbf{x} = \mathbf{K}_E^{-1/2} \mathbf{y}.
\end{equation}
Then,
\begin{equation}
    \mathbf{K}_E^{-1/2} \mathbf{S}_E \mathbf{K}_E^{-1/2}\mathbf{y} = \lambda \mathbf{y} \Rightarrow \mathbf{S}_E \mathbf{x} = \lambda \mathbf{K}_E \mathbf{x}.
\end{equation}
This shows that the eigenvalues of $\mathbf{B}$ coincide with the generalized eigenvalues of the pair $(\mathbf{S}_E, \mathbf{K}_E)$.

\subsection{Proof of Proposition \ref{prop:kernel_tangent_from_stabilization}}
\label{ap:kernel_tangent_from_stabilization}
To simplify notation in this proof, define:
    \begin{equation}
        g = \projector, \quad f = \overline{U}^c_E.
    \end{equation}
    Moreover, set $X:=V_{h,E}$ and $Y:=[\mathbb{P}_1(E)]^2$, and denote by $\|\cdot\|_X$ and $\|\cdot\|_Y$ norms on $X$ and $Y$, respectively. Let $Y^*:=(Y)^*$ be endowed with the induced dual norm $\|\cdot\|_{Y^*}$. Then,
    \begin{equation}
        g: X \longrightarrow Y \ \ \text{is bounded linear}, \qquad f: Y \longrightarrow \mathbb{R}.
    \end{equation}
    Claim that if $f$ is Fréchet differentiable at $g(\mathbf{u})$ and $g$ is bounded linear, then $f \circ g$ is Fréchet differentiable at $\mathbf{u}$, with
    \begin{equation}
        D(f \circ g) (\mathbf{u}) = Df(g(\mathbf{u})) \circ g.
    \end{equation}
    Equivalently,
    \begin{equation}
        DU_E^c (\mathbf{u}) [\mathbf{h}] = D \overline{U}^c_E(\projector \mathbf{u})[\projector \mathbf{h}], \; \forall \mathbf{h} \in V_{h,E}.
    \end{equation}

    \paragraph{}Since $f$ is Fréchet differentiable at 
    \begin{equation}
        \mathbf{z} = g(\ustar) = \projector \ustar,
    \end{equation}
    there exists a bounded linear functional 
    \begin{equation}
        \label{eq:bounded_linear_functional}
        A = Df(\mathbf{z}): Y=[\mathbb{P}_1(E)]^2 \longrightarrow \mathbb{R},
    \end{equation}
    i.e. $A\in Y^*$,
    such that
    \begin{equation}
        \label{eq:aux_frechet_derivative}
        f(\mathbf{z} + \mathbf{t}) = f(\mathbf{z}) + A(\mathbf{t}) + r_f(\mathbf{t}), \; \text{with} \; \frac{|r_f(\mathbf{t})|}{\| \mathbf{t}\|_Y} \longrightarrow 0 \; \text{as} \; \|\mathbf{t}\|_Y \longrightarrow 0.
    \end{equation}
    Take $\mathbf{t} = g(\mathbf{v})$. Because $g$ is linear:
    \begin{equation}
        (f \circ g)(\mathbf{u} + \mathbf{h}) = f(g(\mathbf{u}) + g(\mathbf{h})) = f(\mathbf{z} + \mathbf{t}).
    \end{equation}
    Applying equation (\ref{eq:aux_frechet_derivative}):
    \begin{equation}
        \begin{split}
            f(\mathbf{z} + \mathbf{t}) &= f(\mathbf{z}) + A(\mathbf{t}) + r_f(\mathbf{t}) \\
            &= f(g(\mathbf{u})) + A(\mathbf{t}) + r_f(\mathbf{t}) \\
            &= (f\circ g)(\mathbf{u}) + A(\mathbf{t}) + r_f(\mathbf{t}).
        \end{split}
    \end{equation}
    Therefore,
    \begin{equation}
        \label{eq:aux_composition_frechet_derivative}
        (f \circ g)(\mathbf{u} + \mathbf{h}) - (f \circ g)(\mathbf{u}) - A(\mathbf{t}) = r_f(\mathbf{t}).
    \end{equation}
    Dividing (\ref{eq:aux_composition_frechet_derivative}) by $\| \mathbf{h} \|_X$:
    \begin{equation}
        \label{eq:aux_frechet_prod_zero}
        \frac{|(f \circ g)(\mathbf{u} + \mathbf{h}) - (f \circ g)(\mathbf{u}) - A(\mathbf{t})|}{\|\mathbf{h}\|_X}
        =
        \frac{|r_f(\mathbf{t})|}{\|\mathbf{t}\|_Y}\frac{\|\mathbf{t}\|_Y}{\|\mathbf{h}\|_X}.
    \end{equation}
    Since $g$ is bounded linear, $\|\mathbf{t}\|_Y=\|g(\mathbf{h})\|_Y \le \|g\|_{\mathcal{L}(X,Y)}\|\mathbf{h}\|_X$, and therefore $\|\mathbf{t}\|_Y \longrightarrow 0$ as $\|\mathbf{h}\|_X \longrightarrow 0$. Hence, by (\ref{eq:aux_frechet_derivative}),
    \begin{equation}
        \frac{|r_f(\mathbf{t})|}{\|\mathbf{t}\|_Y} \longrightarrow 0. 
    \end{equation}
    Moreover, since $g$ is bounded,
    \begin{equation}
        \| \projector \mathbf{h}\|_Y \leq \| \projector \|_{\mathcal{L}(X,Y)} \, \| \mathbf{h} \|_X.
    \end{equation}
    Thus, the product in (\ref{eq:aux_frechet_prod_zero}) tends to zero. Hence, $f \circ g$ is Fréchet differentiable at $\ustar$, and by chain rule:
    \begin{equation}
        D(f \circ g)(\mathbf{u}) = Df(g(\mathbf{u})) \circ g.
    \end{equation}
    Then, by writing in the directional form:
    \begin{equation}
        \label{eq:aux_directional_form}
        DU_E^c(\mathbf{u})[\mathbf{h}] = D U_E^c(\projector \mathbf{u}) [\projector \mathbf{h}].
    \end{equation}

    \paragraph{}Now, the goal is to differentiate the map $\mathbf{u}\mapsto$. From (\ref{eq:aux_directional_form}) and for any direction $\increment{u}$:
    \begin{equation}
        \label{eq:derivative_as_map_notation}
        \mathbf{u} \mapsto D U_E^c(\projector \mathbf{u}) [\projector \increment{u}]
    \end{equation}
    Consider 
    \begin{equation}
        G(\mathbf{u}) = Df(\projector \mathbf{u})
    \end{equation}
    which is a map into the dual space $([\mathbb{P}_1(E)]^2)^*$. Then, (\ref{eq:derivative_as_map_notation}) can be written as:
    \begin{equation}
        \mathbf{u} \mapsto \left\langle G(\mathbf{u}), \projector \increment{u} \right\rangle,
    \end{equation}
    i.e. evaluation of a functional $G(\mathbf{u})$ on a vector $\increment{u}$. 

    \paragraph{}From the definition given in (\ref{eq:bounded_linear_functional}), it is known that $A \in ([\mathbb{P}_1(E)]^2)^*$. Define the following scalar function:
    \begin{equation}
        \label{eq:aux_scalar_function}
        \Phi (\mathbf{u}) = Df(\projector \mathbf{u})[\projector \increment{u}].
    \end{equation}
    Differentiate $\Phi(\mathbf{u})$ in direction $\increment{v}$. Since, by definition, $\Phi$ is Fréchet differentiable at $\mathbf{u}$ its directional derivative exists, and it is given by:
    \begin{equation}
        \lim \limits_{\alpha \rightarrow 0} \frac{\Phi(\mathbf{u}+\alpha \increment{v})}{\alpha} = D\Phi (\mathbf{u})[\increment{v}].
    \end{equation}
    Compute
    \begin{equation}
        \begin{split}
            \Phi(\mathbf{u} + \alpha \increment{v}) &= Df(\projector (\mathbf{u} + \alpha \increment{v}))[\projector \increment{u}]\\
            &= Df(\projector \mathbf{u} + \alpha \projector \increment{v})[\projector \increment{u}].
        \end{split}
    \end{equation}
    Thus,
    \begin{equation}
        \frac{\Phi(\mathbf{u}+\alpha \increment{v})}{\alpha} = \frac{Df(\projector \mathbf{u} + \alpha \projector \increment{v})[\projector \increment{u}]-Df(\mathbf{z})[\projector \increment{u}]}{\alpha},
    \end{equation}
    with $\mathbf{z} = \projector \mathbf{u}$.
    
    \paragraph{}Use that $Df: [\mathbb{P}_1(E)]^2\longrightarrow ([\mathbb{P}_1(E)]^2)^*$ is Fréchet differentiable at $\mathbf{z}$, with derivative $D^2f(\mathbf{z}): [\mathbb{P}_1(E)]^2\longrightarrow ([\mathbb{P}_1(E)]^2)^*$ (i.e. the Hessian). Explicitly, that means:
    \begin{equation}
        \label{eq:aux_hessian_frechet}
        Df(\mathbf{z} + \mathbf{y}) = Df(\mathbf{z}) + D^2f(\mathbf{z})[\mathbf{y}] + r_{Df}(\mathbf{y}), \quad \frac{\| r_{Df}(\mathbf{y})\|_{Y^*}}{\| \mathbf{y} \|_Y} \longrightarrow 0.
    \end{equation}
    Using $\mathbf{y} = \alpha \projector \increment{v}$ in (\ref{eq:aux_hessian_frechet}):
    \begin{equation}
        \label{eq:aux_hessian_frechet_2}
        Df(\mathbf{z} + \alpha \projector \increment{v}) = Df (\mathbf{z}) + D^2f(\mathbf{z})[\alpha\projector \increment{v}] + r_{Df}(\alpha \projector \increment{v}).
    \end{equation}
    Evaluate both sides of (\ref{eq:aux_hessian_frechet_2}) on $\projector \increment{u}$:
    \begin{equation}
        Df(\mathbf{z} + \alpha \projector \increment{v})[\projector \increment{u}] = Df(\mathbf{z})[\projector \increment{u}] + D^2f(\mathbf{z})[\alpha \projector \increment{v}][\projector \increment{u}] + r_{Df}(\alpha \projector \increment{v})[\projector \increment{u}].
    \end{equation}
    Subtract $Df(\mathbf{z})[\projector \increment{u}]$ and divide by $\alpha$:
    \begin{equation}
        \frac{Df(\mathbf{z} + \alpha \projector \increment{v})[\projector \increment{u}]-Df(\mathbf{z})[\projector \increment{u}]}{\alpha} =  D^2f(\mathbf{z})[\projector \increment{v}][\projector \increment{u}] + \frac{ r_{Df}(\alpha \projector \increment{v})[\projector \increment{u}]}{\alpha}.
    \end{equation}
    The remainder term vanishes as $\alpha \rightarrow 0$ because:
    \begin{equation}
        \label{eq:aux_remainder_estimate}
        \begin{split}
            \left| \frac{r_{Df}(\alpha \projector \increment{v})[\projector \increment{u}]}{\alpha}\right|
        &\leq
        \frac{\|r_{Df}(\alpha \projector \increment{v})\|_{Y^*}}{|\alpha|}\,\|\projector \increment{u}\|_{Y} \\
        &=  \frac{\|r_{Df}(\alpha \projector \increment{v})\|_{Y^*}}{\| \alpha \projector \increment{v}\|_Y}  \frac{\| \projector \increment{v}\|_Y}{\| \projector \increment{v}\|_Y} \longrightarrow 0.
        \end{split}
    \end{equation}
    Note the in the last equality of (\ref{eq:aux_remainder_estimate}), it was considered that:
    \begin{equation}
        \| \alpha\projector \increment{v} \|_Y = |\alpha| \| \projector \increment{v}\|_Y.
    \end{equation}
    Hence,
    \begin{equation}
        D\Phi(\mathbf{u})[\increment{v}] = D^2f(\projector \mathbf{u})[\projector \increment{v}][\projector \increment{u}]
    \end{equation}
    Since $f = \overline{U}_E^c$,
    \begin{equation}
        \label{eq:aux_second_frechet_derivative}
        D^2 U^c_E (\mathbf{u}) [\increment{u}, \increment{v}] = D^2 \overline{U}_E^c (\projector \mathbf{u})[\projector \increment{u}] [ \projector \increment{v}].
    \end{equation}
    Take $\increment{u}, \increment{v} \in \missingkernel$. Then, 
    \begin{equation}
        \projector \increment{u} = \projector \increment{v} = \mathbf{0}.
    \end{equation}
    From (\ref{eq:aux_second_frechet_derivative}), it is possible to conclude that
    \begin{equation}
        D^2 U_E^c (\mathbf{u}) [\increment{u}] [ \increment{v}] =  D^2 \overline{U}_E^c (\projector \mathbf{u})[\mathbf{0}] [ \mathbf{0}] = \mathbf{0}.
    \end{equation}

\subsection{Proof of Proposition \ref{prop:quadratic_form_spectral_equivalence}}
\label{ap:quadratic_form_spectral_equivalence}
Choose any basis $\{ \mathbf{e}_i \}^m_{i=1}$ of $\missingkernel$. Any $\increment{u} \in \missingkernel$ is written as:
    \begin{equation}
        \label{eq:aux_vector_in_basis}
        \increment{u} = \sum \limits^m_{i=1} x_i \mathbf{e}_i, \; \mathbf{x}\in \mathbb{R}^m.
    \end{equation}
    Because $q^s$ and $q^*$ are quadratic forms coming from second variations, they are symmetric bilinear forms:
    \begin{equation}
        \label{eq:aux_bilinear_form_of_q}
        q^s(\increment{u}) = \mathcal{B}^s(\increment{u}, \increment{v}), q^*(\increment{u}), \; q^*(\increment{u}) = \mathcal{B}^*(\increment{u}, \increment{v}), \; \forall \increment{v}\in \missingkernel,
    \end{equation}
    with $\mathcal{B}^s$ and $\mathcal{B}^*$symmetric bilinear forms. Define the associated matrices $\mathbf{K}^s$ and $\mathbf{K}^*$ as:
    \begin{equation}
        \label{eq:aux_stiff_matrices_of_q}
        (\mathbf{K}^s)_{ij} = \mathcal{B}^s(\mathbf{e}_i, \mathbf{e}_j), \; (\mathbf{K}^*)_{ij} = \mathcal{B}^*(\mathbf{e}_i, \mathbf{e}_j).
    \end{equation}
    From (\ref{eq:aux_vector_in_basis}), (\ref{eq:aux_bilinear_form_of_q}) and (\ref{eq:aux_stiff_matrices_of_q}):
    \begin{equation}
        q^s(\increment{u}) = \mathbf{x}^T \mathbf{K}^s \mathbf{x}, \; q^*(\increment{u}) = \mathbf{x}^T \mathbf{K}^* \mathbf{x}.
    \end{equation}
    These are the representation of the quadratic form by matrices. Also, it is possible to observe that $\mathbf{K}^*$ is symmetric positive definite because $q^*(\increment{u}) > 0$ and $\increment{u} \neq \mathbf{0}$.

    \paragraph{}Rewrite (\ref{eq:stability_criteria_eigenproblem}) in coordinates $\mathbf{x} \neq \mathbf{0}$:
    \begin{equation}
        C_* \mathbf{x}^T \mathbf{K}^* \mathbf{x} \leq \mathbf{x}^T \mathbf{K}^s \mathbf{x} \leq C^* \mathbf{x}^T \mathbf{K}^* \mathbf{x}.
    \end{equation}
    Note that this inequality is based-independent, but it is useful to keep in this algebraic form. The Rayleigh quotient is given by:
    \begin{equation}
        \rho (\mathbf{x}) = \frac{\mathbf{x}^T \mathbf{K}^s \mathbf{x}}{\mathbf{x}^T \mathbf{K}^* \mathbf{x}}
    \end{equation}
    for all $\mathbf{x} \neq \mathbf{0}$. Because $\mathbf{K}^*$ is SPD, this quotient is well-defined since the denominator is always positive. From (\ref{eq:stability_criteria_eigenproblem}):
    \begin{equation}
        \label{eq:reduced_stability_criteria_eigenproblem}
        C_* \leq \rho (\mathbf{x}) \leq C^*.
    \end{equation}
    Recall that $\lambda$ is the generalized eigenvalue of $(\mathbf{K}^s, \mathbf{K}^*)$ if there exists $\mathbf{x} \neq \mathbf{0}$ such that:
    \begin{equation}
        \mathbf{x}^T \mathbf{K}^s \mathbf{x} = \lambda \mathbf{x}^T \mathbf{K}^* \mathbf{x}.
    \end{equation}
    This implies that
    \begin{equation}
        \label{eq:aux_rayleigh_generalized_eigenvalue}
        \frac{\mathbf{x}^T \mathbf{K}^s \mathbf{x}}{\mathbf{x}^T \mathbf{K}^* \mathbf{x}} = \lambda = \rho (\mathbf{x}).
    \end{equation}
    Thus, every generalized eigenvalue is a Rayleigh quotient evaluated at its eigenvector. Finally, from (\ref{eq:reduced_stability_criteria_eigenproblem}) and (\ref{eq:aux_rayleigh_generalized_eigenvalue}):
    \begin{equation}
        C_* \leq \lambda \leq C^*.
    \end{equation}

\subsection{Proof of Lemma \ref{lemma:difference_inequality}}
\label{ap:lemma_proof}

\paragraph{}Trivially, $|e_i| \leq h_E$ for every edge. By the regularity assumptions (M1) to (M4) (see Assumption \ref{assumption:regularity}), it is known that:
\begin{equation}
    c_e h_E \leq |e_i| \leq h_E.
\end{equation}
Hence,
\begin{equation}
    \frac{1}{h_E} \leq \frac{1}{|e_i|} \leq \frac{1}{c_e h_E}\Rightarrow \sum \limits^{N_E}_{i=1} |g_{i+1} - g_i|^2   \simeq h_E \sum \limits^{N_E}_{i=1} \frac{1}{|e_i|} |g_{i+1} - g_i|^2,
\end{equation}
with constants depending only on $c_e$. So, it is enough to proves
\begin{equation}
    \label{eq:aux_equiv_difference_inequality}
    |g|_\fracsobolev \simeq  \sum \limits^{N_E}_{i=1} |g_{i+1} - g_i|^2,
\end{equation}
and then convert to (\ref{eq:difference_inequality}) and (\ref{eq:difference_inequality}) using (\ref{eq:aux_equiv_difference_inequality}).

\paragraph{} Decompose the double integral into edge contributions:
\begin{equation}
    |g|^2_{\fracsobolev}
    =
    \sum_{i=1}^{N_E}\sum_{j=1}^{N_E}
    \int_{e_i}\int_{e_j}
    \frac{|g(\mathbf p)-g(\mathbf q)|^2}{|\mathbf p-\mathbf q|^2}\,ds_{\mathbf p}\,ds_{\mathbf q}.
\end{equation}
Since each integrand is nonnegative, retaining only the diagonal terms ($i=j$) yields the lower bound
\begin{equation}
    \label{eq:aux_lower_bound}
    |g|^2_{\fracsobolev}
    \ge
    \sum_{i=1}^{N_E}
    \int_{e_i}\int_{e_i}
    \frac{|g(\mathbf p)-g(\mathbf q)|^2}{|\mathbf p-\mathbf q|^2}\,ds_{\mathbf p}\,ds_{\mathbf q}.
\end{equation}
Fix an edge $e=[\mathbf a,\mathbf b]$ of length $L:=|e|$ and introduce the arc-length parametrization
\begin{equation}
    \Gamma:[0,L]\to e,
    \qquad
    \Gamma(s)=\mathbf a+\frac{s}{L}(\mathbf b-\mathbf a).
\end{equation}
Define the one-dimensional restriction $v:[0,L]\to\mathbb R^m$ by
\begin{equation}
    v(s):=g(\Gamma(s)).
\end{equation}
Since $g|_e\in \mathbb P_1(e)$, the function $v$ is affine in $s$, and therefore
\begin{equation}
    g(\Gamma(s))
    =
    g(\mathbf a)+\frac{s}{L}\big(g(\mathbf b)-g(\mathbf a)\big),
    \; s\in[0,L].
\end{equation}
So,
\begin{equation}
    g(\Gamma (s)) - g (\Gamma (t)) = \frac{s-t}{L}\big(g(\mathbf b)-g(\mathbf a)\big), \; s, t\in[0,L].
\end{equation}
Therefore,
\begin{equation}
    \label{eq:aux_transfomed_double_integral}
    \int_{e_i}\int_{e_i}
    \frac{|g(\mathbf p)-g(\mathbf q)|^2}{|\mathbf p-\mathbf q|^2}\,ds_{\mathbf p}\,ds_{\mathbf q} = \int \limits^L_0 \int \limits^L_0 \frac{\left|(s-t)\frac{\big(g(\mathbf b)-g(\mathbf a)\big)}{L} \right|^2}{|s-t|^2} dsdt = \left|\big(g(\mathbf b)-g(\mathbf a)\big)\right|^2.
\end{equation}
Applying (\ref{eq:aux_transfomed_double_integral}) edge-by-edge to (\ref{eq:aux_lower_bound}):
\begin{equation}
    |g|^2_{\fracsobolev} \geq \sum \limits^{N_E}_{i=1}|g_{i+1}-g_i|^2.
\end{equation}

\paragraph{} Let $\mathbf{p}, \mathbf{q} \in \partial E$ and let $d_{\partial E}(\mathbf{p}, \mathbf{q})$ be the shorter arc-length distance along $\partial E$ between $\mathbf{p}$ and $\mathbf{q}$. Because $g$ is continuous and piecewise affine, it is absolute continuous along $\partial E$, and by the Fundamental Theorem of Calculus:
\begin{equation}
    g(\mathbf{p}) - g (\mathbf{q}) = \int \limits_{\text{arc}(\mathbf{q} \rightarrow \mathbf{p})} \partial_t g (\bm{\xi}) d_{\bm{\xi}},
\end{equation}
where $\text{arc}(\mathbf{q} \rightarrow \mathbf{p})$ is the shorter arc from $\mathbf{q}$ to $\mathbf{p}$, and $\partial_t g$ is the tangential derivative (piecewise constant on edges). By the Cauchy-Schwarz Inequality:
\begin{equation}
    \label{eq:aux_cauchy_schwarz}
    \left|\big(g(\mathbf b)-g(\mathbf a)\big)\right|^2 \leq d_{\partial E}(\mathbf{p}, \mathbf{q}) \int \limits_{\text{arc}(\mathbf{q} \rightarrow \mathbf{p})} | \partial_t g (\bm{\xi})| ds_{\bm{\xi}}
\end{equation}
By definition of the derived chord-arc bound, there exists $C_{\text{ca}} > 0$ depending on $\rho_0$, $c_e$ and $N_{\max}$ such that:
\begin{equation}
    \label{eq:aux_chord_arc}
    d_{\partial E} (\mathbf{p}, \mathbf{q}) \leq C_{\text{ca}}|\mathbf{p} - \mathbf{q}|, \; \forall \mathbf{p}, \mathbf{q} \in \partial E.
\end{equation}
Combining (\ref{eq:aux_cauchy_schwarz}) and (\ref{eq:aux_chord_arc}):
\begin{equation}
    \left|\big(g(\mathbf b)-g(\mathbf a)\big)\right|^2 \leq \frac{C_{\text{ca}}}{|\mathbf{p} - \mathbf{q}|} \int \limits_{\text{arc}(\mathbf{q} \rightarrow \mathbf{p})} |\partial_t g (\bm{\xi})|^2 ds_{\bm{\xi}}.
\end{equation}
Integrate over $(\mathbf{p}, \mathbf{q}) \in \partial E \times \partial E$:
\begin{equation}
    \label{eq:aux_double_int_inequality_ca}
    |g|^2_{\fracsobolev} = \int \limits_{\partial E} \int \limits_{\partial E}  \frac{|g(\mathbf p)-g(\mathbf q)|^2}{|\mathbf p-\mathbf q|^2}\,ds_{\mathbf p}\,ds_{\mathbf q} \leq C_{\text{ca}} \int \limits_{\partial E} \int \limits_{\partial E} \frac{1}{|\mathbf{p} - \mathbf{q}|}\left( \int \limits_{\text{arc}(\mathbf{q} \rightarrow \mathbf{p})} |\partial_t g (\bm{\xi})|^2 d_{\bm{\xi}} \right) ds_{\mathbf p}\,ds_{\mathbf q}.
\end{equation}

\paragraph{}Define the indicator function:
\begin{equation}
    \chi_{\text{arc}(\mathbf{q} \rightarrow \mathbf{p})} (\bm{\xi})= \begin{cases}
        1 , \; \text{if} \; \bm{\xi} \in \text{arc}(\mathbf{q} \rightarrow \mathbf{p}), \\
        0, \; \text{otherwise}.
    \end{cases}
\end{equation}
Then,
\begin{equation}
    \label{eq:aux_arc_transformation}
    \int \limits_{\text{arc}(\mathbf{q} \rightarrow \mathbf{p})} |\partial_t g (\bm{\xi})|^2 ds_{\bm{\xi}} = \int \limits_{\partial E} |\partial_t g (\bm{\xi})|^2  \chi_{\text{arc}(\mathbf{q} \rightarrow \mathbf{p})} ds_{\bm{\xi}}.
\end{equation}
From (\ref{eq:aux_arc_transformation}) in (\ref{eq:aux_double_int_inequality_ca}), and by the Fubini Theorem:
\begin{equation}
    \label{eq:aux_target_inequality_ca}
    |g|^2_{\fracsobolev} \leq C_{\text{ca}} \int \limits_{\partial E} |\partial_t g (\bm{\xi})|^2 \left[ \int \limits_{\partial E} \int \limits_{\partial E} \frac{ \chi_{\text{arc}(\mathbf{q} \rightarrow \mathbf{p})}}{|\mathbf{p} - \mathbf{q}|} ds_{\mathbf{p}} ds_{\mathbf{q}} \right] ds_{\bm{\xi}}.
\end{equation}
So, now the whole problem is to show that the bracketed quantity is uniformly bounded:
\begin{equation}
    \label{eq:aux_bracketed_quantity}
    \mathcal{J}(\bm{\xi}) = \int \limits_{\partial E} \int \limits_{\partial E} \frac{ \chi_{\text{arc}(\mathbf{q} \rightarrow \mathbf{p})}}{|\mathbf{p} - \mathbf{q}|} ds_{\mathbf{p}} ds_{\mathbf{q}} \leq C_{\text{arc}},
\end{equation}
with $C_{\text{arc}}>0$ depending only on $\rho_0$, $c_e$ and $N_{\max}$. Once (\ref{eq:aux_bracketed_quantity}) is achieved, (\ref{eq:aux_target_inequality_ca}) immediately gives:
\begin{equation}
    |g|^2_{\fracsobolev} \leq C_{\text{geom}} \int \limits_{\partial E}|\partial_t g |^2 ds.
\end{equation}

\paragraph{} Fix $\bm{\xi} \in \partial E$. For each fixed $\mathbf{q} \in \partial E$, define the set of $\mathbf{p}$'s for which $\bm{\xi}$ lies on the chosen shorter arc from $\mathbf{q}$ to $\mathbf{p}$:
\begin{equation}
    S_{\bm{\xi}}(\mathbf{q}) = \{ \mathbf{p} \in \partial E: \; \bm{\xi} \in  \text{arc}(\mathbf{q} \rightarrow \mathbf{p})\}.
\end{equation}
Then, 
\begin{equation}
    \mathcal{J}(\bm{\xi}) = \int \limits_{\partial E} \left( \int \limits_{S_{\bm{\xi}}} \frac{1}{|\mathbf{p} - \mathbf{q}|} ds_{\mathbf{p}} \right) ds_{\mathbf{q}}.
\end{equation}
Thus, it is necessary to bound the inner integral uniformly in $\mathbf{q}$, and then integrate over it. Note that on a simple polygon boundary, for fixed $\bm{\xi}$ and $\mathbf{q}$, the set $S_{\bm{xi}}(\mathbf{q})$ is essentially half of the boundary seen from $\mathbf{q}$. It is a connected arc (possibly with an endpoint ambiguity where $\mathbf{p}$ is exactly the opposite). Naturally, $S_{\bm{\xi}}(\mathbf{q})$ is a single arc, not a complicated union, and under $N_E \leq N_{\max}$ that arc crosses at most $N_{\max}$ edges. This prevents counting $\bm{\xi}$ infinitely many times.

\paragraph{}For a fixed $\mathbf{q} \in \partial E$,
\begin{equation}
    \mathcal{I}(\mathbf{q}) = \int \limits_{\partial E} \frac{1}{|\mathbf{p} - \mathbf{q}|} ds_{\mathbf{p}} = \sum \limits^{N_E}_{i=1} \int \limits_{e_i}\frac{1}{|\mathbf{p} - \mathbf{q}|}ds_{\mathbf{p}}.
\end{equation}
Let $\eta$ be a line segment of length $L_\eta$ in $\mathbb{R}^2$. Let
\begin{equation}
    \delta = \text{dist} (\mathbf{q}, \eta)
\end{equation}
be the distance between $\mathbf{q}$ and $\eta$. Parametrize the segment by the arc-length $s \in [0,L_\eta]$ with a unit speed map $\mathbf{p}(s)$. Then,
\begin{equation}
    \int \limits_\eta \frac{1}{|\mathbf{p} - \mathbf{q}|} ds = \int \limits^{L_\eta}_0 \frac{1}{|\mathbf{p}(s) - \mathbf{q}|} ds.
\end{equation}
Let $\mathbf{q}_\eta$ be the orthogonal projection of $\mathbf{q}$ onto the line containing $\eta$. Along the line:
\begin{equation}
    |\mathbf{p}(s) - \mathbf{q}|^2 = |\mathbf{p}(s) - \mathbf{q}_\eta|^2 + |\mathbf{q} - \mathbf{q}_\eta|^2 \geq (s-s_0)^2 + \delta^2,
\end{equation}
for some shift $s_0$ (the coordinate of the projection). So, one gets the upper bound:
\begin{equation}
    \frac{1}{|\mathbf{p}(s) - \mathbf{q}|^2} \leq \frac{1}{\sqrt{(s-s_0)^2 + \delta^2}}.
\end{equation}
Hence,
\begin{equation}
    \label{eq:aux_integral_inequality}
    \int \limits_\eta \frac{1}{|\mathbf{p} - \mathbf{q}|^2} ds \leq \int \limits^{L_\eta}_{0}\frac{1}{\sqrt{(s-s_0)^2 + \delta^2}} ds.
\end{equation}
Recall the following improper inequality:
\begin{equation}
    \label{eq:aux_improper_integral}
    \int \frac{1}{\sqrt{(s-s_0)^2 + \delta^2}} ds = \text{arcsinh} \left( \frac{s - s_0}{\delta} \right) = \log \left( \frac{s - s_0}{\delta} + \sqrt{1 + \left( \frac{s - s_0}{\delta} \right)^2} \right).
\end{equation}
So, from (\ref{eq:aux_improper_integral}) in (\ref{eq:aux_integral_inequality}):
\begin{equation}
    \int \limits_\eta \frac{1}{|\mathbf{p} - \mathbf{q}|^2} ds \leq \log \left( \frac{L_\eta - s_0}{\delta} + \sqrt{1 + \left( \frac{L_\eta  - s_0}{\delta} \right)^2} \right) - \log \left( -\frac{s_0}{\delta} + \sqrt{1+\left( \frac{s_0}{\delta} \right)^2} \right).
\end{equation}
Each log term is bounded above up to $\log(1 + L_\eta/\delta)$, so overall:
\begin{equation}
    \int \limits_\eta \frac{1}{|\mathbf{p} - \mathbf{q}|^2} ds \leq C \left( 1+ \log \left(\frac{L_\eta}{\delta}\right) \right),
\end{equation}
for some constant $C>0$. If $\mathbf{q} \in \eta$, then the integrand behaves like $1/|s-s_\mathbf{q}|$ near the point $s_\mathbf{q}$, where $p(s_\mathbf{q}) = \mathbf{q}$. The integral $\int_\eta 1/|s-s_\mathbf{q}| ds$ diverges if taken in the usual sense, but its principal value is finite.

\paragraph{}Now, by the mesh regularity assumptions, it is known that $L_\eta$ is never extremely small, and there are only finitely many edges. This yield a bound 
\begin{equation}
    \mathcal{I}(\mathbf{q}) \leq \overline{C} (1 + \log \hat{C}),
\end{equation}
where $\overline{C}>0$ depends only on $c_e$ and $N_{\max}$, and since $\hat{C}$ is a fixed ration (edge length compared to $h_E$), the logarithm can be absorbed by the constant:
\begin{equation}
    \sup_{\mathbf{q} \in \partial E} \int \limits_{\partial E} \frac{1}{|\mathbf{p} - \mathbf{q}|^2} ds_{\mathbf{p}} \leq \overline{C}.
\end{equation}

\paragraph{}Since $S_{\bm{\xi}}(\mathbf{q}) \subset \partial E$,
\begin{equation}
    \label{eq:aux_s_xi_ineq}
    \int \limits_{S_{\bm{\xi}}(\mathbf{q})} \frac{1}{|\mathbf{p} - \mathbf{q}|^2} ds_{\mathbf{p}} \leq \int \limits_{\partial E} \frac{1}{|\mathbf{p} - \mathbf{q}|^2} ds_{\mathbf{p}} \leq \overline{C}, \; \forall \mathbf{q}.
\end{equation}
From (\ref{eq:aux_s_xi_ineq}) in (\ref{eq:aux_bracketed_quantity}):
\begin{equation}
    \mathcal{J}(\bm{\xi}) \leq \int \limits_{\partial E} \overline{C} ds = \overline{C} |\partial E|,
\end{equation}
where $|\partial E|$ can be understood as the perimeter of the polygon $E$. By the mesh regularity assumptions (M2) and (M3), it is known that $|\partial E|$ is comparable to $h_E$. There are at most $N_{\max}$ edges and each has a length less or equal to $h_E$, so:
\begin{equation}
    \label{eq:aux_geom_inequality_leq}
    |\partial E| \leq N_{\max} h_E.
\end{equation}
Also, each edge has length greater or equal to $c_e h_E$:
\begin{equation}
    \label{eq:aux_geom_inequality_geq}
    |\partial E| \geq N_E c_e h_E \geq c_e h_E.
\end{equation}
From (\ref{eq:aux_geom_inequality_leq}) and (\ref{eq:aux_geom_inequality_geq}):
\begin{equation}
    |\partial E| \simeq h_E,
\end{equation}
with constants depending only on $c_e$ and $N_{\max}$. Thus,
\begin{equation}
    \label{eq:aux_mathcal_J_ineq}
    \mathcal{J}(\bm{\xi}) \leq C_{\text{arc}}.
\end{equation}
From (\ref{eq:aux_mathcal_J_ineq}) in (\ref{eq:aux_bracketed_quantity}):
\begin{equation}
    |g|^2_{\fracsobolev, h} \leq C_{\text{arc}} \int \limits_{\partial E} |\delta_t g (\bm{\xi})|^2 ds_{\bm{\xi}}. 
\end{equation}
Thus, using the scaled semi-norm, it holds:
\begin{equation}
    |g|^2_{\fracsobolev, h} \leq C_{\text{arc}} \int \limits_{\partial E} |\partial_E g|^2 ds
\end{equation}

\paragraph{} Since $g$ is affine on each edge $e_i$, $\partial_t g$ is constant in $e_i$ with
\begin{equation}
    \partial_t g |_{e_i} = \frac{g_{i+i} - g_i}{|e_i|}.
\end{equation}
Using $|e_i| \geq c_e h_E$ and $|e_i| \leq h_E$, one gets:
\begin{equation}
    \sum \limits_{i=1}^{N_E} \frac{g_{i+i} - g_i}{|e_i|} \simeq \frac{1}{h_E}  \sum \limits_{i=1}^{N_E} |g_{i+1} - g_i|^2,
\end{equation}
with constants depending only on $c_e$. Therefore:
\begin{equation}
    |g|^2_{\fracsobolev} \leq C_{\text{ca}}  \sum \limits_{i=1}^{N_E} |g_{i+1} - g_i|^2
\end{equation}

\paragraph{}The constants $C_2$ and $C_3$ in \eqref{eq:difference_inequality} are the \emph{uniform equivalence constants} that quantify the comparison between the scaled Gagliardo seminorm on $\partial E$ and the discrete edge-difference seminorm. More precisely, $C_2>0$ is a lower-bound constant such that, for every admissible polygon $E$ and every $g\in C^0(\partial E)$ with $g|_{e_i}\in\mathbb P_1(e_i)$, one has
\begin{equation}
    \frac{1}{h_E}|g|^2_{\fracsobolev}\;\ge\; C_2\sum_{i=1}^{N_E}\frac{1}{|e_i|}|g_{i+1}-g_i|^2,
\end{equation}
and $C_3>0$ is an upper-bound constant such that
\begin{equation}
    \frac{1}{h_E}|g|^2_{\fracsobolev}\;\le\; C_3\sum_{i=1}^{N_E}\frac{1}{|e_i|}|g_{i+1}-g_i|^2.
\end{equation}
Both constants are mesh-uniform. Their values do not depend on $h_E$ or on the particular element $E$, but only on the geometric regularity parameters (e.g.\ the chunkiness/star-shapedness parameter $\rho_0$, the edge-length bound $c_e$ ensuring $c_eh_E\le |e_i|\le h_E$, and the uniform bound $N_{\max}$ on the number of edges). In particular, $C_2$ arises from retaining the nonnegative diagonal contributions in the double integral and applying the edge-length scaling, whereas $C_3$ additionally incorporates the uniform geometric bounds controlling off-diagonal edge interactions (e.g.\ chord--arc type estimates) together with the equivalence between $\int_{\partial E}|\partial_t g|^2\,ds$ and $\sum_i |g_{i+1}-g_i|^2/|e_i|$ for piecewise affine traces.

\subsection{Proof of Theorem \ref{theo:poincare_korn}}
\label{ap:theo_poincare_korn}

\paragraph{}Let $\gamma: [H^1(E)]^2 \longrightarrow [\fracsobolev]^2$ be the trace operator:
\begin{equation}
    \gamma \mathbf{w} = \mathbf{w}|_{\partial E}
\end{equation}
Because $\mathbf{w}$ is harmonic in $E$ (see the virtual element space definition in Section \ref{sec:vem_finite_elasticity}), it is the unique harmonic extension of its boundary trace $\gamma \mathbf{w}$. For scalar harmonic functions one has standard stability estimate:
\begin{equation}
    \label{eq:aux_geometric_equivalence}
    |\mathbf{w}|_{1,E} \simeq |\gamma \mathbf{w}|_{\fracsobolev}.
\end{equation}

\paragraph{} On each edge $e=(\mathbf{x}_i, \mathbf{x}_{i+1})$, $\gamma \mathbf{w}$ is linear. Hence, completely determined by the endpoints:
\begin{equation}
    \gamma \mathbf{w}|_{e} \in [\mathbb{P}_1(e)]^2, \; (\gamma \mathbf{w})(\mathbf{x}_{i}) = \mathbf{w}(\mathbf{x}_{i}), \; (\gamma \mathbf{w})(\mathbf{x}_{i+1}) = \mathbf{w}(\mathbf{x}_{i+1}).
\end{equation}
So the boundary trace is a continuous piecewise linear function on $\partial E$. This allows the comparison between the trace $H^{1/2}$-seminorm and a pure discrete quantity built from nodal values. 

\paragraph{} By Lemma \ref{lemma:difference_inequality}, a standard discrete characterization is:
\begin{equation}
    \label{eq:aux_trace_geometric_equivalence}
    |\gamma \mathbf{w}|^2_{\fracsobolev} \simeq \sum \limits_{e \in \partial E} \frac{1}{|e|} \left( \mathbf{w}(\mathbf{x}_{i+1}) - \mathbf{w}(\mathbf{x}_i) \right).
\end{equation}
This is the equivalence between boundary $H^{1/2}$-seminorm and nodal differences for piecewise traces. Combining (\ref{eq:aux_geometric_equivalence}) and (\ref{eq:aux_trace_geometric_equivalence}):
\begin{equation}
    \label{eq:aux_related_seminorm}
    |\mathbf{w}|^2_{1,E} \simeq \sum \limits_{e \in \partial E} \frac{1}{|e|} |\mathbf{w}(\mathbf{x}_{i+1}) - \mathbf{w}(x_i)|^2.
\end{equation}
Note that (\ref{eq:aux_related_seminorm}) relates the interior $H^1$-seminorm to a boundary discrete edge-difference seminorm.

\paragraph{}The target quantity in (\ref{eq:poincare_korn}) is $\sum_i |\mathbf{w}(\mathbf{x}_i)|^2$ and not the differences. To pass from the differences to values, the fact that $\mathbf{w} \in \missingkernel$ is used. Since $\projector$ projects onto affine polynomials, $\projector \mathbf{w} = \mathbf{0}$ implies that the affine content (including constants) is removed. Therefore, on $\missingkernel$ one obtains a Poincaré inequality:
\begin{equation}
    \label{eq:aux_poincare_like_ineq}
    \| \mathbf{w} \|_{0,E} \leq C h_E |\mathbf{w}|_{1,E}, \; C>0. 
\end{equation}
This is the continuous Poincaré inequality on the subspace where the zero-mode is eliminated. Using (\ref{eq:aux_related_seminorm}) and (\ref{eq:aux_poincare_like_ineq}) one can write the discrete version of the Poincaré inequality:
\begin{equation}
    \label{eq:aux_discrete_poincare}
    \sum \limits_{i=1}^{N_E} |\mathbf{w}(\mathbf{x}_i)|^2 \leq C h_E \sum \limits_{e \in \partial E}\frac{1}{|e|} |\mathbf{w}(\mathbf{x}_{i+1}) - \mathbf{w}(x_i)|^2.
\end{equation}
Under the standard shape regularity, the constant $C$ depends only on those regularity parameters presented in (M1)-(M4).

\paragraph{}Divide (\ref{eq:aux_discrete_poincare}) by $h_E^2$:
\begin{equation}
    \label{eq:aux_normalized_discrete_poincare}
    \frac{1}{h_E^2} \sum \limits^{N_E}_{i=1} |\mathbf{w}(\mathbf{x}_i)|^2 \leq \frac{C}{h_E} \sum \limits_{e \in \partial E}\frac{1}{|e|} |\mathbf{w}(\mathbf{x}_{i+1}) - \mathbf{w}(x_i)|^2.
\end{equation}\
Using the shape regularity assumptions:
\begin{equation}
    |e| \simeq h_E.
\end{equation}
Thus, the right-hand side of (\ref{eq:aux_normalized_discrete_poincare}) becomes:
\begin{equation}
    \label{eq:rhs_discrete_norm_porincare}
    \frac{1}{h_E^2} \sum \limits^{N_E}_{i=1} |\mathbf{w}(\mathbf{x}_i)|^2 \leq C^\prime \sum \limits_{e \in \partial E}\frac{1}{|e|} |\mathbf{w}(\mathbf{x}_{i+1}) - \mathbf{w}(x_i)|^2.
\end{equation}
Applying (\ref{eq:aux_related_seminorm}) to the right-hand side of (\ref{eq:rhs_discrete_norm_porincare}):
\begin{equation}
    \label{eq:aux_rhs_poincare_prime}
    \sum \limits_{e \in \partial E}\frac{1}{|e|} |\mathbf{w}(\mathbf{x}_{i+1}) - \mathbf{w}(x_i)|^2 \simeq |\mathbf{w}|_{1,E}.
\end{equation}
From (\ref{eq:aux_rhs_poincare_prime}) in (\ref{eq:rhs_discrete_norm_porincare}):
\begin{equation}
    \frac{1}{h_E^2} \sum \limits^{N_E}_{i=1} |\mathbf{w}(\mathbf{x}_i)|^2 \leq C^* |\mathbf{w}|_{1,E}.
\end{equation}

\paragraph{}Parametrize the length of edge $e$ by its arc-length using a map:
\begin{equation}
    \bm{\Gamma}: [0, |e|] \longrightarrow e, \; \bm{\Gamma}(0) = \mathbf{x}_i, \; \bm{\Gamma}(e) = \mathbf{x}_{i+1}, \; |\bm{\Gamma}^\prime (s)| =1,
\end{equation}
with $s \in [0, |e|]$. Define the restriction of the trace to the edge as one-dimensional function by:
\begin{equation}
    \mathbf{v}(s) = (\gamma \mathbf{w}) (\bm{\Gamma}(s)) \in \mathbb{R}^2.
\end{equation}
Then, the tangential derivative along the adge corresponds to differentiation with respect to arc-length:
\begin{equation}
    \partial_t (\gamma \mathbf{w})(\bm{\Gamma}(s)) = \frac{d}{ds}\mathbf{v}(s) = v^\prime (s).
\end{equation}
By the Fundamental Theorem of Calculus:
\begin{equation}
    \mathbf{w}(\mathbf{x}_{i+1}) - \mathbf{w}(\mathbf{x}_i) = \mathbf{v}(|e|) - \mathbf{v}(0) = \int \limits^{|e|}_{0} \mathbf{v}^\prime ds = \int \limits^{|e|}_{0}  \partial_t (\gamma \mathbf{w})(\bm{\Gamma}(s)) ds.
\end{equation}
Applying the Cauchy-Schwarz inequality:
\begin{equation}
    | \mathbf{w}(\mathbf{x}_{i+1}) - \mathbf{w}(\mathbf{x}_i)| \leq \int \limits^{|e|}_{0}  \left|\partial_t (\gamma \mathbf{w})(\bm{\Gamma}(s))\right| ds \leq \left( \int \limits^{|e|}_{0} 1^2 ds \right)^{1/2} \left( \int \limits^{|e|}_{0} \left|\partial_t (\gamma \mathbf{w})(\bm{\Gamma}(s))\right|^2 ds \right)^{1/2}.
\end{equation}
Squaring both sides:
\begin{equation}
    | \mathbf{w}(\mathbf{x}_{i+1}) - \mathbf{w}(\mathbf{x}_i)|^2 \leq |e| \int \limits^{|e|}_{0} \left|\partial_t (\gamma \mathbf{w})(\bm{\Gamma}(s))\right|^2 ds.
\end{equation}
Note that $ds$ is the arc-length measure of $e$, so the last integral is exactly:
\begin{equation}
    \label{eq:aux_integral_ineq_arc_length}
    \int \limits^{|e|}_{0} \left|\partial_t (\gamma \mathbf{w})(\bm{\Gamma}(s))\right|^2 ds = \int \limits_e |\partial_t (\gamma \mathbf{w})|^2 ds.
\end{equation}
Using (\ref{eq:aux_geometric_equivalence}) in (\ref{eq:aux_integral_ineq_arc_length}):
\begin{equation}
    \int \limits_e |\partial_t (\gamma \mathbf{w})|^2 ds \leq C^{\prime \prime} |\mathbf{w}|^2_{1,E}.
\end{equation}
Therefore,
\begin{equation}
    \frac{1}{h_E^2} \sum \limits^{N_E}_{i=1} |\mathbf{w}(\mathbf{x}_i)|^2 \geq C_* |\mathbf{w}|_{1,E}.
\end{equation}

\end{document}